\documentclass[reqno]{amsart}
\usepackage{amsmath}
\usepackage{amsfonts}
\usepackage{graphics}
\usepackage{epsfig}
\usepackage{amssymb}
\usepackage{amscd}
\usepackage[all]{xy}
\usepackage{latexsym}
\usepackage{graphicx}
\usepackage{multirow}
\usepackage{geometry}
\usepackage{color}
\usepackage[usenames,dvipsnames]{pstricks}
\usepackage{epsfig}
\usepackage{pst-grad} 
\usepackage{pst-plot}
\usepackage{textcomp} 

\newtheorem{thm}{Theorem}[section]
\newtheorem{cor}[thm]{Corollary}

\newtheorem{prop}[thm]{Proposition}

\newtheorem{quest}{Question}[section]

\theoremstyle{definition}
\newtheorem{Def}[thm]{Definition}
\newtheorem{rem}[thm]{Remark}

\newtheorem*{ack}{Acknowledgements}
\newtheorem{ex}[thm]{Example}

\numberwithin{equation}{section}
\numberwithin{figure}{section}

\begin{document}
\large
\setcounter{section}{0}

\allowdisplaybreaks


\title[Topological Recursion for Generalized $bc$-Motzkin Numbers]{Topological Recursion for Generalized $bc$-Motzkin Numbers}


\author[C.\ Jacob]{Cooper Jacob}
\address{Cooper Jacob:
Department of Mathematics\\
University of California\\
Davis, CA, USA, 95616}
\email{cooperjacob@math.ucdavis.edu}


\begin{abstract}

We present a higher genus generalization of $bc$-Motzkin numbers, which are themselves a generalization of Catalan numbers, and we derive a recursive formula which can be used to calculate them. Further, we show that this leads to a topological recursion which is identical to the topological recursion that had previously been proved for generalized Catalan numbers, and which is an example of the Eynard-Orantin topological recursion.

\end{abstract}





\maketitle


\tableofcontents


\section{Introduction}
\label{sect:introbasicdefs}

In this paper, we show that a generalization of the $bc$-Motzkin numbers, which were originally defined in \cite{sun}, satisfies a topological recursion. It is known (see \cite{dumitrescumulaselec}) that the Catalan numbers can be given a higher genus analogue, and this generalization satisfies a topological recursion. Thus, it is a natural question to ask whether the $bc$-Motzkin numbers, which are defined in terms of Catalan numbers, also satisfy a topological recursion.

Our main result in this paper is the following:

\begin{thm}
\label{thm:bcmtoprecintro}
Define symmetric $v$-linear differential forms on $(\mathbb{P}^1)^v$ for $2g-2+v > 0$, called Eynard-Orantin differential forms, by
$$W_{g,v}^{\widetilde{M}(b,c)}(t_1,t_2,\dots,t_v) = d_{t_1} \cdots d_{t_v} F_{g,v}^{\widetilde{M}(b,c)}(t_1,t_2,\dots,t_v),$$
and for $(g,v) = (0,2)$ by
$$W_{0,2}^{\widetilde{M}(b,c)}(t_1,t_2) = \frac{dt_1 \, dt_2}{(t_1-t_2)^2}.$$
Here, $F_{g,v}^{\widetilde{M}(b,c)}(t_1,t_2,\dots,t_v)$ is the discrete Laplace transform of the generalized $bc$-Motzkin numbers. 

Then, these differential forms satisfy the following integral recursion formula:
\begin{multline*}
W_{g,v}^{\widetilde{M}(b,c)}(t_1,t_2,\dots,t_v) = - \frac{1}{64} \frac{1}{2\pi i} \int_{\gamma} \bigg( \frac{1}{t+t_1} + \frac{1}{t-t_1} \bigg) \frac{(t^2-1)^3}{t^2} \,\frac{1}{dt} \, dt_1 \\
\cdot \bigg[ \sum_{j=2}^v \bigg( W_{0,2}^{\widetilde{M}(b,c)}(t,t_j) W_{g,v-1}^{\widetilde{M}(b,c)}(-t,t_2,\dots,\widehat{t_j},\dots,t_v) \\ + W_{0,2}^{\widetilde{M}(b,c)}(-t,t_j) W_{g,v-1}^{\widetilde{M}(b,c)}(t,t_2,\dots,\widehat{t_j},\dots,t_v) \bigg) \\
+ W_{g-1,v+1}^{\widetilde{M}(b,c)}(t,-t,t_2,\dots,t_v) \\
+ \sum_{g_1+g_2=g, \; I \sqcup J = \{2,\dots,v\}, \; \text{stable}} W_{g_1,|I|+1}^{\widetilde{M}(b,c)}(t,t_I) W_{g_2,|J|+1}^{\widetilde{M}(b,c)}(-t,t_J) \bigg]
\end{multline*}
where the curve $\gamma$ is as given in Figure~\ref{fig:gammacontour}. Here, $t_I = (t_{i_1},t_{i_2},\dots,t_{i_{|I|}})$ for an index set $I$, and the notation $\widehat{t_j}$ means that we delete $t_j$ from this sequence. The last sum in the above formula is for all partitions of $g$ and all set partitions of $\{2, \dots, v\}$, and the ``stable'' summation means $2g_1 + |I| - 1 > 0$ and $2g_2 + |J| - 1 > 0$.
\end{thm}

\begin{figure}[hbt]
\includegraphics[height=1.25in]{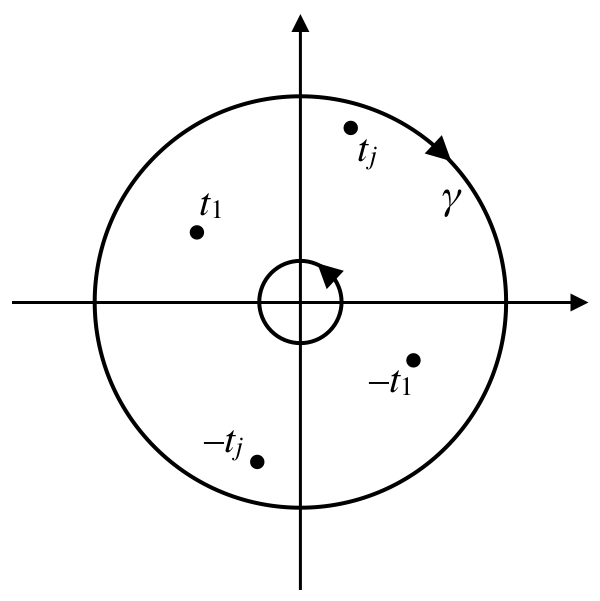}
\caption{The curve $\gamma$.}
\label{fig:gammacontour}
\end{figure}

What is remarkable is that this topological recursion is identical to the topological recursion for generalized Catalan numbers given in \cite{dumitrescumulaselec} (also see Section~\ref{sect:catbackground} of this paper), with the only difference being that we have made a change of variable depending on $b$ and $c$. 

\bigskip

Topological recursion was first introduced by Chekhov, Eynard, and Orantin in their 2006 paper \cite{CEO} (see also \cite{eynardorantin} for a more precise initial definition of topological recursion), but instances of such formulas started appearing earlier. In their paper, the recursion structure was used to calculate multi-resolvent correlation functions of random matrices. However, a topological recursion-like formula had also appeared in a geometry problem before this. Mirzakhani's recursion formula for the Weil-Petersson volume of the moduli space of genus $g$ bordered Riemann surfaces with $n$ geodesic boundaries, which she proved in her thesis in 2004 (see \cite{mirzakhanigeodesics} and \cite{mirzakhaniintersection}) was shown by Eynard and Orantin to also satisfy a topological recursion after applying the Laplace transform (see \cite{eynardorantinweil}).

Since then, many examples of topological recursion formulas have been discovered. As discussed in the preface to \cite{liumulase}, such formulas have appeared in topological quantum field theory and cohomological field theory, intersection numbers of cohomology classes on the moduli space $\overline{\mathcal{M}}_{g,n}$ of stable curves, Gromov-Witten theory, $A$-polynomials and polynomial invariants of hyperbolic knots, WKB analysis of classical ordinary differential equations, enumeration of Hurwitz numbers, Witten-Kontsevich intersection numbers, and moduli spaces of Higgs bundles (see, for example, \cite{BHLM}, \cite{DLN}, \cite{dumitrescumulase}, \cite{EMS}, \cite{mulasezhang}, \cite{norbury}). 

Further, and of particular interest for this paper, topological recursion formulas have also come from various different counting problems in combinatorics, such as counting Grothendieck's dessins d'enfants, and more general counting problems related to graphs drawn on surfaces (see \cite{CMS}, \cite{DMSS}, \cite{DoNorbury}, \cite{mulasepenkava}, \cite{mulasesulkowski}). One example of this is a generalization of the Catalan numbers, which is described in detail in \cite{dumitrescumulaselec} and is also discussed briefly in Section~\ref{sect:catbackground} of this paper. This topological recursion for the generalized Catalan numbers is what prompted the author to look for a topological recursion based on $bc$-Motzkin numbers, which can be viewed as a generalization of the Catalan numbers. It has been observed in \cite{DMSS} that the discrete Laplace transform of edge contraction operations in many graph counting problems corresponds to a topological recursion, and this can be seen in the examples mentioned above, which well as for the $bc$-Motzkin numbers as will be described later in this paper. More generally, the Laplace transform can be identified as a mirror symmetry between the $A$-model side of enumerative geometry and the $B$-model side of holomorphic geometry.

\bigskip

We will not discuss the definition of topological recursion in detail in this paper. The reader is referred to the excellent survey papers \cite{eynard} by Eynard and \cite{borot} by Borot to learn more about topological recursion. We will, however, give a short definition of topological recursion, tailored to combinatorial examples, later in this section. 

Our topological recursion for generalized $bc$-Motzkin numbers gives yet another example of such a combinatorial topological recursion formula. Surprisingly, this formula turns out to be identical, up to a change of variable, to the topological recursion obtained in \cite{dumitrescumulase} for generalized Catalan numbers.

In the Catalan case, an analogy with counting graphs on the genus $g$ surface was used to generalize Catalan numbers to the case of higher genus and greater number of vertices. The $k$th Catalan number is given by
$$C_k = \frac{1}{k+1} \binom{2k}{k}.$$
The first few Catalan numbers for $n = 0, 1, 2, 3, \dots$ are
$$1, 1, 2, 5, 14, 42, 132, 429, 1430, \dots$$
Catalan numbers appear in numerous different counting problems. For example, the $k$th Catalan number (roughly) counts the number of graphs drawn on a sphere with one vertex and $k$ edges, where one of the half-edges incident on this vertex is chosen to be marked with an arrow. (See Figure~\ref{fig:catsphere}.)

\begin{figure}[hbt]
\includegraphics[height=1.25in]{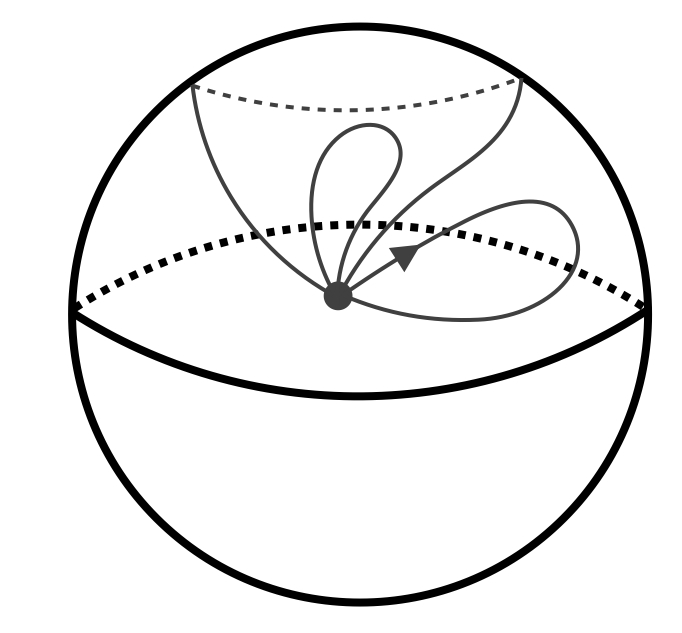}
\caption{Catalan graph with three edges.}
\label{fig:catsphere}
\end{figure}

This graph analogy leads to the definition of generalized Catalan numbers, from which a differential recursion formula on the discrete Laplace transform of these numbers can be proved, and this subsequently gives a topological recursion (see \cite{dumitrescumulaselec}).

Thus, it is natural to ask if any other numbers frequently appearing in combinatorial problems admit a generalization of this kind, and if this generalization satisfies a topological recursion. Our focus in this paper is to study $bc$-Motzkin numbers, which are themselves a generalization of Motzkin numbers.

The $n$th Motzkin number is defined in terms of the Catalan numbers as 
$$M_n = \sum_{k=0}^{\lfloor n/2 \rfloor} \binom{n}{2k} C_k.$$

The first few Motzkin numbers for $n = 0, 1, 2, 3, \dots$ are
$$1, 1, 2, 4, 9, 21, 51, 127, 323, \dots$$
Motzkin numbers were first introduced by Theodore Motzkin in 1948 in his paper \cite{motzkin}. There are many different well-known combinatorial interpretations of Motzkin numbers. For example, the $n$th Motzkin number counts the number of routes on a grid from the coordinate $(0,0)$ to the coordinate $(n,0)$ in $n$ steps, subject to the requirement that the path does not cross below the $x$-axis. (See Figure~\ref{fig:motzkinpaths}.)

\begin{figure}[hbt]
\includegraphics[height=1.25in]{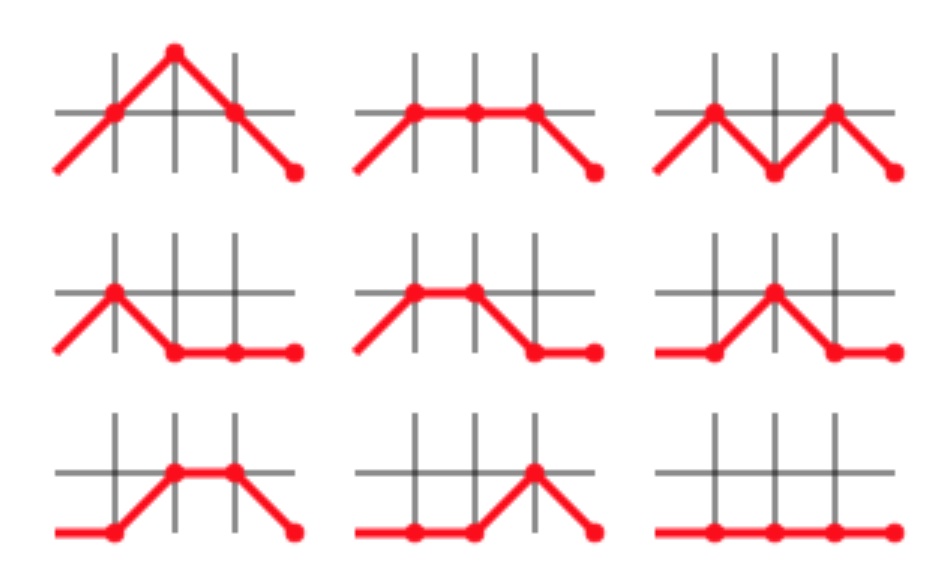}
\caption{All possible Motzkin paths for $n = 4$. (Image credit: Wikipedia.)}
\label{fig:motzkinpaths}
\end{figure}

The $bc$-Motzkin numbers are defined by 
\begin{equation}
\label{eqn:bcmotzkin}
M_n(b,c) = \sum_{k=0}^{\lfloor n/2 \rfloor} \binom{n}{2k} C_k b^{n-2k} c^k,
\end{equation}
where $b,c \in \mathbb{N}$. These numbers were introduced by Sun in 2014 \cite{sun}, and they appear, for example, in the work of lattice models of statistical physics. 

In particular, when $(b,c) = (1,1)$, we recover the definition of Motzkin numbers,
$$M_n(1,1) = M_n.$$
And, Catalan numbers are the special case of $bc$-Motzkin numbers when $(b,c) = (0,1)$, 
$$M_n(0,1) = C_n.$$
Thus, we can also consider these $bc$-Motzkin numbers as a generalization of Catalan numbers, as well as a generalization of Motzkin numbers.  

This leads us to ask the following questions.

\begin{quest}
Do the $bc$-Motzkin numbers admit a higher genus generalization?
\end{quest}

\begin{quest}
If so, then does this generalization satisfy a Catalan-like recursion formula?
\end{quest}

\begin{quest}
If we define the discrete Laplace transform for these higher genus $bc$-Motzkin numbers, does this satisfy a recursion formula? Do we also obtain a topological recursion?
\end{quest}

\begin{quest}
How does the fact that these $bc$-Motzkin numbers are a generalization of Catalan numbers translate into the properties of these recursion formulas?
\end{quest}

In this paper, we answer all these questions affirmatively by giving concrete formulas. An unexpected result is that the recursion formula for the discrete Laplace transform, and hence the topological recursion, is (almost) identical to those for the Catalan numbers. The difference between these results for the Catalan numbers and the generalized $bc$-Motzkin numbers is the coordinate transformation between the defining variables for the discrete Laplace transform and the variables appearing in these recursion formulas. These coordinate transformations give a family of deformations of the ``spectral curve'' of the topological recursion. It is very interesting to see that a surprisingly simple deformation of the spectral curve produces a vast generalization of the Catalan numbers.

Note that Catalan numbers are a special case of $bc$-Motzkin numbers, so this is a two-parameter generalization of the results for generalized Catalan numbers. 

Using an analogy with graph colorings, discussed in more detail in Section~\ref{sect:highergenus}, we define the generalized $bc$-Motzkin numbers as follows.

\begin{Def}
For $b,c \in \mathbb{R}$ with $b \geq 0$ and $c > 0$, we define the generalized $bc$-Motzkin numbers by 
\begin{multline*}
\widetilde{M}_{g,v}(n_1,n_2,\dots,n_v;b,c) = \sum_{\mu_1=0}^{n_1} \sum_{\mu_2=0}^{n_2} \cdots \sum_{\mu_v=0}^{n_v} \binom{n_1}{\mu_1} \binom{n_2}{\mu_2} \cdots \binom{n_v}{\mu_v} \\
\cdot C_{g,v}(\mu_1,\mu_2,\dots,\mu_v) \, b^{(n_1+n_2+\cdots+n_v)-(\mu_1+\mu_2+\cdots+\mu_v)}c^{\mu_1+\mu_2+\cdots+\mu_v}.
\end{multline*}
\end{Def}

From this definition, in Section~\ref{sect:recformula} we prove the following recursion formula for these generalized $bc$-Motzkin numbers:

\begin{thm}
The generalized $bc$-Motzkin numbers satisfy the following formula:
\begin{multline*}
\widetilde{M}_{g,v}(n_1,n_2,\dots,n_v;b,c) - b \widetilde{M}_{g,v}(n_1-1,n_2,\dots,n_v;b,c) \\
\begin{aligned}
& =  c^2 \bigg\{\sum_{j=2}^v n_j \, \widetilde{M}_{g,v-1} (n_1+n_j-2, n_2, \dots, \hat{n_j}, \dots, n_v; b; c) \\
& \quad + \sum_{\zeta+\xi=n_1-2} \bigg[ \widetilde{M}_{g-1,v+1}(\zeta,\xi,n_2,\dots,n_v;b,c) \\
& \quad + \sum_{g_1+g_2=g} \sum_{I \sqcup J = \{2,\dots,v\}} \widetilde{M}_{g_1,|I|+1}(\zeta,n_I;b,c)\widetilde{M}_{g_2,|J|+1}(\xi,n_J;b,c) \bigg] \bigg\} \\
\end{aligned}
\end{multline*}
\end{thm}

Note that this is not truly a recursion formula, as the $(g,v)$ term also appears on the right side in this formula.

Then, in Section~\ref{sect:diffrecformula}, using this formula and the definition of the discrete Laplace transform of the $bc$-Motzkin numbers,
$$F_{g,v}^{\widetilde{M}(b,c)}(x_1,\dots,x_v) = 
\begin{cases}
\displaystyle \sum_{n=1}^{\infty} \frac{\widetilde{M}_{0,1}(n;b,c)}{n} x_1^{-n} - \widetilde{M}_{0,1}(0;b,c) \log x_1 \text{ if $(g,v) = (0,1)$}, \\
\displaystyle \sum_{n_1=1}^{\infty} \cdots \sum_{n_v=1}^{\infty} \frac{\widetilde{M}_{g,v}(n_1,n_2,\dots,n_v;b,c)}{n_1 n_2 \cdots n_v} x_1^{-n_1} \cdots x_v^{-n_v}, \\
\end{cases}$$
we prove that these functions satisfy a differential recursion formula.

\begin{thm}
The discrete Laplace transform $F^{\widetilde{M}(b,c)}_{g,v}(t_1,t_2,\dots,t_v)$ satisfies the following differential recursion formula, for every $(g,v) \neq (0,1),(0,2)$:
\begin{multline*}
\frac{\partial}{\partial t_1} F_{g,v}^{\widetilde{M}(b,c)}(t_1,t_2,\dots,t_v) = - \frac{1}{16} \sum_{j=2}^v \bigg[ \frac{t_j}{t_1^2 - t_j^2} \bigg( \frac{(t_1^2-1)^3}{t_1^2} \frac{\partial}{\partial t_1} F_{g,v-1}^{\widetilde{M}(b,c)}(t_1,\dots,\widehat{t_j},\dots,t_v)\\
\begin{aligned}
& - \frac{(t_j^2-1)^3}{t_j^2} \frac{\partial}{\partial t_j} F_{g,v-1}^{\widetilde{M}(b,c)}(t_2,\dots,t_v) \bigg) \bigg] \\
& - \frac{1}{16} \sum_{j=2}^v \frac{(t_1^2-1)^2}{t_1^2} \bigg[ \frac{\partial}{\partial t_1} F_{g,v-1}^{\widetilde{M}(b,c)}(t_1,\dots,\widehat{t_j},\dots,t_v) \bigg] \\
& - \frac{1}{32} \frac{(t_1^2-1)^3}{t_1^2} \frac{\partial}{\partial u_1} \frac{\partial}{\partial u_2} F_{g-1,v+1}^{\widetilde{M}(b,c)}(u_1,u_2,t_2,\dots,t_v) \bigg|_{u_1=u_2=t_1} \\
& - \frac{1}{32} \frac{(t_1^2-1)^3}{t_1^2} \sum_{g_1+g_2=g, \; I \sqcup J = \{2,\dots,v\}, \; \text{stable}} \frac{\partial}{\partial t_1} F_{g_1,|I|+1}^{\widetilde{M}(b,c)}(t_1,t_I) \cdot \frac{\partial}{\partial t_1} F_{g_2,|J|+1}^{\widetilde{M}(b,c)}(t_1,t_J) \\
\end{aligned}
\end{multline*}
where the ``stable'' summation means $2g_1 + |I| - 1 > 0$ and $2g_2 + |J| - 1 > 0$.

We have the initial conditions
$$\frac{\partial}{\partial t} F_{0,1}^{\widetilde{M}(b,c)}(t) = \frac{8t}{(t+1)(t-1)^3}$$
and
$$\frac{\partial}{\partial t_1} F_{0,2}^{\widetilde{M}(b,c)}(t_1,t_2) = \frac{(t_2+1)}{(t_1-1)(t_1+t_2)}.$$
\end{thm}

Here, we have used the change of variable $t_i = t_i(x_i,b,c)$, for $i \in \{1,2,\dots,v\}$, defined by 
$$\frac{x_i - b}{c} = 2 + \frac{4}{t_i^2-1}.$$
This leads to the topological recursion stated in Theorem~\ref{thm:bcmtoprecintro}.

Thus, in response to the questions posed above, we see that:

\begin{enumerate}
\item The $bc$-Motzkin numbers can indeed be generalized to a higher genus case, just like the Catalan numbers.
\item These generalized $bc$-Motzkin numbers satisfy a recursion formula.
\item We can obtain a differential recursion formula on the discrete Laplace transform of these generalized $bc$-Motzkin numbers, and this leads to a topological recursion.
\item The differential recursion and the topological recursion are (almost) identical to those for the Catalan numbers. The difference is in the coordinate transformation from $x_i$ to $t_i$ given above.
\end{enumerate}

\bigskip

Now, what, precisely, is topological recursion? The combinatorial definition of topological recursion which we will be using in this paper is as given in \cite{DMSS}.

\begin{Def}
Let $t$ be a choice of coordinate on $\mathbb{P}^1$. Let $S \subset \mathbb{P}^1$ be a finite collection of points and compact real curves such that the complement $\Sigma = \mathbb{P}^1 \setminus S$ is connected. The spectral curve of genus $0$ is the data $(\Sigma,\pi)$ consisting of a Riemann surface $\Sigma$ and a simply ramified holomorphic map 
$$\pi : \Sigma \ni t \mapsto \pi(t) = x \in \mathbb{P}^1$$
so that its differential $dx$ has only simple zeros. Let $R = \{p_1,p_2,\dots,p_r\} \subset \Sigma$ denote the ramification points, and let $U = \sqcup_{j=1}^r U_j$ denote the disjoint union of small neighborhoods $U_j$ around each $p_j$, such that $\pi: U_j \to \pi(U_j) \subset \mathbb{P}^1$ is a double-sheeted covering ramified only at $p_j$. We denote by $\overline{t}$ the local Galois conjugate of $t \in U_j$ (i.e., interchanging the two sheets). The canonical sheaf of $\Sigma$ is denoted by $\mathcal{K}$. Because of our choice of coordinate $t$, we have a preferred basis $dt$ for $\mathcal{K}$ and $\partial/\partial t$ for $\mathcal{K}^{-1}$. The meromorphic differential forms $W_{g,n}(t_1,t_2,\dots,t_v)$ are said to satisfy the Eynard-Orantin topological recursion if the following conditions are satisfied:

\begin{enumerate}

\item $W_{0,1}(t) \in H^0(\Sigma,\mathcal{K})$.

\item We have 
$$W_{0,2}(t_1,t_2) = \frac{dt_1\cdot dt_2}{(t_1-t_2)^2} - \pi^*\frac{dx_1\cdot dx_2}{(x_1-x_2)^2} \in H^0(\Sigma \times \Sigma, \mathcal{K}^{\otimes 2}(2\Delta)),$$ 
where $\Delta$ is the diagonal of $\Sigma \times \Sigma$.

\item The recursion kernel $K_j(t,t_1) \in H^0(U_j \times \Sigma, (\mathcal{K}_{U_j}^{-1}  \otimes \mathcal{K})(\Delta))$ for $t \in U_j$ and $t_1 \in \Sigma$ is defined by
$$K_j(t,t_1) = \frac{1}{2} \frac{\int_t^{\overline{t}} W_{0,2}(\cdot,t_1)}{W_{0,1}(\overline{t}) - W_{0,1}(t)}.$$
The kernel is an algebraic operator that multiplies $dt_1$ while contracting $dt$.

\item The general differential forms $W_{g,n}(t_1,t_2,\dots,t_v) \in H^0(\Sigma^v, \mathcal{K}(*R)^{\otimes v})$ are meromorphic symmetric differential forms with poles only at the ramification points $R$ for $2g-2+v>0$, and are given by the recursion formula 
\begin{multline*}
W_{g,v}(t_1,t_2,\dots,t_v) = \frac{1}{2\pi i} \sum_{j=1}^r \oint_{U_j} K_j(t,t_1) \bigg[ W_{g-1,v+1}(t,\overline{t},t_2,\dots,t_v) \\
+ \sum_{g_1+g_2=g, \; I \sqcup J = \{2,\dots,v\}, \; \text{no $(0,1)$ terms}} W_{g_1,|I|+1}(t,t_I) W_{g_2,|J|+1}(\overline{t},t_J) \bigg].
\end{multline*}
Here, the integration is taken with respect to $t_j \in U_j$ along a positively oriented simple closed loop around $p_j$, and $t_I = (t_i)_{i \in I}$ for a subset $I \subset \{1,2,\dots,v\}$. 

\item The differential form $W_{1,1}(t_1)$ requires a separate treatment since $W_{0,2}(t_1,t_2)$ is regular at the ramification points but has poles elsewhere.
\begin{align*}
W_{1,1}(t_1) & = \frac{1}{2\pi i} \sum_{j=1}^r \oint_{U_j} K_j(t,t_1) \bigg[ W_{0,2}(u,v) + \pi^* \frac{dx(u) \cdot dx(v)}{(x(u) - x(v))^2} \bigg] \bigg|_{u = t, v = \overline{t}} \\
& = \frac{1}{2\pi i} \sum_{j=1}^r \oint_{U_j} K_j(t,t_1) \bigg[ \frac{dt \cdot d\overline{t}}{(t - \overline{t})^2} \bigg].
\end{align*}
Let $y: \Sigma \to \mathbb{C}$ be a holomorphic function defined by the equation
$$W_{0,1}(t) = y(t) \; dx(t).$$
Equivalently, we can define the function by contraction $y = i_{X}W_{0,1}$, where $X$ is the vector field on $\Sigma$ dual to $dx(t)$ with respect to the coordinate $t$. Then, we have an embedding
$$\Sigma \ni t \mapsto (x(t),y(t)) \in \mathbb{C}^2.$$

\item If the spectral curve has at most two branches, then we choose a preferred coordinate $t$ with the branch points located at $t = \infty$ and $t = 0$. This results in differentials $W_{g,v}$ that are Laurent polynomials in $t$ and serves to simplify many of the residue calculations.

\end{enumerate}

\end{Def}

This paper is organized as follows. In the second section, we review the results in \cite{dumitrescumulaselec} regarding generalized Catalan numbers, first defining these generalized Catalan numbers via an analogy with graphs on the genus $g$ surface. This gives a recursive definition of the generalized Catalan numbers. From this formula, one can then prove a differential recursion formula on the discrete Laplace transform of these numbers, then from this formula prove a subsequent topological recursion. In the third section, we introduce our definition of higher genus $bc$-Motzkin numbers, which we define via an analogy with counting colored graphs on the genus $g$ surface. In the fourth section, we state and give a combinatorial proof of a recursion formula for these generalized $bc$-Motzkin numbers. An algebraic proof is given in the appendix. In the fifth section, we show that this recursion formula does actually lead to a differential recursion formula on the discrete Laplace transform of these generalized $bc$-Motzkin numbers, and, further, that with a particular choice of change of variable the differential recursion formula is identical to the differential recursion formula for generalized Catalan numbers. Finally, in the sixth section, we show that this differential recursion formula for generalized $bc$-Motzkin numbers leads to the topological recursion given in Theorem~\ref{thm:bcmtoprecintro} above.


\section{Background: Higher Genus Catalan Numbers}
\label{sect:catbackground}

In this section we summarize some results obtained by Dumitrescu and Mulase in \cite{dumitrescumulaselec}, regarding generalized Catalan numbers.

As introduced in Section~\ref{sect:introbasicdefs}, the $k$th Catalan number (roughly) counts the number of graphs drawn on a sphere with one vertex and $k$ edges, where one of the half-edges incident on this vertex is chosen to be marked with an arrow. (See Figure~\ref{fig:catsphere}.)

This analogy with counting graphs on a sphere and having one vertex can be extended, and is used to define the generalized Catalan numbers 
$$C_{g,v}(\mu_1,\mu_2,\dots,\mu_v),$$
which count the number of graphs drawn on the (oriented) genus $g$ surface with $v$ vertices that give a cell decomposition of the surface, where the $i$th vertex has degree $\mu_i$, and at each vertex one of the incident half-edges is chosen to be marked with an arrow  (see \cite{dumitrescumulaselec} and \cite{walshlehman}). 

We will call such a graph a Catalan graph of degree $(\mu_1,\mu_2,\dots,\mu_v)$ on the genus $g$ surface. (See Figure~\ref{fig:image1} for an example of such a graph.)

\begin{figure}[hbt]
\includegraphics[height=1.25in]{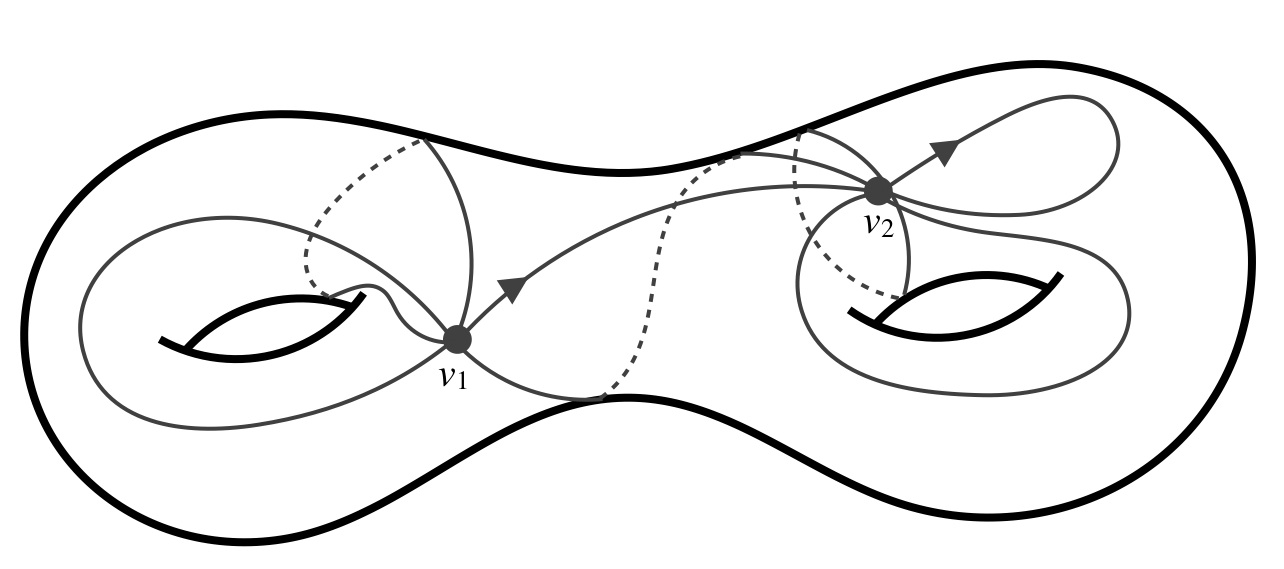}
\caption{Catalan graph of degree $(6,8)$ on the genus $2$ surface.}
\label{fig:image1}
\end{figure}

Now, the Catalan numbers are known to satisfy the recursion formula, 
\begin{equation}
\label{eqn:catrec}
C_n = \sum_{a+b=n-1} C_a C_b.
\end{equation}
This formula has the higher genus analogue given in the following proposition.

\begin{prop}
\label{prop:gencatrec}
The generalized Catalan numbers satisfy
\begin{equation}
\label{eqn:gencatrec}
\begin{aligned}
C_{g,v}(\mu_1,\mu_2,\dots,\mu_v) & = \sum_{j=2}^v \mu_j C_{g,v-1} (\mu_1 + \mu_j - 2, \mu_2, \dots, \widehat{\mu_j}, \dots, \mu_v) \\
& \quad + \sum_{\alpha + \beta = \mu_1 - 2} \bigg[ C_{g-1,v+1}(\alpha,\beta,\mu_2,\dots,\mu_v) \\
& \quad + \sum_{g_1+g_2=g, I \sqcup J = \{2,\dots,v\}} C_{g_1,|I|+1}(\alpha,\mu_I)C_{g_2,|J|+1}(\beta,\mu_J)\bigg] \\
\end{aligned}
\end{equation}
where $\mu_I = (\mu_{i_1},\mu_{i_2},\dots,\mu_{i_{|I|}})$ for an index set $I$, the notation $\widehat{\mu_j}$ means that we delete $\mu_j$ from this sequence, and the third sum in the above formula is for all partitions of $g$ and set partitions of $\{2, \dots, v\}$.
\end{prop}

\begin{rem}
Observe that this is not truly a recursion formula, since the $(g,v)$ term also appears on the right side of the equation.
\end{rem}

This formula serves to define the generalized Catalan numbers.

\begin{rem}
With this definition, $C_{0,1}(\mu)$ is actually the aerated Catalan numbers,
$$C_{0,1}(\mu) = 
\begin{cases} 
C_{\mu/2} & \text{if $\mu$ is even,} \\ 
0 & \text{if $\mu$ is odd.} \\ 
\end{cases}$$
\end{rem}

In Section~\ref{sect:recformula} of this paper, we will show that the generalized $bc$-Motzkin numbers satisfy a similar formula, which reduces to the Catalan formula when $b = 0$ and $c = 1$.

To prove this formula, we may proceed as follows. A proof of this result is also given in \cite{dumitrescumulaselec}.

\begin{proof}

We start by contracting the edge $E$ corresponding to the arrowed half-edge which is incident on $v_1$. There are two cases which we need to study.

\emph{Case 1.}  Assume $E$ connects $v_1$ and $v_j$ ($j \neq 1$). 

Contracting $E$ replaces the two vertices $v_1$ and $v_j$ with one vertex, of degree $\mu_1 + \mu_j - 2$. To make this counting bijective, we need to be able to go back to the original graph if we are given $\mu_1$ and $\mu_j$, which are the degrees of $v_1$ and $v_j$, respectively. We do this by putting an arrow on the half-edge next to $E$ with respect to the counterclockwise cyclic ordering that comes from the orientation on the surface. (See Figure~\ref{fig:catalanproof1}.) However, we must first delete the arrow that was assigned to a half-edge incident on $v_j$ in the original graph. So, there are $\mu_j$ different graphs which produce the same result.

This gives the first term in the Catalan recursion.

\begin{figure}[hbt]
\includegraphics[height=1.25in]{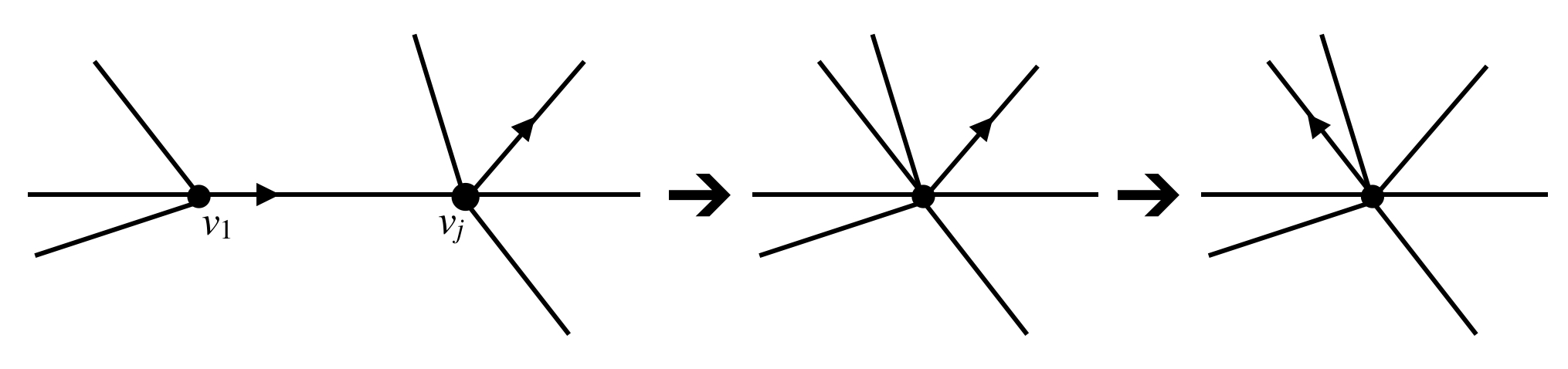}
\caption{}
\label{fig:catalanproof1}
\end{figure}

\emph{Case 2.} Assume $E$ is a loop on $v_1$. 

Contracting $E$ separates the incident half-edges at $v_1$ into two collections, with $\alpha$ edges on one side and $\beta$ edges on the other (note that $\alpha$ or $\beta$ may be zero). Since $E$ is a loop, contracting it causes pinching on the surface and produces a double point. We separate the double point into two new vertices. The result may be a surface of genus $g-1$, or two surfaces, of genus $g_1$ and $g_2$ with $g_1+g_2 = g$. We assign an arrow to the half edge(s) next to $E$, again with respect to the counterclockwise cyclic ordering that comes from the orientation on the surface. (See Figure~\ref{fig:catalanproof2}.) 

This gives the remaining terms in the Catalan recursion.

\begin{figure}[hbt]
\includegraphics[height=1in]{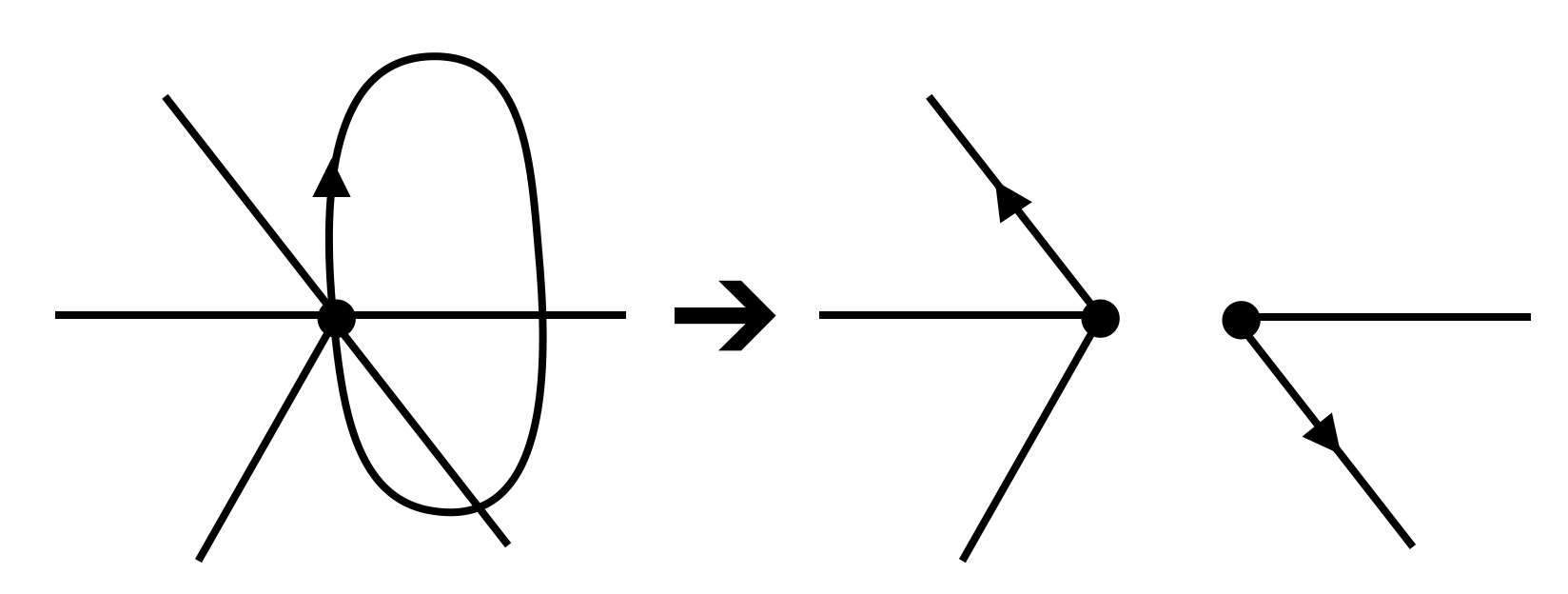}
\caption{}
\label{fig:catalanproof2}
\end{figure}
\end{proof}

Let us now look at some examples.

\begin{ex}

Contracting the arrowed half-edge in the Catalan graph of degree $4$ on the genus $1$ surface in Figure~\ref{fig:catalanproof3} gives a Catalan graph of degree $(1,1)$ on the genus $0$ surface.

\begin{figure}[hbt]
\includegraphics[height=1in]{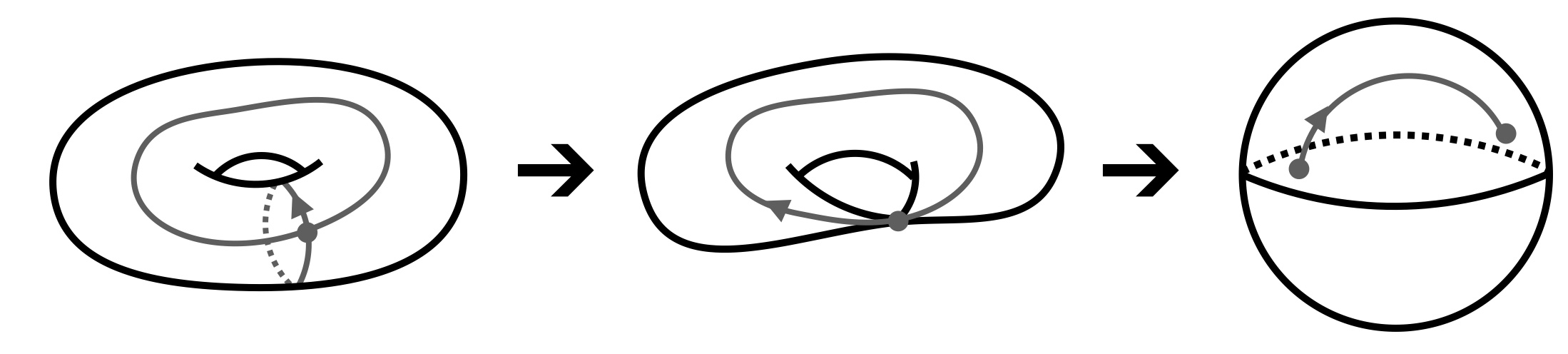}
\caption{}
\label{fig:catalanproof3}
\end{figure}

\end{ex}

\begin{ex}

Contracting the arrowed half-edge in the Catalan graph of degree $(4,4)$ on the genus $1$ surface in Figure~\ref{fig:catalanproof4} gives a Catalan graph of degree $(4,1)$ on the genus $1$ surface and a Catalan graph of degree $0$ on the genus $0$ surface. 

\begin{figure}[hbt]
\includegraphics[height=1in]{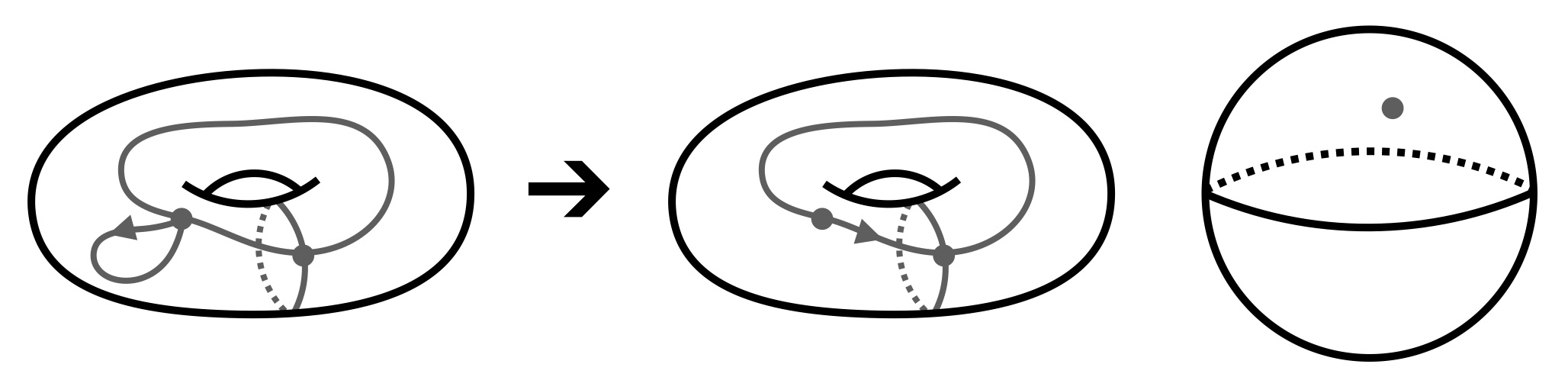}
\caption{}
\label{fig:catalanproof4}
\end{figure}

\end{ex}

\begin{ex}

Contracting the arrowed half-edge in the Catalan graph of degree $6$ on the genus $0$ surface in Figure~\ref{fig:catalanproof5} gives a Catalan graph of degree $2$ on the genus $0$ surface and a Catalan graph of degree $2$ on the genus $0$ surface.

\begin{figure}[hbt]
\includegraphics[height=1.25in]{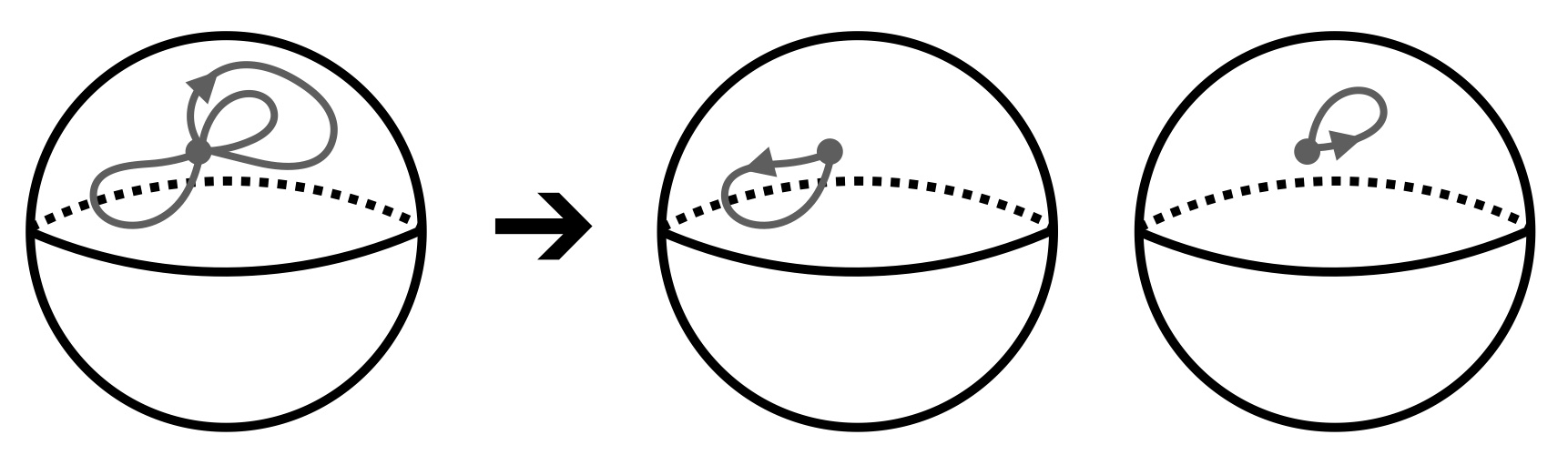}
\caption{}
\label{fig:catalanproof5}
\end{figure}

\end{ex}

To obtain a topological recursion from Proposition~\ref{prop:gencatrec}, we first need to look at the discrete Laplace transform for the generalized Catalan numbers.

This is defined by 
\begin{equation}
\label{eqn:catdlt}
F_{g,v}^C(x_1,x_2,\dots,x_v) = 
\begin{cases}
\displaystyle \sum_{n=1}^{\infty} \frac{C_{0,1}(n)}{n} x_1^{-n} - C_{0,1}(0) \log x_1 \text{ if $(g,v) = (0,1)$}, \\
\displaystyle \sum_{\mu_1=1}^{\infty} \cdots \sum_{\mu_v=1}^{\infty} \frac{C_{g,v}(\mu_1,\mu_2,\dots,\mu_v)}{\mu_1 \mu_2 \cdots \mu_v} x_1^{-\mu_1} x_2^{-\mu_2} \cdots x_v^{-\mu_v}, \\
\end{cases}
\end{equation}
and it can be calculated by a recursion formula. The recursion formula is actually a ``differential'' recursion, on 
$$\frac{\partial}{\partial t_1} F_{g,v}^C(t_1,t_2,\dots,t_v),$$ 
where $(t_1,t_2,\dots,t_v)$ is a particular choice of change of variables, given by
\begin{equation}
\label{eqn:catvarchange}
x_i = 2 + \frac{4}{t_i^2-1}, \qquad i \in \{1,2,\dots,v\}.
\end{equation}

One may compute from the definition of the discrete Laplace transform, using Proposition~\ref{prop:gencatrec} and the above choice of change of variable, that we have
\begin{equation}
\label{eqn:catF01}
\frac{\partial}{\partial t} F_{0,1}^{C}(t) = \frac{8t}{(t+1)(t-1)^3}
\end{equation}
and
\begin{equation}
\label{eqn:catF11}
\frac{\partial}{\partial t_1} F_{0,2}^{C}(t_1,t_2) = \frac{(t_2+1)}{(t_1-1)(t_1+t_2)}.
\end{equation}

More generally, the following result is given in \cite{dumitrescumulaselec}.

\begin{prop}[Dumitrescu, Mulase, Safnuk, Sorkin `13]
\label{prop:dmss13diff}
The discrete Laplace transform $F_{g,v}^C(t_1,t_2,\dots,t_v)$ satisfies the following differential recursion equation for every $(g,v) \neq (0,1),(0,2)$.
\begin{multline}
\label{eqn:catdiffrec}
\frac{\partial}{\partial t_1} F_{g,v}^{C}(t_1,t_2,\dots,t_v) = - \frac{1}{16} \sum_{j=2}^v \bigg[ \frac{t_j}{t_1^2 - t_j^2} \bigg( \frac{(t_1^2-1)^3}{t_1^2} \frac{\partial}{\partial t_1} F_{g,v-1}^{C}(t_1,\dots,\widehat{t_j},\dots,t_v) \\
\begin{aligned}
& - \frac{(t_j^2-1)^3}{t_j^2} \frac{\partial}{\partial t_j} F_{g,v-1}^{C}(t_2,\dots,t_v) \bigg) \bigg] \\
& - \frac{1}{16} \sum_{j=2}^v \frac{(t_1^2-1)^2}{t_1^2} \bigg[ \frac{\partial}{\partial t_1} F_{g,v-1}^{C}(t_1,\dots,\widehat{t_j},\dots,t_v) \bigg] \\
& - \frac{1}{32} \frac{(t_1^2-1)^3}{t_1^2} \frac{\partial}{\partial u_1} \frac{\partial}{\partial u_2} F_{g-1,v+1}^{C}(u_1,u_2,t_2,\dots,t_v) \bigg|_{u_1=u_2=t_1} \\
& - \frac{1}{32} \frac{(t_1^2-1)^3}{t_1^2} \sum_{g_1+g_2=g, \; I \sqcup J = \{2,\dots,v\}, \; \text{stable}} \frac{\partial}{\partial t_1} F_{g_1,|I|+1}^{C}(t_1,t_I) \cdot \frac{\partial}{\partial t_1} F_{g_2,|J|+1}^{C}(t_1,t_J), \\
\end{aligned}
\end{multline}
where the ``stable'' summation means $2g_1 + |I| - 1 > 0$ and $2g_2 + |J| - 1 > 0$.
\end{prop}

Remarkably, the discrete Laplace transform of the generalized $bc$-Motzkin numbers satisfies an identical differential recursion formula, with the only difference being that the change of variable from $x_i$ to $t_i$ depends also on $b$ and $c$. When $b = 0$ and $c = 1$, it reduces to the same change of variables as in the Catalan case.

The proof of Proposition~\ref{prop:dmss13diff} is very similar to the proof in the $bc$-Motzkin numbers case (see Section~\ref{sect:diffrecformula} and Appendix~\ref{subsect:appendixdiffrec}), hence a proof of this result is not given here.

The differential recursion formula in Proposition~\ref{prop:dmss13diff} leads to a recursion in $2g-2+v$ on the Eynard-Orantin differential forms on $(\mathbb{P}^1)^v$, defined by
\begin{equation}
\label{eqn:eynardorantin}
W_{g,v}^{C}(t_1,t_2,\dots,t_v) = d_{t_1} \cdots d_{t_n} F_{g,v}^{C}(t_1,t_2,\dots,t_v).
\end{equation}
The resulting recursion formula is called the topological recursion for these generalized Catalan numbers.

\begin{prop}[Dumitrescu, Mulase, Safnuk, Sorkin `13]
\label{prop:dmss13top}
For $2g-2+v > 0$, these differential forms satisfy the following integral recursion equation:
\begin{multline}
W_{g,v}^{C}(t_1,t_2,\dots,t_v) = - \frac{1}{64} \frac{1}{2\pi i} \int_{\gamma} \bigg( \frac{1}{t+t_1} + \frac{1}{t-t_1} \bigg) \frac{(t^2-1)^3}{t^2} \, \frac{1}{dt} \, dt_1 \\
\cdot \bigg[ \sum_{j=2}^v \bigg( W_{0,2}^{C}(t,t_j) W_{g,v-1}^{C}(-t,t_2,\dots,\widehat{t_j},\dots,t_v) + W_{0,2}^{C}(-t,t_j) W_{g,v-1}^{C}(t,t_2,\dots,\widehat{t_j},\dots,t_v) \bigg) \\
+ W_{g-1,v+1}^{C}(t,-t,t_2,\dots,t_v) + \sum_{g_1+g_2=g, \; I \sqcup J = \{2,\dots,v\}, \; \text{stable}} W_{g_1,|I|+1}^{C}(t,t_I) W_{g_2,|J|+1}^{C}(-t,t_J) \bigg] 
\end{multline}
Here, $\gamma$ is the contour in Figure~\ref{fig:gammacontour}.
\end{prop}

To prove this result, one may apply the definition of the Eynard-Orantin differential forms to the differential recursion formula in Proposition~\ref{prop:dmss13diff} to obtain a formula for the $W_{g,v}^{C}(t_1,t_2,\dots,t_v)$, then simplify the integral around the contour $\gamma$ on the right side of the equation in Proposition~\ref{prop:dmss13top} to show that the two formulas for $W_{g,v}^{C}(t_1,t_2,\dots,t_v)$ are indeed equal.

Similarly to the case of the differential recursion formula, the generalized $bc$-Motzkin numbers satisfy an identical result, again with the only difference being that the change of variables from $x_i$ to $t_i$ now depends on $b$ and $c$. Since the differential recursion formula is identical to the (more general) $bc$-Motzkin numbers case, a proof of the Catalan case is not given here. The proof of the topological recursion for generalized $bc$-Motzkin numbers is given in Section~\ref{sect:toprecformula} and Appendix~\ref{subsect:appendixtoprec}.


\section{Higher Genus $bc$-Motzkin Numbers}
\label{sect:highergenus}

Recall that in Section~\ref{sect:introbasicdefs} we defined the $bc$-Motzkin numbers $M_n(b,c)$ by equation \eqref{eqn:bcmotzkin}, following the definition in \cite{sun}. Now, we will define a slightly different version of $bc$-Motzkin numbers, which are better suited for our purposes in this paper, by 
\begin{equation}
\widetilde{M}_{0,1}(n;b,c) = \sum_{k=0}^{\lfloor n/2 \rfloor} \binom{n}{2k} C_k b^{n-2k} c^{2k},
\end{equation}
where $b \in \mathbb{R}_{\geq0}$ and $c \in \mathbb{R}_{>0}$.

The difference is that we now write $c^{2k}$ rather than $c^{k}$, and we also allow $b$ to be a nonnegative real number, and we allow $c$ to be a positive real number, instead of restricting to the natural numbers.

Thus, we may write these $(0,1)$-$bc$-Motzkin numbers in terms of the $(0,1)$-Catalan numbers as
\begin{equation}
\label{eqn:bcmotzkin01}
\widetilde{M}_{0,1}(n;b,c) = \sum_{\mu=0}^{n} \binom{n}{\mu} C_{0,1}(\mu) b^{n-\mu} c^{\mu}.
\end{equation}

The $(0,1)$-$bc$-Motzkin numbers satisfy a recursion formula, similar to the $(0,1)$-Catalan numbers. 

\begin{prop}
\label{prop:bcM01rec}
For $n \geq 1$, the $(0,1)$-$bc$-Motzkin numbers $M_{0,1}(n;b,c)$ satisfy
\begin{equation}
\widetilde{M}_{0,1}(n;b,c) - b \, \widetilde{M}_{0,1}(n-1;b,c) = c^2 \sum_{\alpha+\beta=n-2} \widetilde{M}_{0,1}(\alpha,b,c) \widetilde{M}_{0,1}(n_2,b,c).
\end{equation}
\end{prop}

This recursion formula was initially proved by Wang and Zhang in \cite{wangzhang}. A proof of this recursion formula is also given in Appendix~\ref{subsect:appendix01proof}. 

We would like to generalize the $(0,1)$-$bc$-Motzkin numbers, just as was done for Catalan numbers. To do this, we first construct an analogy with colorings of Catalan graphs on the genus $g$ surface. We are led to make the following new definition.

\begin{Def}
\label{def:bcmcoloring}
Let 
$$\Gamma_{g,v}(n_1,n_2,\dots,n_v;\mu_1,\mu_2,\dots,\mu_v)$$ 
denote the number of ways to color Catalan graphs of degree $(\mu_1,\mu_2,\dots,\mu_v)$ on the genus $g$ surface, subject to the following requirements, for all $i \in \{1,2,\dots,v\}$.
\begin{enumerate}
\item The degree $\mu_i$ of the $i$th vertex is less than or equal to the number of colors $n_i$ with which the half-edges adjacent to that vertex can be colored.
\item We choose $\mu_i$ colors from the set of $n_i$ colors with which to color the half-edges adjacent to the $i$th vertex.
\item The set of $n_i$ colors associated with the $i$th vertex is ordered. Of the $\mu_i$ colors chosen from this set, the lowest-indexed color is assigned to the arrowed half-edge incident on that vertex, and the colors increase in ordering as we traverse the edges by going counterclockwise around the vertex.
\end{enumerate}
\end{Def}

Then, we define the generalized $bc$-Motzkin number to be equal to the following weighted sum of cardinalities of this set:
\begin{multline}
\widetilde{M}_{g,v}(n_1,n_2,\dots,n_v;b,c) = \sum_{\mu_1=0}^{n_1} \sum_{\mu_2=0}^{n_2} \cdots \sum_{\mu_v=0}^{n_v} |\Gamma_{g,v}(n_1,n_2,\dots,n_v;\mu_1,\mu_2,\dots,\mu_v)| \\
\cdot b^{(n_1+n_2+\cdots+n_v)-(\mu_1+\mu_2+\cdots+\mu_v)} c^{\mu_1+\mu_2+\cdots+\mu_v}.
\end{multline}
Here, $b \in \mathbb{R}_{\geq0}$ and $c \in \mathbb{R}_{>0}$.

\begin{ex}
The degree $(6,8)$ Catalan graph in Figure~\ref{fig:image2} has been colored according to the requirements in Definition~\ref{def:bcmcoloring}, with $6$ colors $(k_2^{(1)},k_3^{(1)},k_5^{(1)},k_6^{(1)},k_7^{(1)},k_{10}^{(1)})$ chosen from a set of $10$ possible colors and assigned to the first vertex $v_1$, going counterclockwise around this vertex with respect to the orientation on this genus $2$ surface and starting with the lowest indexed color on the arrowed half-edge in the underlying Catalan graph. And, $8$ colors $(k_1^{(2)},k_2^{(2)},k_3^{(2)},k_4^{(2)},k_5^{(2)},k_6^{(2)},k_7^{(2)},k_8^{(2)})$ have been chosen from a set of $8$ possible colors and assigned to the second vertex $v_2$. (The superscript $^{(1)}$ on the first set of colors denotes that they are to be assigned to the vertex $v_1$, and similarly for $v_2$.) 

\begin{figure}[hbt]
\includegraphics[height=1.5in]{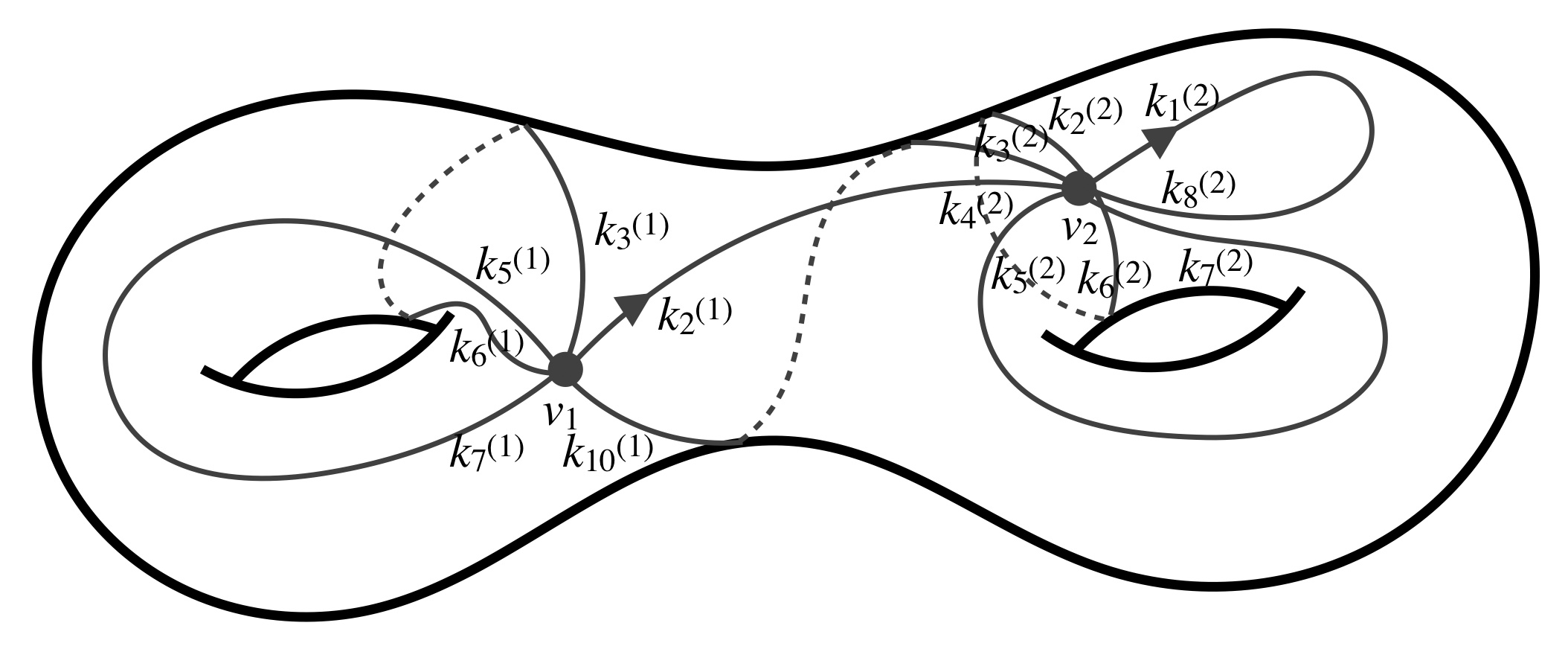}
\caption{}
\label{fig:image2}
\end{figure}

\end{ex}

\begin{rem}
In the case when $g=0$ and $v=1$, we see that Definition~\ref{def:bcmcoloring} gives
$$\widetilde{M}_{0,1}(n;b,c) = \sum_{\mu=0}^n |\Gamma_{0,1}(n;\mu)| \, b^{n-\mu} c^{\mu},$$
and $|\Gamma_{0,1}(n;\mu)| = \binom{n}{\mu} C_{0,1}(\mu)$, since once we have chosen $\mu$ different colors from the set of $n$ possible colors, this uniquely determines the coloring on the Catalan graph of degree $\mu$. Thus, this definition does indeed coincide with our earlier definition of $(0,1)$-$bc$-Motzkin numbers in equation \eqref{eqn:bcmotzkin01}.
\end{rem}

More generally, we have the following equivalent definition of the generalized $bc$-Motzkin numbers, in terms of the generalized Catalan numbers defined in Section~\ref{sect:catbackground}.

\begin{Def}
\label{def:genbcmotzkin}
For $b,c \in \mathbb{R}$ with $b \geq 0$ and $c > 0$, we define the generalized $bc$-Motzkin numbers by 
\begin{multline}
\label{eqn:genbcmotzkin}
\widetilde{M}_{g,v}(n_1,n_2,\dots,n_v;b,c) = \sum_{\mu_1=0}^{n_1} \sum_{\mu_2=0}^{n_2} \cdots \sum_{\mu_v=0}^{n_v} \binom{n_1}{\mu_1} \binom{n_2}{\mu_2} \cdots \binom{n_v}{\mu_v} \\
\cdot C_{g,v}(\mu_1,\mu_2,\dots,\mu_v) \, b^{(n_1+n_2+\cdots+n_v)-(\mu_1+\mu_2+\cdots+\mu_v)}c^{\mu_1+\mu_2+\cdots+\mu_v}.
\end{multline}
\end{Def}

Note that, as expected, this definition is symmetric in its arguments $n_i$, since the generalized Catalan numbers are symmetric in their arguments $\mu_i$. 

\begin{rem}
If $(b,c) = (0,1)$, then we recover the generalized Catalan numbers,
\begin{equation}
\widetilde{M}_{g,v}(n_1,n_2,\dots,n_v;0,1) = C_{g,v}(n_1,n_2,\dots,n_v).
\end{equation}
\end{rem}

Hence these generalized $bc$-Motzkin numbers are also a generalization of Catalan numbers.


\section{The Recursion Formula for Generalized $bc$-Motzkin Numbers}
\label{sect:recformula}

We wish to obtain a recursion formula for generalized $bc$-Motzkin numbers. To do this, we first observe that the generalized $bc$-Motzkin numbers are defined in terms of the generalized Catalan numbers in Definition~\ref{def:genbcmotzkin}, and Proposition~\ref{prop:gencatrec} gives a ``recursion'' formula for the generalized Catalan numbers. Thus, we are led to prove the following result. 

\begin{thm}
\label{thm:genbcmrec}
The generalized $bc$-Motzkin numbers satisfy the following formula:

\begin{multline}
\widetilde{M}_{g,v}(n_1,n_2,\dots,n_v;b,c) - b \widetilde{M}_{g,v}(n_1-1,n_2,\dots,n_v;b,c) \\
\begin{aligned}
& =  c^2 \bigg\{\sum_{j=2}^v n_j \, \widetilde{M}_{g,v-1} (n_1+n_j-2, n_2, \dots, \hat{n_j}, \dots, n_v; b; c) \\
& \quad + \sum_{\zeta+\xi=n_1-2} \bigg[ \widetilde{M}_{g-1,v+1}(\zeta,\xi,n_2,\dots,n_v;b,c) \\
& \quad + \sum_{g_1+g_2=g, \; I \sqcup J = \{2,\dots,v\}} \widetilde{M}_{g_1,|I|+1}(\zeta,n_I;b,c)\widetilde{M}_{g_2,|J|+1}(\xi,n_J;b,c) \bigg] \bigg\} \\
\end{aligned}
\end{multline}

\end{thm}

Observe that, as in the Catalan case, this is not truly a recursion formula, since the $(g,v)$ term also appears on the right side of the equation. 

\begin{rem}
The $(0,1)$ case of Theorem~\ref{thm:genbcmrec}, for the $(0,1)$-$bc$-Motzkin numbers, is the same as in Proposition~\ref{prop:bcM01rec} of Section~\ref{sect:highergenus}, as expected. Further, when $b = 0$ and $c = 1$, we recover the formula in Proposition~\ref{prop:gencatrec} for generalized Catalan numbers.
\end{rem} 

The above theorem may be proved algebraically, and a complete proof is given in Appendix~\ref{subsect:appendixrecursion}. In brief, one applies Definition~\ref{def:genbcmotzkin} of the generalized $bc$-Motzkin numbers and Pascal's identity to simplify the left side of the above relation, then plugs in the Catalan formula from Proposition~\ref{prop:gencatrec}, simplifies, and re-writes the result entirely in terms of generalized $bc$-Motzkin numbers. The key ingredients in this proof are Vandermonde’s identity,
\begin{equation}
\sum_{i+j=k}  \binom{a}{i}  \binom{b}{j} =  \binom{a + b}{k},
\end{equation}
and the ``Vandermonde-like'' identity
\begin{equation}
\sum_{a+b=n} \binom{a}{i} \binom{b}{j} = \binom{n+1}{i+j+1},
\end{equation}
of which a proof is provided in Appendix~\ref{subsect:appendixvandermonde}.

Alternatively, this theorem can be proved combinatorially, using Definition~\ref{def:bcmcoloring} for generalized $bc$-Motzkin numbers in terms of the graph coloring analogy, and following the same approach as for the proof of the Catalan ``recursion'' formula.

In particular, referencing the proof of Proposition~\ref{prop:gencatrec} in Section~\ref{sect:catbackground} for the Catalan case, this combinatorial proof goes as follows.

\begin{proof}

Recall that $\widetilde{M}_{g,v}(n_1,n_2,\dots,n_v;b,c)$ equals the number of ways to color all Catalan graphs of degree $(\mu_1,\mu_2,\dots,\mu_v)$ on the genus $g$ surface and having $v$ vertices, subject to the coloring requirements in Definition~\ref{def:bcmcoloring}, where $\mu_i \leq n_i$ for all $i$, and the sum of cardinalities of sets of colored Catalan graphs is weighted in a particular way by $b$ and $c$. 

Now, in each underlying Catalan graph, we fix the color on the arrowed half-edge incident on $v_1$  to be the lowest indexed color in the set of possible colors for $v_1$, call it $k_1^{(1)}$. Then, the resulting number of ways to color such Catalan graphs, subject to the necessary restrictions and weightings, is
$$\widetilde{M}_{g,v}(n_1,n_2,\dots,n_v;b,c) - b \widetilde{M}_{g,v}(n_1-1,n_2,\dots,n_v;b,c).$$
This can be seen by observing that $\widetilde{M}_{g,v}(n_1-1,n_2,\dots,n_v;b,c)$ gives the number of ways to color the graphs in the definition of $\widetilde{M}_{g,v}(n_1,n_2,\dots,n_v;b,c)$ without using the color $k_1^{(1)}$ on the vertex $v_1$. The exponent of $b$ essentially counts the difference between the number of possible colors which can be assigned to a vertex and the degree of that vertex.

To show that this equals the right side of the above formula, we will to contract the edge $E$ corresponding to the edge incident on $v_1$ which is colored by $k_1^{(1)}$. As in the proof of the analogous result for generalized Catalan numbers, there are two cases.

\emph{Case 1.}  Assume $E$ connects $v_1$ and $v_j$ ($j \neq 1$). 

Contracting $E$ replaces the two vertices $v_1$ and $v_j$ with one vertex, call it $v$. Since there were $n_1$ possible colors which could be assigned to $v_1$, and $n_j$ possible colors which could be assigned to $v_j$, we see that there are $n_1 + n_j - 2$ possible colors which can be assigned to this resulting vertex $v$.  

For notational convenience, assume that the colors on vertex $v_1$ are $(k_{1}^{(1)}, k_{a_2}^{(1)}, \dots k_{a_{\mu_1}}^{(1)})$ and the colors on vertex $v_j$ are $(k_{b_1}^{(j)}, k_{b_2}^{(j)}, \dots k_{b_{\mu_j}}^{(j)})$, with the color on the other half-edge of $E$ being $k_{b_{\ell}}^{(j)}$. We will make the convention that all colors on vertex $v_j$ have higher index than those on vertex $v_1$, so that our list of colors on the vertex $v$ is, in order,
$$(k_{a_2}^{(1)}, \dots k_{a_{\mu_1}}^{(1)}, k_{b_1}^{(j)}, \dots, \widehat{k_{b_{\ell}}^{(j)}}, \dots k_{b_{\mu_j}}^{(j)}).$$ 

Then, we color the half-edges incident on the new vertex $v$ starting with the lowest indexed color $k_{a_2}^{(1)}$ on the half-edge that was next to $E$ with respect to the counterclockwise cyclic ordering coming from the orientation on the surface, and continue by going counterclockwise around the vertex $v$. Observe that this will not change the coloring on the half-edges that were originally assigned to $v_1$. (See Figure~\ref{fig:bcmotzkinproof1}.)

\begin{figure}[hbt]
\includegraphics[height=1.25in]{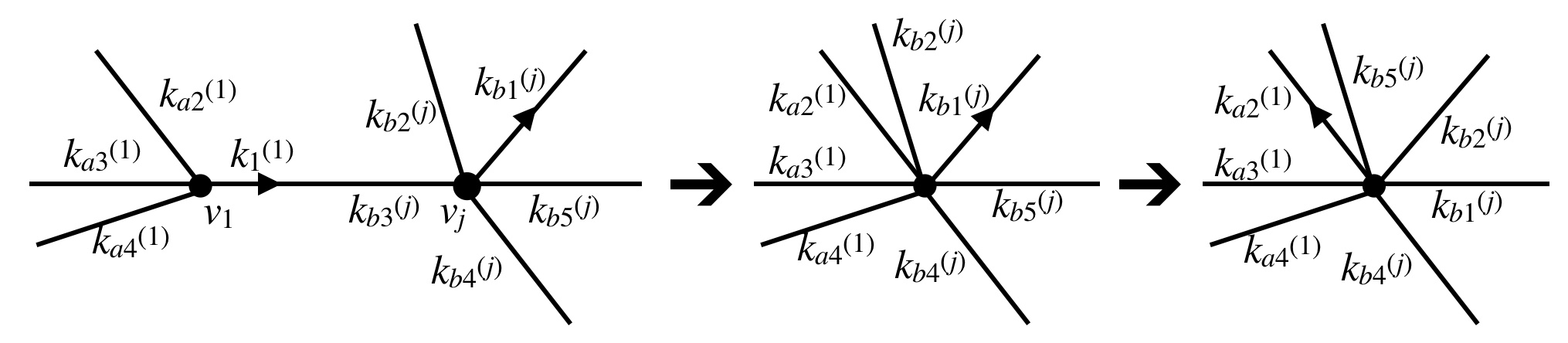}
\caption{}
\label{fig:bcmotzkinproof1}
\end{figure}

Since there are $n_j$ choices for the color $k_{b_{\ell}}^{(j)}$, we see that there are $n_j$ different graphs which produce the same result, which gives the factor of $n_j$ in the sum.

Further, since the weight $c$ counts the degree of the underlying Catalan graph, and the graph resulting from contracting edge $E$ has degree two less than the original graph, we thus have a factor of $c^2$ in front of the resulting term.

This gives the first term in the above formula,
$$c^2 \sum_{j=2}^v n_j \, \widetilde{M}_{g,v-1} (n_1+n_j-2, n_2, \dots, \hat{n_j}, \dots, n_v; b; c).$$

\emph{Case 2.} Assume $E$ is a loop on $v_1$. 

Just as in the proof of the analogous result for generalized Catalan numbers, contracting $E$ separates the incident half-edges at $v_1$ into two collections, with $\alpha$ edges on one side and $\beta$ edges on the other (note that $\alpha$ or $\beta$ may be zero). Since $E$ is a loop, contracting it causes pinching on the surface and produces a double point. We separate the double point into two new vertices. The result may be a surface of genus $g-1$, or two surfaces, of genus $g_1$ and $g_2$ with $g_1+g_2 = g$. Note that we do not need to re-assign the colorings in this case.

Observe that, since we contracted an edge which was incident twice on $v_1$, there are two less colors with which we can color the resulting vertices. Further, the remaining colors must now be shared between two vertices. (See Figure~\ref{fig:bcmotzkinproof2}.)

\begin{figure}[hbt]
\includegraphics[height=1.25in]{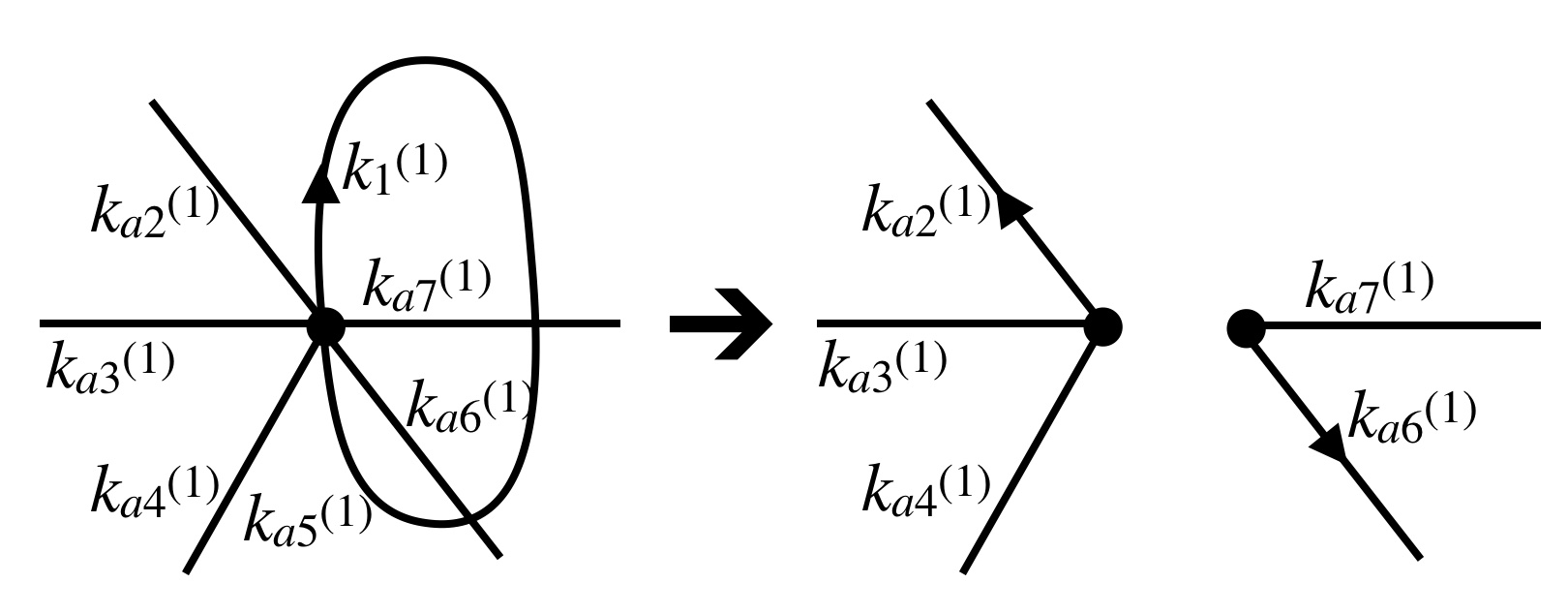}
\caption{}
\label{fig:bcmotzkinproof2}
\end{figure}

As in the previous case, since the weight $c$ counts the degree of the underlying Catalan graph, and the sum of the degrees of the graphs resulting from contracting edge $E$ is two less than the degree of the original graph, we are left with a factor of $c^2$ in front of the resulting terms.

This gives the remaining terms in the above formula,
\begin{multline*}
c^2 \sum_{\zeta+\xi=n_1-2} \bigg[ \widetilde{M}_{g-1,v+1}(\zeta,\xi,n_2,\dots,n_v;b,c) \\
+ \sum_{g_1+g_2=g, \; I \sqcup J = \{2,\dots,v\}} \widetilde{M}_{g_1,|I|+1}(\zeta,n_I;b,c)\widetilde{M}_{g_2,|J|+1}(\xi,n_J;b,c) \bigg]
\end{multline*}
\end{proof}


\section{Differential Recursion for Generalized $bc$-Motzkin Numbers}
\label{sect:diffrecformula}

We now wish to use the formula in Theorem~\ref{thm:genbcmrec} to obtain a topological recursion. In analogy with the approach in \cite{dumitrescumulaselec}, we first define the discrete Laplace transform of the generalized $bc$-Motzkin numbers, from which we will then obtain a differential recursion formula.

\begin{Def}
\label{def:bcmdlt}
We define the discrete Laplace Transform of the generalized $bc$-Motzkin numbers by 
\begin{equation}
F_{g,v}^{\widetilde{M}(b,c)}(x_1,\dots,x_v) = 
\begin{cases}
\displaystyle \sum_{n=1}^{\infty} \frac{\widetilde{M}_{0,1}(n;b,c)}{n} x_1^{-n} - \widetilde{M}_{0,1}(0;b,c) \log x_1 \text{ if $(g,v) = (0,1)$}, \\
\displaystyle \sum_{n_1=1}^{\infty} \cdots \sum_{n_v=1}^{\infty} \frac{\widetilde{M}_{g,v}(n_1,n_2,\dots,n_v;b,c)}{n_1 n_2 \cdots n_v} x_1^{-n_1} \cdots x_v^{-n_v}. \\
\end{cases}
\end{equation}
\end{Def}

We would like to obtain a recursion on 
$$\frac{\partial}{\partial t_1} F_{g,v}^{\widetilde{M}(b,c)}(t_1,t_2,\dots,t_v),$$ 
for a particular choice of change of variable $t_i$, similar to the Catalan differential recursion given in Proposition~\ref{prop:dmss13diff}.

We will first study the case when $(g,v) = (0,1)$. We differentiate the definition of $F_{0,1}^{\widetilde{M}(b,c)}(x)$ with respect to $x = x_1$ and apply the recursion formula for $(0,1)$-$bc$-Motzkin numbers from Proposition~\ref{prop:bcM01rec} to obtain
\begin{align*}
\frac{\partial}{\partial x} F_{0,1}^{\widetilde{M}(b,c)}(x) & = - \sum_{n=1}^{\infty} \widetilde{M}_{0,1}(n;b,c) x_1^{-n-1} - \widetilde{M}_{0,1}(0;b,c) x^{-1} \\
& = - x^{-1} + b x^{-1} \bigg( \frac{\partial}{\partial x} F_{0,1}^{\widetilde{M}(b,c)}(x) \bigg) - c^2 x^{-1} \bigg( \frac{\partial}{\partial x} F_{0,1}^{\widetilde{M}(b,c)}(x) \bigg)^2. \\
\end{align*}

Solving for $\partial/\partial x \, F_{0,1}^{\widetilde{M}(b,c)}(x)$ and applying the quadratic formula then implies
$$\frac{\partial}{\partial x} F_{0,1}^{\widetilde{M}(b,c)}(x) = \frac{1}{2c} \bigg[ - \bigg( \frac{x-b}{c} \bigg) - \sqrt{ \bigg( \frac{x-b}{c} \bigg)^2 - 4} \; \bigg].$$

Thus, we want to make a particular choice of change of variable so that the quantity under the square root is a perfect square, and we can simplify the expression for $\partial/\partial x \, F_{0,1}^{\widetilde{M}(b,c)}(x)$ further. We also want to choose the change of variable in such a way that it reduces to the change of variables for the Catalan case, which was given in equation \eqref{eqn:catvarchange}, when $b = 0$ and $c = 1$.

With this in mind, we let
$$\frac{x - b}{c} = 2 + \frac{4}{t^2-1} = 2\frac{t^2+1}{t^2-1}.$$

Then, after applying this change of variable, 
\begin{equation}
\label{eqn:bcmdr01}
\frac{\partial}{\partial t} F_{0,1}^{\widetilde{M}(b,c)}(t) = \frac{8t}{(t+1)(t-1)^3}.
\end{equation}

Observe this is precisely the same formula in terms of $t$ as was obtained for $\partial/\partial t F_{0,1}^C(t)$ in the Catalan case, and, when $b = 0$ and $c = 1$, the above change of variable formula does indeed reduce to the change of variable formula \eqref{eqn:catvarchange} used in the Catalan case. 

More generally, let $t_i = t_i(x_i,b,c)$, for $i \in \{1,2,\dots,v\}$, be defined by 
\begin{equation}
\label{eqn:bcmxtsub}
\frac{x_i - b}{c} = 2 + \frac{4}{t_i^2-1}.
\end{equation}

Then, we are led to prove the following new result.

\begin{thm}
\label{thm:bcmdiffrec}
The discrete Laplace transform $F^{\widetilde{M}(b,c)}_{g,v}(t_1,t_2,\dots,t_v)$ satisfies the following differential recursion formula, for every $(g,v) \neq (0,1),(0,2)$:
\begin{multline}
\frac{\partial}{\partial t_1} F_{g,v}^{\widetilde{M}(b,c)}(t_1,t_2,\dots,t_v) = - \frac{1}{16} \sum_{j=2}^v \bigg[ \frac{t_j}{t_1^2 - t_j^2} \bigg( \frac{(t_1^2-1)^3}{t_1^2} \frac{\partial}{\partial t_1} F_{g,v-1}^{\widetilde{M}(b,c)}(t_1,\dots,\widehat{t_j},\dots,t_v)\\
\begin{aligned}
& - \frac{(t_j^2-1)^3}{t_j^2} \frac{\partial}{\partial t_j} F_{g,v-1}^{\widetilde{M}(b,c)}(t_2,\dots,t_v) \bigg) \bigg] \\
& - \frac{1}{16} \sum_{j=2}^v \frac{(t_1^2-1)^2}{t_1^2} \bigg[ \frac{\partial}{\partial t_1} F_{g,v-1}^{\widetilde{M}(b,c)}(t_1,\dots,\widehat{t_j},\dots,t_v) \bigg] \\
& - \frac{1}{32} \frac{(t_1^2-1)^3}{t_1^2} \frac{\partial}{\partial u_1} \frac{\partial}{\partial u_2} F_{g-1,v+1}^{\widetilde{M}(b,c)}(u_1,u_2,t_2,\dots,t_v) \bigg|_{u_1=u_2=t_1} \\
& - \frac{1}{32} \frac{(t_1^2-1)^3}{t_1^2} \sum_{g_1+g_2=g, \; I \sqcup J = \{2,\dots,v\}, \; \text{stable}} \frac{\partial}{\partial t_1} F_{g_1,|I|+1}^{\widetilde{M}(b,c)}(t_1,t_I) \cdot \frac{\partial}{\partial t_1} F_{g_2,|J|+1}^{\widetilde{M}(b,c)}(t_1,t_J) \\
\end{aligned}
\end{multline}
where the ``stable'' summation means $2g_1 + |I| - 1 > 0$ and $2g_2 + |J| - 1 > 0$.

And,
\begin{equation}
\label{eqn:bcmdr02}
\frac{\partial}{\partial t_1} F_{0,2}^{\widetilde{M}(b,c)}(t_1,t_2) = \frac{(t_2+1)}{(t_1-1)(t_1+t_2)}.
\end{equation}

\end{thm}

\begin{rem}
With this choice of change of variable, the formula in Theorem~\ref{thm:bcmdiffrec} has no dependence on $b$ and $c$. The dependence on $b$ and $c$ only appears in the definition of $t_i(x_i,b,c)$. Further, this result has precisely the same form as the Catalan differential recursion formula in Proposition~\ref{prop:dmss13diff}, and the choice of $t_i$ in that case corresponds to $t_i(x_i,0,1)$, which is as expected, since the $bc$-Motzkin numbers reduce to the Catalan numbers in the case when $b = 0$ and $c = 1$.
\end{rem}

Below is the proof of this result. Detailed computations are provided in Appendix~\ref{subsect:appendixdiffrec}.

\begin{proof} 

We first differentiate the formula for $F_{g,v}^{\widetilde{M}(b,c)}(x_1,x_2,\dots,x_v)$ with respect to $x_1$, then plug in the $bc$-Motzkin ``recursion'' formula from Theorem~\ref{thm:genbcmrec}. 

Since that formula has four terms, for notational convenience we may write
\begin{multline}
\label{eqn:recursionfourterms}
\frac{\partial}{\partial x_1} F_{g,v}^{\widetilde{M}(b,c)}(x_1,x_2,\dots,x_v) = -\sum_{n_1=1}^{\infty} \cdots \sum_{n_v=1}^{\infty} \frac{\widetilde{M}_{g,v}(n_1,n_2,\dots,n_v;b,c)}{n_2 \cdots n_v} x_1^{-n_1-1} x_2^{-n_2} \cdots x_v^{-n_v} \\
= \text{I}_{g,v}(x_1,x_2,\dots,x_v) + \text{II}_{g,v}(x_1,x_2,\dots,x_v) \\
+ \text{III}_{g,v}(x_1,x_2,\dots,x_v) + \text{IV}_{g,v}(x_1,x_2,\dots,x_v).
\end{multline}

Then, for the first term, we may compute
\begin{align*}
\text{I}_{g,v}(x_1,x_2,\dots,x_v) & = -\sum_{n_1=1}^{\infty} \cdots \sum_{n_v=1}^{\infty} \bigg[ b \widetilde{M}_{g,v}(n_1-1,n_2,\dots,n_v;b,c) \bigg] \frac{x_1^{-n_1-1} x_2^{-n_2} \cdots x_v^{-n_v}}{n_2 \cdots n_v} \\
& = bx_1^{-1} \frac{\partial}{\partial x_1} F_{g,v}^{\widetilde{M}(b,c)}(x_1,x_2,\dots,x_v). \\
\end{align*}

For the second term, 
\begin{align*}
& \quad \text{II}_{g,v}(x_1,x_2,\dots,x_v) \\
& = -\sum_{n_1=1}^{\infty} \cdots \sum_{n_v=1}^{\infty} \bigg[ c^2 \sum_{j=2}^v n_j \, \widetilde{M}_{g,v-1} (n_1+n_j-2, n_2, \dots, \widehat{n_j}, \dots, n_v; b; c) \bigg] \\
& \quad \cdot \frac{x_1^{-n_1-1} x_2^{-n_2} \cdots x_v^{-n_v}}{n_2 \cdots n_v} \\
& = -c^2 \sum_{j=2}^v \frac{1}{x_1(x_j-x_1)} \bigg[ -\frac{\partial}{\partial x_1} F_{g,v-1}^{\widetilde{M}(b,c)}(x_1,\dots,\widehat{x_j},\dots,x_v) + \frac{\partial}{\partial x_j} F_{g,v-1}^{\widetilde{M}(b,c)}(x_2,\dots,x_v) \bigg] \\
\end{align*}
where we used the fact that
$$\sum_{n_1+n_j=k+2, \, n_1,n_j>0} x_1^{-n_1-1} x_j^{-n_j} = \frac{x_1^{-k-1}-x_j^{-k-1}}{x_1(x_j-x_1)}.$$

For the third term, 
\begin{align*}
\text{III}_{g,v}(x_1,x_2,\dots,x_v) & = - \sum_{n_1=1}^{\infty} \cdots \sum_{n_v=1}^{\infty} \bigg[ c^2 \sum_{\zeta+\xi=n_1-2} \widetilde{M}_{g-1,v+1}(\zeta,\xi,n_2,\dots,n_v;b,c) \bigg] \\
& \quad \cdot \frac{x_1^{-n_1-1} x_2^{-n_2} \cdots x_v^{-n_v}}{n_2 \cdots n_v} \\
& = - c^2 x_1^{-1} \frac{\partial}{\partial u_1} \frac{\partial}{\partial u_2} F_{g-1,v+1}^{\widetilde{M}(b,c)}(u_1,u_2,x_2,\dots,x_v) \bigg|_{u_1=u_2=x_1}. \\
\end{align*}

For the fourth term, 
\begin{align*}
& \quad \text{IV}_{g,v}(x_1,x_2,\dots,x_v) \\
& =  -\sum_{n_1=1}^{\infty} \cdots \sum_{n_v=1}^{\infty} \bigg[ c^2 \sum_{\zeta+\xi=n_1-2} \sum_{g_1+g_2=g} \sum_{I \sqcup J = \{2,\dots,v\}} \widetilde{M}_{g_1,|I|+1}(\zeta,n_I;b,c)\widetilde{M}_{g_2,|J|+1}(\xi,n_J;b,c) \bigg] \\
& \qquad \cdot \frac{x_1^{-n_1-1} x_2^{-n_2} \cdots x_v^{-n_v}}{n_2 \cdots n_v} \\
& = -c^2 x_1^{-1} \sum_{g_1+g_2=g, \; I \sqcup J = \{2,\dots,v\}} \frac{\partial}{\partial x_1} F_{g_1,|I|+1}^{\widetilde{M}(b,c)}(x_1,x_I) \cdot \frac{\partial}{\partial x_1} F_{g_2,|J|+1}^{\widetilde{M}(b,c)}(x_1,x_J). \\
\end{align*}

Upon substituting these four terms back into the original equation \eqref{eqn:recursionfourterms}, we obtain
\begin{align*}
& \quad \frac{\partial}{\partial x_1} F_{g,v}^{\widetilde{M}(b,c)}(x_1,x_2,\dots,x_v) \\
& = bx_1^{-1} \frac{\partial}{\partial x_1} F_{g,v}^{\widetilde{M}(b,c)}(x_1,x_2,\dots,x_v) \\
& \quad - c^2 \sum_{j=2}^v \frac{1}{x_1(x_j-x_1)} \bigg[ -\frac{\partial}{\partial x_1} F_{g,v-1}^{\widetilde{M}(b,c)}(x_1,\dots,\widehat{x_j},\dots,x_v) + \frac{\partial}{\partial x_j} F_{g,v-1}^{\widetilde{M}(b,c)}(x_2,\dots,x_v) \bigg] \\
& \quad - c^2 x_1^{-1} \frac{\partial}{\partial u_1} \frac{\partial}{\partial u_2} F_{g-1,v+1}^{\widetilde{M}(b,c)}(u_1,u_2,x_2,\dots,x_v) \bigg|_{u_1=u_2=x_1} \\
& \quad - c^2 x_1^{-1} \sum_{g_1+g_2=g, \; I \sqcup J = \{2,\dots,v\}} \frac{\partial}{\partial x_1} F_{g_1,|I|+1}^{\widetilde{M}(b,c)}(x_1,x_I) \cdot \frac{\partial}{\partial x_1} F_{g_2,|J|+1}^{\widetilde{M}(b,c)}(x_1,x_J). \\
\end{align*}

We will first look at the particular case when $(g,v) = (0,2)$. Then, this becomes
\begin{align*}
\frac{\partial}{\partial x_1} F_{0,2}^{\widetilde{M}(b,c)}(x_1,x_2) & = bx_1^{-1} \frac{\partial}{\partial x_1} F_{0,2}^{\widetilde{M}(b,c)}(x_1,x_2) \\
& \quad - c^2 \frac{1}{x_1(x_2-x_1)} \bigg[ -\frac{\partial}{\partial x_1} F_{0,1}^{\widetilde{M}(b,c)}(x_1) + \frac{\partial}{\partial x_2} F_{0,1}^{\widetilde{M}(b,c)}(x_2) \bigg] \\
& \quad - 2 c^2 x_1^{-1} \frac{\partial}{\partial x_1} F_{0,1}^{\widetilde{M}(b,c)}(x_1) \cdot \frac{\partial}{\partial x_1} F_{0,2}^{\widetilde{M}(b,c)}(x_1,x_2). \\
\end{align*}

After changing variables to $t_i$ using equation \eqref{eqn:bcmxtsub} and solving for $\partial/\partial t_1 F_{0,2}^{\widetilde{M}(b,c)}(t_1,t_2)$, we then obtain
\begin{equation}
\frac{\partial}{\partial t_1} F_{0,2}^{\widetilde{M}(b,c)}(t_1,t_2) = \frac{(t_2+1)}{(t_1-1)(t_1+t_2)},
\end{equation}
as was claimed.

Now, returning to the general case, we may separate out the $(0,1)$ and $(0,2)$ terms from the sum in $\text{IV}_{g,v}(x_1,x_2,\dots,x_v)$ to obtain
\begin{align*}
& \quad \text{IV}_{g,v}(x_1,x_2,\dots,x_v) \\
& = -c^2 x_1^{-1} \bigg[ \sum_{g_1+g_2=g, \; I \sqcup J = \{2,\dots,v\}, \; \text{stable}} \frac{\partial}{\partial x_1} F_{g_1,|I|+1}^{\widetilde{M}(b,c)}(x_1,x_I) \cdot \frac{\partial}{\partial x_1} F_{g_2,|J|+1}^{\widetilde{M}(b,c)}(x_1,x_J) \\
& \quad + 2 \frac{\partial}{\partial x_1} F_{0,1}^{\widetilde{M}(b,c)}(x_1) \cdot \frac{\partial}{\partial x_1} F_{g,v}^{\widetilde{M}(b,c)}(x_1,x_2,\dots,x_v) \\
& \quad + 2 \sum_{j=2}^v \frac{\partial}{\partial x_1} F_{0,2}^{\widetilde{M}(b,c)}(x_1,x_j) \cdot \frac{\partial}{\partial x_1} F_{g,v-1}^{\widetilde{M}(b,c)}(x_1,x_2,\dots,\widehat{x_j},\dots,x_v) \bigg]. \\
\end{align*}

Substituting this and the other three terms into equation \eqref{eqn:recursionfourterms} again, and rearranging terms so that $\partial/\partial x_1 F_{g,v}^{\widetilde{M}(b,c)}(x_1,x_2,\dots,x_v)$ appears only on the left side, gives
\begin{align*}
& \quad \frac{\partial}{\partial x_1} F_{g,v}^{\widetilde{M}(b,c)}(x_1,x_2,\dots,x_v) \\
& = - \bigg( \frac{x_1-b}{c^2} + 2 \frac{\partial}{\partial x_1} F_{0,1}^{\widetilde{M}(b,c)}(x_1) \bigg) ^{-1} \bigg\{ \sum_{j=2}^v \frac{1}{(x_j-x_1)} \bigg[ \frac{\partial}{\partial x_j} F_{g,v-1}^{\widetilde{M}(b,c)}(x_2,\dots,x_v) \bigg] \\
& + \sum_{j=2}^v \bigg( 2 \frac{\partial}{\partial x_1} F_{0,2}^{\widetilde{M}(b,c)}(x_1,x_j) - \frac{1}{(x_j-x_1)} \bigg) \bigg[ \frac{\partial}{\partial x_1} F_{g,v-1}^{\widetilde{M}(b,c)}(x_1,\dots,\widehat{x_j},\dots,x_v) \bigg] \\
& + \frac{\partial}{\partial u_1} \frac{\partial}{\partial u_2} F_{g-1,v+1}^{\widetilde{M}(b,c)}(u_1,u_2,x_2,\dots,x_v) \bigg|_{u_1=u_2=x_1} \\
& + \sum_{g_1+g_2=g, \; I \sqcup J = \{2,\dots,v\}, \; \text{stable}} \frac{\partial}{\partial x_1} F_{g_1,|I|+1}^{\widetilde{M}(b,c)}(x_1,x_I) \cdot \frac{\partial}{\partial x_1} F_{g_2,|J|+1}^{\widetilde{M}(b,c)}(x_1,x_J) \bigg\}. \\
\end{align*}

Then, changing variables to $t_i$ using equation \eqref{eqn:bcmxtsub} and simplifying the result completes the proof of the above theorem. 
\end{proof}

From this differential recursion formula, we can obtain all $F^{\widetilde{M}(b,c)}_{g,v}(t_1,\dots,t_v)$, where $(g,v) \neq (0,1),(0,2)$, by integrating the right side of the equation in Theorem~\ref{thm:bcmdiffrec} from $-1$ to $t_1$ with respect to the variable $t_1$. We will now look at two examples.

\begin{ex}
For the case when $(g,v) = (1,1)$, we see that
\begin{align*}
\frac{\partial}{\partial t} F_{1,1}^{\widetilde{M}(b,c)}(t) & = - \frac{1}{32} \frac{(t^2-1)^3}{t^2} \frac{\partial}{\partial u_1} \frac{\partial}{\partial u_2} F_{0,2}^{\widetilde{M}(b,c)}(u_1,u_2) \bigg|_{u_1=u_2=t} \\
& = - \frac{1}{128} \frac{(t^2-1)^3}{t^4}, \\
\end{align*}
so
\begin{equation}
\begin{aligned}
F_{1,1}^{\widetilde{M}(b,c)}(t) & = - \frac{1}{128} \int_{-1}^t \frac{(\tau^2-1)^3}{\tau^4} \; d\tau \\
& = - \frac{1}{384} \frac{(1+t)^4}{t^2} \bigg( t - 4 + \frac{1}{t} \bigg). \\
\end{aligned}
\end{equation}
\end{ex}

\begin{ex}
When $(g,v) = (0,3)$, we see that
\begin{equation}
F_{0,3}^{\widetilde{M}(b,c)}(t_1,t_2,t_3) = -\frac{1}{16} (t_1+1) (t_2+1) (t_3+1) \bigg( 1 + \frac{1}{t_1t_2t_3} \bigg).
\end{equation}
\end{ex}

\begin{rem}
Since $F_{1,1}^{\widetilde{M}(b,c)}(t)$ and $F_{0,3}^{\widetilde{M}(b,c)}(t_1,t_2,t_3)$ are both Laurent polynomials, the differential recursion formula in Theorem~\ref{thm:bcmdiffrec} implies that all higher order discrete Laplace transforms $F^{\widetilde{M}(b,c)}_{g,v}(t_1,t_2,\dots,t_v)$ will also be Laurent polynomials. This can be seen by observing that $F^{\widetilde{M}(b,c)}_{g,v}$, for $(g,v) \neq (0,1), (0,2), (0,3),$ and $(1,1)$ depends only on $F_{1,1}^{\widetilde{M}(b,c)}, F_{0,3}^{\widetilde{M}(b,c)},$ and higher order $F^{\widetilde{M}(b,c)}_{g,v}$, which, recursively, must be Laurent polynomials, because $F_{1,1}^{\widetilde{M}(b,c)}$ and $F_{0,3}^{\widetilde{M}(b,c)}$ are Laurent polynomials.
\end{rem}

It follows from Theorem~\ref{prop:dmss13diff} and Theorem~\ref{thm:bcmdiffrec} that, since the differential recursion formulas for the $bc$-Motzkin numbers and the Catalan numbers are identical (up to the change of variable) and have the same initial conditions, their discrete Laplace transforms must be the same. This gives the following Corollary. 

\begin{cor}
We have
\begin{equation}
F^{\widetilde{M}(b,c)}_{g,v}(t_1,t_2,\dots,t_v) = F^C_{g,v}(t_1,t_2,\dots,t_v),
\end{equation}
where on the left side the $t_i$ are defined by equation \eqref{eqn:bcmxtsub} and on the right side the $t_i$ are defined by equation \eqref{eqn:catvarchange}.
\end{cor}

Therefore, we see that the following results which are known to hold for generalized Catalan numbers are also true for the case of generalized $bc$-Motzkin numbers. (See also \cite{CMS}, \cite{dumitrescumulaselec}, \cite{DMSS}, and \cite{mulasepenkava}.)

For all $(g,v)$ with $2g-2+v > 0$, the discrete Laplace transform $F_{g,v}^{\widetilde{M}(b,c)}(t_1,t_2,\dots,t_v)$ satisfies the following corollaries of Theorem~\ref{thm:bcmdiffrec}, where the $t_i$ are as defined above.

\begin{cor}
$F_{g,v}^{\widetilde{M}(b,c)}(t_1,\dots,t_v)$ is a Laurent polynomial in the $t_i$-variables, of degree $3(2g-2+v)$. And, 
$$F_{g,v}^{\widetilde{M}(b,c)}(1/t_1,1/t_2,\dots,1/t_v) = F_{g,v}^{\widetilde{M}(b,c)}(t_1,t_2,\dots,t_v).$$
\end{cor}

\begin{cor}
The special values at $t_i = -1$ are given by 
$$F_{g,v}^{\widetilde{M}(b,c)}(t_1,t_2,\dots,t_v) |_{t_i=-1} = 0,$$
for each $i$.

The diagonal value at $t_i = 1$ gives the orbifold Euler characteristic of the moduli space $\mathcal{M}_{g,n}$,
$$F_{g,v}^{\widetilde{M}(b,c)}(1,1,\dots,1) = (-1)^n \chi(\mathcal{M}_{g,n}).$$
\end{cor}

\begin{cor}
The restriction of the Laurent polynomial $F_{g,v}^{\widetilde{M}(b,c)}(t_1,\dots,t_v)$ to its highest degree terms gives a homogeneous polynomial defined by 
$$F_{g,v}^{\widetilde{M}(b,c), \text{highest}}(t_1,\dots,t_v) = \frac{(-1)^v}{2^{2g-2+v}} \sum_{d_1+\cdots+d_v=3g-3+v} \langle \tau_{d_1} \cdots \tau_{d_v} \rangle_{g,n} \prod_{i=1}^v |2d_i - 1|!! \bigg( \frac{t_i}{2} \bigg)^{2d_i+1},$$
where the $\langle \tau_{d_1} \cdots \tau_{d_v} \rangle_{g,n}$ represent the intersection numbers on the moduli space of stable curves.
\end{cor}


\section{Topological Recursion for Generalized $bc$-Motzkin Numbers}
\label{sect:toprecformula}

Just as for the Catalan case, as discussed in Section~\ref{sect:catbackground} of this paper and in \cite{dumitrescumulaselec}, the differential recursion formula for generalized $bc$-Motzkin numbers in Theorem~\ref{thm:bcmdiffrec} leads directly to the following topological recursion for generalized $bc$-Motzkin numbers. Note that, since the $bc$-Motzkin differential recursion formula has the same form, with the same initial conditions, as the differential recursion formula for Catalan numbers which was given in Proposition~\ref{prop:dmss13diff} (up to the slightly different change of variables from $x_i$ to $t_i$), the topological recursion for generalized $bc$-Motzkin numbers and its proof are identical to those in the Catalan case, as given in Proposition~\ref{prop:dmss13top}.

\begin{rem}
Just as with the differential recursion formula, the dependence on $b$ and $c$ only appears in this change of variable from the $x_i$ to the $t_i$. This tells us that there does indeed exist a topological recursion for these generalized $bc$-Motzkin numbers, and furthermore it has precisely the same form as the result which was obtained in \cite{dumitrescumulaselec} for the generalized Catalan numbers.
\end{rem}

Thus, we have the following new theorem for the generalized $bc$-Motzkin numbers.

\begin{thm}
\label{thm:bcmtoprec}
Define symmetric $v$-linear differential forms on $(\mathbb{P}^1)^v$ for $2g-2+v > 0$ by
\begin{equation}
W_{g,v}^{\widetilde{M}(b,c)}(t_1,t_2,\dots,t_v) = d_{t_1} \cdots d_{t_v} F_{g,v}^{\widetilde{M}(b,c)}(t_1,t_2,\dots,t_v),
\end{equation}
and for $(g,v) = (0,2)$ by
\begin{equation}
W_{0,2}^{\widetilde{M}(b,c)}(t_1,t_2) = \frac{dt_1 \, dt_2}{(t_1-t_2)^2}.
\end{equation}
Then, these differential forms satisfy the following integral recursion equation:
\begin{multline}
\label{eqn:bcmtoprec}
W_{g,v}^{\widetilde{M}(b,c)}(t_1,t_2,\dots,t_v) = - \frac{1}{64} \frac{1}{2\pi i} \int_{\gamma} \bigg( \frac{1}{t+t_1} + \frac{1}{t-t_1} \bigg) \frac{(t^2-1)^3}{t^2} \,\frac{1}{dt} \, dt_1 \\
\cdot \bigg[ \sum_{j=2}^v \bigg( W_{0,2}^{\widetilde{M}(b,c)}(t,t_j) W_{g,v-1}^{\widetilde{M}(b,c)}(-t,t_2,\dots,\widehat{t_j},\dots,t_v) \\ + W_{0,2}^{\widetilde{M}(b,c)}(-t,t_j) W_{g,v-1}^{\widetilde{M}(b,c)}(t,t_2,\dots,\widehat{t_j},\dots,t_v) \bigg) \\
+ W_{g-1,v+1}^{\widetilde{M}(b,c)}(t,-t,t_2,\dots,t_v) \\
+ \sum_{g_1+g_2=g, \; I \sqcup J = \{2,\dots,v\}, \; \text{stable}} W_{g_1,|I|+1}^{\widetilde{M}(b,c)}(t,t_I) W_{g_2,|J|+1}^{\widetilde{M}(b,c)}(-t,t_J) \bigg]
\end{multline}
The ``stable'' summation means $2g_1 + |I| - 1 > 0$ and $2g_2 + |J| - 1 > 0$.

The curve $\gamma$ is as given in Figure~\ref{fig:gammacontour}.

\end{thm}

As in the Catalan case, given in Proposition~\ref{prop:dmss13top}, these differential forms are called the Eynard-Orantin differential forms, and the recursion is called the topological recursion for the generalized $bc$-Motzkin numbers. 

\begin{rem} 
Observe that, using Equation~\ref{eqn:bcmdr02} above for $\partial/\partial t_1 F_{0,2}^{\widetilde{M}(b,c)} (t_1,t_2)$ gives
\begin{equation}
\frac{\partial}{\partial t_1} \frac{\partial}{\partial t_2} F_{0,2}^{\widetilde{M}(b,c)} (t_1,t_2) \; dt_1 \, dt_2 = \frac{dt_1 \, dt_2}{(t_1+t_2)^2} = \frac{dt_1 \, dt_2}{(t_1-t_2)^2} - (\tilde{\pi} \times \tilde{\pi})^* \frac{dx_1 \, dx_2}{(x_1-x_2)^2},
\end{equation}
where $\tilde{\pi}: \mathbb{P}^1 \to \mathbb{P}^1$ is the variable transformation 
\begin{equation}
x_i = 2c \bigg( \frac{t_i^2+1}{t_i^2-1} \bigg) + b
\end{equation}
which was defined in \eqref{eqn:bcmxtsub}.
\end{rem}

Below is the proof of Theorem~\ref{thm:bcmtoprec}. Detailed computations are provided in Appendix~\ref{subsect:appendixtoprec}.

\begin{proof}

For notational convenience, we define the functions $w_{g,v}^{\widetilde{M}(b,c)}$ by
\begin{equation}
W_{g,v}^{\widetilde{M}(b,c)}(t_1,t_2,\dots,t_v) = w_{g,v}^{\widetilde{M}(b,c)}(t_1,t_2,\dots,t_v) \; dt_1 \, dt_2 \cdots dt_v.
\end{equation}

We will use the following observations:

\begin{enumerate}
\item When $(g,v)$ is stable, the $w_{g,v}^{\widetilde{M}(b,c)}(t_1,t_2,\dots,t_v)$ are symmetric in the variables $t_i$. And, they are Laurent polynomials, so the only singularities can be when one of the $t_i$ is zero.
\item Further, $W_{g,v}^{\widetilde{M}(b,c)}$ is an odd differential form, so $w_{g,v}^{\widetilde{M}(b,c)}(t_1,t_2,\dots,t_v)$ is an even function. 
\end{enumerate}

We first apply the definition of the Eynard-Orantin differential forms to the differential recursion formula in Theorem~\ref{thm:bcmdiffrec}. Since there are four terms in this formula, we may write
\begin{multline}
W_{g,v}^{\widetilde{M}(b,c)}(t_1,t_2,\dots,t_v) = \frac{\partial}{\partial t_2} \cdots \frac{\partial}{\partial t_v} \bigg[ \frac{\partial}{\partial t_1} F_{g,v}^{\widetilde{M}(b,c)}(t_1,t_2,\dots,t_v) \bigg] dt_1 \, dt_2 \cdots dt_v \\
= \text{I}_{g,v}(t_1,t_2,\dots,t_v) + \text{II}_{g,v}(t_1,t_2,\dots,t_v) + \text{III}_{g,v}(t_1,t_2,\dots,t_v) + \text{IV}_{g,v}(t_1,t_2,\dots,t_v).
\end{multline}

Then, we will show that the result is equal to the formula given in equation \eqref{eqn:bcmtoprec} of Theorem~\ref{thm:bcmtoprec}.

We see that
\begin{multline*}
\text{I}_{g,v}(t_1,t_2,\dots,t_v) = \frac{1}{16} \sum_{j=2}^v \bigg[ \frac{\partial}{\partial t_j} \bigg( \frac{(t_j^2-1)^3}{t_j(t_1^2-t_j^2)} W_{g,v-1}^{\widetilde{M}(b,c)}(t_2,\dots,t_v) \bigg) \; dt_1 \\
- \frac{(t_1^2+t_j^2)(t_1^2-1)^3}{t_1^2(t_1^2-t_j^2)^2} W_{g,v-1}^{\widetilde{M}(b,c)}(t_1,t_2,\dots,\widehat{t_j},\dots,t_v) \; dt_j \bigg].
\end{multline*}

And, we may compute 
\begin{multline*}
\begin{aligned}
& - \frac{1}{64} \frac{1}{2\pi i} \int_{\gamma} \bigg( \frac{1}{t+t_1} + \frac{1}{t-t_1} \bigg) \frac{(t^2-1)^3}{t^2} \, \frac{1}{dt} \, dt_1 \\
& \quad \times \sum_{j=2}^v \bigg( W_{0,2}^{\widetilde{M}(b,c)}(t,t_j) W_{g,v-1}^{\widetilde{M}(b,c)}(-t,t_2,\dots,\widehat{t_j},\dots,t_v) \\
& \quad \quad \quad + W_{0,2}^{\widetilde{M}(b,c)}(-t,t_j) W_{g,v-1}^{\widetilde{M}(b,c)}(t,t_2,\dots,\widehat{t_j},\dots,t_v) \bigg) \\
\end{aligned} \\
= \frac{1}{64} \sum_{j=2}^v \bigg[ \frac{1}{2\pi i} \int_{\gamma} \bigg( \frac{2(t^2-1)^3/t}{(t+t_1)(t-t_1)} \bigg)\bigg( \frac{1}{(t-t_j)^2} + \frac{1}{(t+t_j)^2}  \bigg) \\
\cdot w_{g,v-1}^{\widetilde{M}(b,c)}(t,t_2,\dots,\widehat{t_j},\dots,t_v) \; dt \bigg] dt_1 \, dt_2 \cdots dt_v.
\end{multline*}

Applying the Cauchy Residue Theorem to evaluate this integral around the contour $\gamma$ given in Figure~\ref{fig:gammacontour}, and simplifying the result, shows that these two formulas are indeed equal. 

The second term in the differential recursion formula of Theorem~\ref{thm:bcmdiffrec} becomes zero when we differentiate with respect to $t_j$, so 
$$\text{II}_{g,v}(t_1,t_2,\dots,t_v) = 0.$$

Further, we may compute from the third term in the differential recursion formula that
$$\text{III}_{g,v}(t_1,t_2,\dots,t_v) = - \frac{1}{32} \frac{(t_1^2-1)^3}{t_1^2} W_{g-1,v+1}^{\widetilde{M}(b,c)}(t_1,t_1,t_2,\dots,t_v) \frac{1}{dt_1}.$$

And, we have
\begin{multline*}
- \frac{1}{64} \frac{1}{2\pi i} \int_{\gamma} \bigg( \frac{1}{t+t_1} + \frac{1}{t-t_1} \bigg) \frac{(t^2-1)^3}{t^2} \, \frac{1}{dt} \, dt_1 \times W_{g-1,v+1}^{\widetilde{M}(b,c)}(t,-t,t_2,\dots,t_v) \\
= \frac{1}{64} \bigg[ \frac{1}{2\pi i} \int_{\gamma} \bigg( \frac{1}{t+t_1} + \frac{1}{t-t_1} \bigg) \frac{(t^2-1)^3}{t^2} \; dt \bigg] \times w_{g-1,v+1}^{\widetilde{M}(b,c)}(t_1,t_1,t_2,\dots,t_v) \; dt_1 \, dt_2 \cdots dt_v.
\end{multline*}

Again applying the Cauchy Residue Theorem to this integral around $\gamma$ and simplifying the result shows that these two formulas are equal. 

Finally, from the fourth term in the differential recursion formula, we have
\begin{multline*}
\text{IV}_{g,v}(t_1,t_2,\dots,t_v) = \\
- \frac{1}{32} \frac{(t_1^2-1)^3}{t_1^2} \sum_{g_1+g_2=g, \; I \sqcup J = \{2,\dots,v\}, \; \text{stable}} W_{g_1,|I|+1}^{\widetilde{M}(b,c)}(t_1,t_I) \, W_{g_2,|J|+1}^{\widetilde{M}(b,c)}(t_1,t_J) \, \frac{1}{dt_1}. 
\end{multline*}

And, 
\begin{multline*}
\begin{aligned}
& - \frac{1}{64} \frac{1}{2\pi i} \int_{\gamma} \bigg( \frac{1}{t+t_1} + \frac{1}{t-t_1} \bigg) \frac{(t^2-1)^3}{t^2} \, \frac{1}{dt} \, dt_1 \\
& \times \bigg[ \sum_{g_1+g_2=g, \; I \sqcup J = \{2,\dots,v\}, \; \text{stable}} \frac{\partial}{\partial t_1} W_{g_1,|I|+1}^{\widetilde{M}(b,c)}(t,t_I) W_{g_2,|J|+1}^{\widetilde{M}(b,c)}(-t,t_J) \bigg] \\
\end{aligned} \\
= \frac{1}{64} \bigg[ \frac{1}{2\pi i} \int_{\gamma} \bigg( \frac{1}{t+t_1} + \frac{1}{t-t_1} \bigg) \frac{(t^2-1)^3}{t^2} \; dt \bigg] \\
\times \sum_{g_1+g_2=g, \; I \sqcup J = \{2,\dots,v\}, \; \text{stable}} w_{g_1,|I|+1}^D(t_1,t_I) w_{g_2,|J|+1}^D(t_1,t_J) \; dt_1 \, dt_2 \cdots dt_v. 
\end{multline*}

Applying the Cauchy Residue Theorem yet again shows that these two formulas are indeed equal. 

This concludes the proof of Theorem~\ref{thm:bcmtoprec}.
\end{proof}

From this topological recursion in Theorem~\ref{thm:bcmtoprec}, and the initial case of $W_{0,2}^{\widetilde{M}(b,c)}(t_1,t_2)$ given in that theorem, we can recursively compute all of the Eynard-Orantin differential forms $W_{g,v}^{\widetilde{M}(b,c)}(t_1,t_2,\dots,t_v)$.

Let us now look at some examples.

\begin{ex}
When $(g,v) = (1,1)$, we have
\begin{equation}
\begin{aligned}
W_{1,1}^{\widetilde{M}(b,c)}(t_1) & = - \frac{1}{64} \frac{1}{2\pi i} \int_{\gamma} \bigg( \frac{1}{t+t_1} + \frac{1}{t-t_1} \bigg) \frac{(t^2-1)^3}{t^2} \, \frac{1}{dt} \, dt_1 \times W_{0,2}^{\widetilde{M}(b,c)}(t,-t) \\
& = \frac{1}{128} \frac{1}{2\pi i} \bigg[ \int_{\gamma} \frac{(t^2-1)^3/t^3}{(t+t_1)(t-t_1)} \; dt\bigg] \, dt_1 \\
& = \frac{1}{128} \frac{1}{2\pi i} \bigg[ - 2\pi i \text{Res}_{t=t_1} \frac{(t^2-1)^3/t^3}{(t+t_1)(t-t_1)} - 2\pi i \text{Res}_{t=-t_1} \frac{(t^2-1)^3/t^3}{(t+t_1)(t-t_1)} \bigg] \, dt_1 \\
& = - \frac{1}{128} \frac{(t_1^2-1)^3}{t_1^4} \; dt_1. \\
\end{aligned}
\end{equation}
\end{ex}

\begin{ex}
Similarly, when $(g,v) = (0,3)$, we may compute
\begin{equation}
W_{0,3}^{\widetilde{M}(b,c)}(t_1,t_2,t_3) = -\frac{1}{16} \bigg( \frac{1}{t_1t_2t_3} - 1\bigg).
\end{equation}
\end{ex}


\section{Appendices}
\label{sect:appendices}
\renewcommand{\thesubsection}{\Alph{subsection}}


\subsection{Proof of the ``Vandermonde-Like'' Identity}
\label{subsect:appendixvandermonde}

In this appendix, we prove the ``Vandermonde-like'' identity which was used in the proof of the recursion formula for generalized $bc$-Motzkin numbers.

\begin{prop}
\label{prop:vandermondelike}
For all $0 \leq i+j \leq n$, we have
$$\sum_{a+b=n} \binom{a}{i} \binom{b}{j} = \binom{n+1}{i+j+1}.$$
\end{prop}

\begin{proof}
We will prove this statement using induction on $n$.

\emph{Base case.} Assume $n = 0$. This implies that we must have $i=j=0$, so
$$\sum_{a+b=0} \binom{a}{0} \binom{b}{0} =  \binom{0}{0} \binom{0}{0} = 1 \cdot 1 = 1 = \binom{0+1}{0+0+1}.$$
Thus the claim holds for $n=0$.

\emph{Inductive case.} Assume we know 
$$\sum_{a+b=n} \binom{a}{i} \binom{b}{j} = \binom{n+1}{i+j+1}$$
is true for all $0 \leq i+j \leq n$.

Then, using Pascal's Identity and the inductive hypothesis, 

\begin{align*}
\binom{(n+1)+1}{i+j+1} & = \binom{n+1}{i+j} + \binom{n+1}{i+j+1} \\
& = \sum_{a+b=n} \binom{a}{i-1} \binom{b}{j} + \sum_{a+b=n} \binom{a}{i} \binom{b}{j} \\
& = \sum_{a+b=n} \bigg[ \binom{a}{i-1} + \binom{a}{i} \bigg] \binom{b}{j} \\
& = \sum_{a+b=n} \binom{a+1}{i} \binom{b}{j} \\
& = \sum_{a+b=n+1} \binom{a}{i} \binom{b}{j} \\
\end{align*}

which proves Proposition~\ref{prop:vandermondelike}.
\end{proof}


\subsection{Proof of the Recursion Formula for $(0,1)$-$bc$-Motzkin Numbers}
\label{subsect:appendix01proof}

Here, we provide a proof the recursion formula for $(0,1)$-$bc$-Motzkin numbers which was given in Proposition~\ref{prop:bcM01rec}.

\begin{proof}
Using equation \eqref{eqn:bcmotzkin01} as the definition of the $(0,1)$-$bc$-Motzkin numbers, we may compute
\begin{align*}
& \quad \widetilde{M}_{0,1}(n;b,c) - b \, \widetilde{M}_{0,1}(n-1;b,c) \\
& = \sum_{\mu=0}^{n} \binom{n}{\mu} \, C_{0,1}(\mu) \, b^{n-\mu}c^{\mu} - b \, \sum_{\mu=0}^{n-1} \binom{n-1}{\mu} \, C_{0,1}(\mu) \, b^{n-1-\mu}c^{\mu} \\
& = C_{0,1}(n) \, c^{n} + \sum_{\mu=0}^{n-1} \left[ \binom{n}{\mu} - \binom{n-1}{\mu} \right] C_{0,1}(\mu) \, b^{n-\mu}c^{\mu} \\
& = C_{0,1}(n) \, c^{n} + \sum_{\mu=0}^{n-1} \binom{n-1}{\mu-1} C_{0,1}(\mu) \, b^{n-\mu}c^{\mu} \\
& = \sum_{\mu=0}^{n} \binom{n-1}{\mu-1} C_{0,1}(\mu) \, b^{n-\mu}c^{\mu} \\
& = \sum_{\mu=0}^{n} \binom{n-1}{\mu-1} \left[\sum_{i+j=\mu-2} C_{0,1}(i) C_{0,1}(j) \right] \, b^{n-\mu}c^{\mu} \\
& = c^2 \sum_{\mu-2=0}^{n-2} \sum_{i+j=\mu-2} \binom{(n-2)+1}{(\mu-2)+1} C_{0,1}(i) C_{0,1}(j) b^{(n-2)-(\mu-2)} c^{\mu-2} \\
& = c^2 \sum_{\mu=0}^{n-2} \sum_{i+j=\mu} \binom{(n-2)+1}{i+j+1} C_{0,1}(i) C_{0,1}(j) b^{(n-2)-(i+j)} c^{i+j} \\
& = c^2 \sum_{\mu=0}^{n-2} \sum_{i+j=\mu} \left[ \sum_{\alpha+\beta=n-2} \binom{\alpha}{i} \binom{\beta}{j} \right] C_{0,1}(i) C_{0,1}(j) b^{(\alpha+\beta)-(i+j)} c^{i+j} \\
& = c^2 \sum_{\alpha+\beta=n-2} \left[ \sum_{i=0}^{\alpha} \binom{\alpha}{i} C_{0,1}(i) b^{\alpha-i} c^{i} \right] \left[ \sum_{j=0}^{\beta} \binom{\beta}{j} C_{0,1}(j) b^{\beta-j} c^{j} \right] \\
& = c^2 \sum_{\alpha+\beta=n-2} \widetilde{M}_{0,1}(\alpha,b,c) \widetilde{M}_{0,1}(\beta,b,c) \\
\end{align*}

Here, we recall from equation \eqref{eqn:catrec} that the Catalan numbers satisfy the recursion formula
$$C_{0,1}(\mu) = \sum_{i+j=\mu-2}C_{0,1}(i)C_{0,1}(j).$$
And, we also used the ``Vandermonde-like’’ identity 
\begin{equation}
\label{eqn:vandermondelike1}
\sum_{a+b=n} \binom{a}{i} \binom{b}{j} = \binom{n+1}{i+j+1},
\end{equation}
which is proved in Appendix~\ref{subsect:appendixvandermonde}.

This proves Proposition~\ref{prop:bcM01rec}.
\end{proof}


\subsection{Algebraic Proof of the Recursion Formula for Generalized $bc$-Motzkin Numbers}
\label{subsect:appendixrecursion}

Here, we provide a complete algebraic proof of the recursion formula for generalized $bc$-Motzkin numbers, which was given in Theorem~\ref{thm:genbcmrec}.

\begin{proof}

We first apply Definition~\ref{def:genbcmotzkin} of the generalized $bc$-Motzkin numbers, and Pascal's identity, to simplify the left side of the formula in Theorem~\ref{thm:genbcmrec}.
\begin{multline}
\widetilde{M}_{g,v}(n_1,n_2,\dots,n_v;b,c) - b \widetilde{M}_{g,v}(n_1-1,n_2,\dots,n_v;b,c) \\
= \sum_{\mu_1=0}^{n_1} \sum_{\mu_2=0}^{n_2} \cdots \sum_{\mu_v=0}^{n_v} \binom{n_1-1}{\mu_1-1} \binom{n_2}{\mu_2} \cdots \binom{n_v}{\mu_v} \\
\cdot C_{g,v}(\mu_1,\dots,\mu_v) \, b^{(n_1+\cdots+n_v)-(\mu_1+\cdots+\mu_v)}c^{\mu_1+\cdots+\mu_v} 
\end{multline}

Then, we plug in the Catalan ``recursion'' formula of Proposition~\ref{prop:gencatrec} to the above equation, simplify the resulting three terms, and re-write them in terms of generalized $bc$-Motzkin numbers. The key ingredients in this proof are Vandermonde’s identity,
\begin{equation}
\label{eqn:appvandermonde}
\sum_{i+j=k}  \binom{a}{i}  \binom{b}{j} =  \binom{a + b}{k},
\end{equation}
and a particular version of the ``Vandermonde-like'' identity proved in Appendix~\ref{subsect:appendixvandermonde},
\begin{equation}
\label{eqn:appvandermondelike}
\sum_{\zeta+\xi=k} \binom{\zeta-1}{\alpha-1} \binom{\xi}{\beta} = \binom{k}{\alpha+\beta}.
\end{equation}

For notational convenience, we let
\begin{multline}
\label{eqn:recthreeterms}
\widetilde{M}_{g,v}(n_1,\dots,n_v;b,c) - b \widetilde{M}_{g,v}(n_1-1,n_2,\dots,n_v;b,c) \\
= \text{I}_{g,v}(n_1,\dots,n_v;b,c) + {\text{II}}_{g,v}(n_1,\dots,n_v;b,c) + {\text{III}}_{g,v}(n_1,\dots,n_v;b,c), 
\end{multline}
where the three terms correspond to the three terms of the Catalan ``recursion'' formula.

To shorten notation, we will sometimes write 
$$\vec{n} = (n_1,\dots,n_v)$$ 
and 
$$|\vec{n}| = n_1 + \cdots + n_v,$$
and similarly for $\vec{\mu}$.

For the first term, we have
\begin{multline*}
\text{I}_{g,v}(\vec{n};b,c) = \sum_{j=2}^v \sum_{\mu_2=0}^{n_2} \cdots \sum_{\mu_{j-1}=0}^{n_{j-1}} \sum_{\mu_{j+1}=0}^{n_{j+1}} \cdots \sum_{\mu_v=0}^{n_v} \binom{n_2}{\mu_2} \cdots \binom{n_{j-1}}{\mu_{j-1}} \binom{n_{j+1}}{\mu_{j+1}} \cdots \binom{n_v}{\mu_v} \\
\cdot \bigg[ \sum_{\mu_1=0}^{n_1} \sum_{\mu_j=0}^{n_j} \binom{n_1-1}{\mu_1-1} \binom{n_j}{\mu_j} \mu_j C_{g,v-1} (\mu_1 + \mu_j - 2, \mu_2, \dots, \hat{\mu_j}, \dots, \mu_v) \, b^{|\vec{n}|-|\vec{\mu}|}c^{|\vec{\mu}|} \bigg].
\end{multline*}

Now, since
$$\binom{n_j}{\mu_j} \cdot \mu_j = \binom{n_j-1}{\mu_j-1} \cdot n_j,$$
the term in brackets becomes
\begin{align*}
& \quad \sum_{\mu_1=1}^{n_1} \sum_{\mu_j=1}^{n_j} \binom{n_1-1}{\mu_1-1} \binom{n_j-1}{\mu_j-1} n_j \, C_{g,v-1} (\mu_1 + \mu_j - 2, \mu_2, \dots, \hat{\mu_j}, \dots, \mu_v) \, b^{|\vec{n}|-|\vec{\mu}|}c^{|\vec{\mu}|} \\
& = \sum_{\mu_1=0}^{n_1-1} \sum_{\mu_j=0}^{n_j-1} \binom{n_1-1}{\mu_1} \binom{n_j-1}{\mu_j} n_j \, C_{g,v-1} (\mu_1 + \mu_j, \mu_2, \dots, \hat{\mu_j}, \dots, \mu_v) \, b^{|\vec{n}|-(2+|\vec{\mu}|)}c^{2+|\vec{\mu}|} \\
& = \sum_{k=0}^{n_1+n_j-2} \sum_{\mu_1+\mu_j=k} \binom{n_1-1}{\mu_1} \binom{n_j-1}{\mu_j} n_j \, C_{g,v-1} (k, \mu_2, \dots, \hat{\mu_j}, \dots, \mu_v) \\
& \quad \cdot b^{|\vec{n}|-(2+k+\mu_2+\cdots+\widehat{\mu_j}+\cdots+\mu_v)}c^{(2+k+\mu_2+\cdots+\widehat{\mu_j}+\cdots+\mu_v)} \\
& = \sum_{k=0}^{n_1+n_j-2} \binom{n_1+n_j-2}{k} n_j \, C_{g,v-1} (k, \mu_2, \dots, \hat{\mu_j}, \dots, \mu_v) \\
& \quad \cdot c^2 b^{((n_1+n_j-2)+n_2+\cdots+\widehat{n_j}+\cdots+n_v)-(k+\mu_2+\cdots+\widehat{\mu_j}+\cdots+\mu_v)}c^{(k+\mu_2+\cdots+\widehat{\mu_j}+\cdots+\mu_v)} \\
\end{align*}
where we have applied Vandermonde's identity \eqref{eqn:appvandermonde}.

Substituting this back into the above expression for $\text{I}_{g,v}(\vec{n};b,c)$ then gives
$$\text{I}_{g,v}(\vec{n};b,c) = c^2 \sum_{j=2}^v n_j \, \widetilde{M}_{g,v-1} (n_1+n_j-2, n_2, \dots, \hat{n_j}, \dots, n_v; b, c).$$

Now, for the second term, we may proceed as follows, using the Vandermonde-like identity of equation \eqref{eqn:appvandermondelike}.
\begin{align*}
{\text{II}}_{g,v}(\vec{n};b,c) & = \sum_{\mu_2=0}^{n_2} \cdots \sum_{\mu_v=0}^{n_v} \binom{n_2}{\mu_2} \cdots \binom{n_v}{\mu_v} \bigg[ \sum_{\mu_1=0}^{n_1} \sum_{\alpha + \beta = \mu_1 - 2} \binom{n_1-1}{\mu_1-1} \bigg] \\
& \quad \cdot C_{g-1,v+1}(\alpha,\beta,\mu_2,\dots,\mu_v) \, b^{|\vec{n}|-|\vec{\mu}|}c^{|\vec{\mu}|} \\
& = \sum_{\mu_2=0}^{n_2} \cdots \sum_{\mu_v=0}^{n_v} \binom{n_2}{\mu_2} \cdots \binom{n_v}{\mu_v} \bigg[ \sum_{\mu_1=0}^{n_1} \sum_{\alpha + \beta = \mu_1 - 2} \sum_{k=0}^{n_1-2} \binom{k}{\mu_1-2} \bigg] \\
& \quad \cdot C_{g-1,v+1}(\alpha,\beta,\mu_2,\dots,\mu_v) \, b^{|\vec{n}|-|\vec{\mu}|}c^{|\vec{\mu}|} \\
& = \sum_{\mu_2=0}^{n_2} \cdots \sum_{\mu_v=0}^{n_v} \binom{n_2}{\mu_2} \cdots \binom{n_v}{\mu_v} \bigg[ \sum_{\mu_1=0}^{n_1} \sum_{\alpha + \beta = \mu_1 - 2} \sum_{k=0}^{n_1-2} \sum_{\zeta+\xi=k} \binom{\zeta-1}{\alpha-1} \binom{\xi}{\beta} \bigg] \\
& \quad \cdot C_{g-1,v+1}(\alpha,\beta,\mu_2,\dots,\mu_v) \, b^{|\vec{n}|-|\vec{\mu}|}c^{|\vec{\mu}|} \\
& = \sum_{\mu_2=0}^{n_2} \cdots \sum_{\mu_v=0}^{n_v} \binom{n_2}{\mu_2} \cdots \binom{n_v}{\mu_v} \bigg[ \sum_{k=0}^{n_1-2} \sum_{\zeta+\xi=k} \sum_{\alpha=0}^{\zeta} \sum_{\beta=0}^{\xi} \binom{\zeta-1}{\alpha-1} \binom{\xi}{\beta} \bigg] \\
& \quad \cdot C_{g-1,v+1}(\alpha,\beta,\mu_2,\dots,\mu_v) \, b^{|\vec{n}|-((\alpha+\beta+2)+\mu_2+\cdots+\mu_v)} c^{(\alpha+\beta+2)+\mu_2+\cdots+\mu_v} \\
& = \sum_{k=0}^{n_1-2} \sum_{\zeta+\xi=k} \sum_{\alpha=0}^{\zeta} \sum_{\beta=0}^{\xi} \sum_{\mu_2=0}^{n_2} \cdots \sum_{\mu_v=0}^{n_v} \binom{\zeta-1}{\alpha-1} \binom{\xi}{\beta} \binom{n_2}{\mu_2} \cdots \binom{n_v}{\mu_v} \\
& \quad \cdot C_{g-1,v+1}(\alpha,\beta,\mu_2,\dots,\mu_v) \, b^{n_1-2-k} b^{((\zeta+\xi)+n_2+\cdots+n_v)-((\alpha+\beta)+\mu_2+\cdots+\mu_v)} \\
& \quad \cdot c^{(\alpha+\beta+2)+\mu_2+\cdots+\mu_v} \\
& = c^2 \sum_{k=0}^{n_1-2} b^{(n_1-2)-k} \sum_{\zeta+\xi=k} \bigg[ \widetilde{M}_{g-1,v+1}(\zeta,\xi,n_2,\dots,n_v;b,c) \\
& \quad - b \widetilde{M}_{g-1,v+1}(\zeta-1,\xi,n_2,\dots,n_v;b,c) \bigg] \\
& = c^2 \bigg[ \sum_{k=1}^{n_1-2} b^{(n_1-2)-k} \sum_{\zeta+\xi=k} \widetilde{M}_{g-1,v+1}(\zeta,\xi,n_2,\dots,n_v;b,c) \\
&  \quad - \sum_{k=1}^{n_1-2} b^{(n_1-2)-(k-1)} \sum_{\zeta+\xi=k-1} \widetilde{M}_{g-1,v+1}(\zeta,\xi,n_2,\dots,n_v;b,c) \bigg] \\
& = c^2 \bigg[ \sum_{k=1}^{n_1-2} b^{(n_1-2)-k} \sum_{\zeta+\xi=k} \widetilde{M}_{g-1,v+1}(\zeta,\xi,n_2,\dots,n_v;b,c) \\
&  \quad - \sum_{k=0}^{n_1-3} b^{(n_1-2)-k} \sum_{\zeta+\xi=k} \widetilde{M}_{g-1,v+1}(\zeta,\xi,n_2,\dots,n_v;b,c) \bigg] \\
& = c^2 b^{(n_1-2)-(n_1-2)} \sum_{\zeta+\xi=n_1-2} \widetilde{M}_{g-1,v+1}(\zeta,\xi,n_2,\dots,n_v;b,c) \\
& = c^2 \sum_{\zeta+\xi=n_1-2} \widetilde{M}_{g-1,v+1}(\zeta,\xi,n_2,\dots,n_v;b,c). \\
\end{align*}

Finally, for the third term, we have
\begin{align*}
{\text{III}}_{g,v}(\vec{n};b,c) & = \sum_{g_1+g_2=g, \; I \sqcup J = \{2,\dots,v\}} \sum_{\mu_2=0}^{n_2} \cdots \sum_{\mu_v=0}^{n_v} \binom{n_2}{\mu_2} \cdots \binom{n_v}{\mu_v} \\
& \quad \cdot\bigg[ \sum_{\mu_1=0}^{n_1} \sum_{\alpha + \beta = \mu_1 - 2} \binom{n_1-1}{\mu_1-1} \bigg] C_{g_1,|I|+1}(\alpha,\mu_I)C_{g_2,|J|+1}(\beta,\mu_J) b^{|\vec{n}|-|\vec{\mu}|}c^{|\vec{\mu}|} \\
& = \sum_{g_1+g_2=g, \; I \sqcup J = \{2,\dots,v\}} \sum_{\mu_2=0}^{n_2} \cdots \sum_{\mu_v=0}^{n_v} \binom{n_2}{\mu_2} \cdots \binom{n_v}{\mu_v} \\
& \quad \cdot \bigg[ \sum_{\mu_1=0}^{n_1} \sum_{\alpha + \beta = \mu_1 - 2} \sum_{\zeta+\xi=n_1-2} \binom{\zeta}{\alpha} \binom{\xi}{\beta} \bigg] C_{g_1,|I|+1}(\alpha,\mu_I)C_{g_2,|J|+1}(\beta,\mu_J) \\
& \quad \cdot b^{((\zeta+\xi)+n_2+\cdots+n_v)-((\alpha + \beta)+\mu_2+\cdots+\mu_v)}c^{(\alpha + \beta+2)+\mu_2+\cdots+\mu_v} \\
& = c^{2} \sum_{g_1+g_2=g, \; I \sqcup J = \{2,\dots,v\}} \sum_{\mu_2=0}^{n_2} \cdots \sum_{\mu_v=0}^{n_v} \binom{n_2}{\mu_2} \cdots \binom{n_v}{\mu_v} \\
& \quad \cdot\bigg[ \sum_{\zeta+\xi=n_1-2} \sum_{\alpha=0}^{\zeta} \sum_{\beta=0}^{\xi} \binom{\zeta}{\alpha} \binom{\xi}{\beta} \bigg] C_{g_1,|I|+1}(\alpha,\mu_I)C_{g_2,|J|+1}(\beta,\mu_J) \\
& \quad \cdot b^{(\zeta+\xi+n_2+\cdots+n_v)-(\alpha+\beta+\mu_2+\cdots+\mu_v)}c^{\alpha+\beta+\mu_2+\cdots+\mu_v} \\
& = c^2 \sum_{g_1+g_2=g, \; I \sqcup J = \{2,\dots,v\}} \sum_{\zeta+\xi=n_1-2} \widetilde{M}_{g_1,|I|+1}(\zeta,n_I;b,c)\widetilde{M}_{g_2,|J|+1}(\xi,n_J;b,c) \\
\end{align*}
where we again used the Vandermonde-like identity of equation \eqref{eqn:appvandermondelike}.

Putting all this back into equation \eqref{eqn:recthreeterms} thus completes the proof.
\end{proof}


\subsection{Detailed Proof of the $bc$-Motzkin Differential Recursion Formula}
\label{subsect:appendixdiffrec}

In this section, we provide a detailed proof of the differential recursion formula for generalized $bc$-Motzkin numbers, which was given in equation \eqref{eqn:bcmdr01} and Theorem~\ref{thm:bcmdiffrec}.

For the case when $(g,v) = (0,1)$, we may compute
\begin{align*}
\frac{\partial}{\partial x} F_{0,1}^{\widetilde{M}(b,c)}(x) & = - \sum_{n=1}^{\infty} \widetilde{M}_{0,1}(n;b,c) x^{-n-1} - \widetilde{M}_{0,1}(0;b,c) x^{-1} \\
& = - x^{-1} - \sum_{n=1}^{\infty} \bigg( b \widetilde{M}_{0,1}(n-1;b,c) \\
& \quad + c^2 \sum_{\zeta+\xi=n-2} \widetilde{M}_{0,1}(\zeta;b,c) \widetilde{M}_{0,1}(\xi;b,c) \bigg) x^{-n-1} \\
& = - x^{-1} - b\sum_{k+1=1}^{\infty} \widetilde{M}_{0,1}(k;b,c) \, x^{-(k+1)-1} \\
& \quad - c^2\sum_{\zeta=0}^{\infty} \sum_{\xi=0}^{\infty} \widetilde{M}_{0,1}(\zeta;b,c) \widetilde{M}_{0,1}(\xi;b,c) \, x^{-(\zeta+\xi+2)-1} \\
& = - x^{-1} - b x^{-1} \sum_{k=0}^{\infty} \widetilde{M}_{0,1}(k;b,c) \, x^{-k-1} \\
& \quad - c^2 x^{-1} \sum_{\zeta=0}^{\infty} \sum_{\xi=0}^{\infty} \widetilde{M}_{0,1}(\zeta;b,c) \widetilde{M}_{0,1}(\xi;b,c) \, x^{-\zeta-1}x^{-\xi-1} \\
& = - x^{-1} - b x^{-1} \bigg( - \frac{\partial}{\partial x} F_{0,1}^{\widetilde{M}(b,c)}(x) \bigg) - c^2 x^{-1} \bigg( - \frac{\partial}{\partial x} F_{0,1}^{\widetilde{M}(b,c)}(x) \bigg)^2 \\
\end{align*}

which implies
$$c^2 \bigg( \frac{\partial}{\partial x} F_{0,1}^{\widetilde{M}(b,c)}(x) \bigg)^2 + (x-b) \bigg( \frac{\partial}{\partial x} F_{0,1}^{\widetilde{M}(b,c)}(x) \bigg) + 1 = 0.$$

Thus,
$$\frac{\partial}{\partial x} F_{0,1}^{\widetilde{M}(b,c)}(x) = \frac{1}{2c} \bigg[ - \bigg( \frac{x-b}{c} \bigg) - \sqrt{ \bigg( \frac{x-b}{c} \bigg)^2 - 4} \; \bigg],$$
where we took the negative square root in the quadratic formula.

Now, we define $t(x,b,c)$ to be such that 
$$\frac{x - b}{c} = 2 + \frac{4}{t^2-1} = 2\frac{t^2+1}{t^2-1}.$$

Then,
$$\frac{\partial x}{\partial t} = c \frac{-8t}{(t^2-1)^2},$$
and we obtain
\begin{align*}
\frac{\partial}{\partial t} F_{0,1}^{\widetilde{M}(b,c)}(t) & = \frac{\partial x}{\partial t} \frac{\partial}{\partial x} F_{0,1}^{\widetilde{M}(b,c)}(x) \\
& = \bigg( c\frac{-8t}{(t^2-1)^2}\bigg) \frac{1}{2c} \bigg[ - 2\frac{t^2+1}{t^2-1} - \sqrt{ \bigg( 2\frac{t^2+1}{t^2-1} \bigg)^2 - 4 } \; \bigg] \\
& = \frac{-8t}{(t^2-1)^2} \bigg[ - \frac{t^2+1}{t^2-1} - \sqrt{ \frac{ (t^2+1)^2 - (t^2-1)^2}{(t^2-1)^2} } \; \bigg] \\
& = \frac{8t}{(t+1)(t-1)^3}. \\
\end{align*}

This gives equation \eqref{eqn:bcmdr01}.

More generally, for the case when $(g,v) \neq (0,1)$, we may proceed as follows to prove Theorem~\ref{thm:bcmdiffrec}.

\begin{proof}

With notation as in Section~\ref{sect:diffrecformula}, we have the following computations.

For the first term, 
\begin{align*}
\text{I}_{g,v}(x_1,x_2,\dots,x_v) & = \sum_{n_1=1}^{\infty} \cdots \sum_{n_v=1}^{\infty} \bigg[ b \widetilde{M}_{g,v}(n_1-1,n_2,\dots,n_v;b,c) \bigg] \frac{x_1^{-n_1-1} x_2^{-n_2} \cdots x_v^{-n_v}}{n_2 \cdots n_v} \\
& = bx_1^{-1} \sum_{k=1}^{\infty} \sum_{n_2=1}^{\infty} \cdots \sum_{n_v=1}^{\infty} \frac{\widetilde{M}_{g,v}(k,n_2,\dots,n_v;b,c)}{n_2 \cdots n_v} x_1^{-k-1} x_2^{-n_2} \cdots x_v^{-n_v} \\
& = -bx_1^{-1} \frac{\partial}{\partial x_1} F_{g,v}^{\widetilde{M}(b,c)}(x_1,x_2,\dots,x_v). \\
\end{align*}

For the second term, 
\begin{align*}
& \quad \text{II}_{g,v}(x_1,x_2,\dots,x_v) \\
& = \sum_{n_1=1}^{\infty} \cdots \sum_{n_v=1}^{\infty} \bigg[ c^2 \sum_{j=2}^v n_j \, \widetilde{M}_{g,v-1} (n_1+n_j-2, n_2, \dots, \widehat{n_j}, \dots, n_v; b; c) \bigg] \frac{x_1^{-n_1-1} x_2^{-n_2} \cdots x_v^{-n_v}}{n_2 \cdots n_v} \\
& = c^2 \sum_{j=2}^v \sum_{n_1=1}^{\infty} \cdots \sum_{n_v=1}^{\infty} \frac{\widetilde{M}_{g,v-1} (n_1+n_j-2, n_2, \dots, \widehat{n_j}, \dots, n_v; b; c)}{n_2 \cdots \widehat{n_j} \cdots n_v} x_1^{-n_1-1} x_2^{-n_2} \cdots x_v^{-n_v} \\
& = c^2 \sum_{j=2}^v \sum_{k=1}^{\infty} \sum_{n_2=1}^{\infty} \cdots \widehat{ \sum_{n_j=1}^{\infty} } \cdots \sum_{n_v=1}^{\infty} \sum_{n_1+n_j=k+2} \frac{\widetilde{M}_{g,v-1} (k, n_2, \dots, \widehat{n_j}, \dots, n_v; b; c)}{n_2 \cdots \widehat{n_j} \cdots n_v} \\
& \quad \cdot (x_1^{-n_1-1} x_j^{-n_j}) x_2^{-n_2} \cdots \widehat{ x_j^{-n_j} } \cdots x_v^{-n_v} \\
& = c^2 \sum_{j=2}^v \sum_{k=1}^{\infty} \sum_{n_2=1}^{\infty} \cdots \widehat{ \sum_{n_j=1}^{\infty} } \cdots \sum_{n_v=1}^{\infty} \frac{\widetilde{M}_{g,v-1} (k, n_2, \dots, \widehat{n_j}, \dots, n_v; b; c)}{n_2 \cdots \widehat{n_j} \cdots n_v} \\
& \quad \cdot \bigg[ \sum_{n_1+n_j=k+2} (x_1^{-n_1-1} x_j^{-n_j}) \bigg] x_2^{-n_2} \cdots \widehat{ x_j^{-n_j} } \cdots x_v^{-n_v} \\
& = c^2 \sum_{j=2}^v \sum_{k=1}^{\infty} \sum_{n_2=1}^{\infty} \cdots \widehat{ \sum_{n_j=1}^{\infty} } \cdots \sum_{n_v=1}^{\infty} \frac{\widetilde{M}_{g,v-1} (k, n_2, \dots, \widehat{n_j}, \dots, n_v; b; c)}{n_2 \cdots \widehat{n_j} \cdots n_v} \\
& \quad \bigg[ \frac{x_1^{-k-1}-x_j^{-k-1}}{x_1(x_j-x_1)} \bigg] x_2^{-n_2} \cdots \widehat{ x_j^{-n_j} } \cdots x_v^{-n_v} \\
& = c^2 \sum_{j=2}^v \frac{1}{x_1(x_j-x_1)} \bigg[ \sum_{k=1}^{\infty} \sum_{n_2=1}^{\infty} \cdots \widehat{ \sum_{n_j=1}^{\infty} } \cdots \sum_{n_v=1}^{\infty} \frac{\widetilde{M}_{g,v-1} (k, n_2, \dots, \widehat{n_j}, \dots, n_v; b; c)}{n_2 \cdots \widehat{n_j} \cdots n_v} \\
& \quad  \cdot \big[ x_1^{-k-1} x_2^{-n_2} \cdots \widehat{ x_j^{-n_j} } \cdots x_v^{-n_v} - x_2^{-n_2} \cdots x_{j-1}^{-n_{j-1}} x_j^{-k-1} x_{j+1}^{-n_{j+1}} \cdots x_v^{-n_v} \big] \bigg] \\
& = c^2 \sum_{j=2}^v \frac{1}{x_1(x_j-x_1)} \bigg[ -\frac{\partial}{\partial x_1} F_{g,v-1}^{\widetilde{M}(b,c)}(x_1,\dots,\widehat{x_j},\dots,x_v) + \frac{\partial}{\partial x_j} F_{g,v-1}^{\widetilde{M}(b,c)}(x_2,\dots,x_v) \bigg]. \\
\end{align*}

For the third term, 
\begin{align*}
& \quad \text{III}_{g,v}(x_1,x_2,\dots,x_v) \\
& = \sum_{n_1=1}^{\infty} \cdots \sum_{n_v=1}^{\infty} \bigg[ c^2 \sum_{\zeta+\xi=n_1-2} \widetilde{M}_{g-1,v+1}(\zeta,\xi,n_2,\dots,n_v;b,c) \bigg] \frac{x_1^{-n_1-1} x_2^{-n_2} \cdots x_v^{-n_v}}{n_2 \cdots n_v} \\
& = c^2 \sum_{\zeta=1}^{\infty} \sum_{\xi=1}^{\infty} \sum_{n_2=1}^{\infty} \cdots \sum_{n_v=1}^{\infty} \bigg[ \widetilde{M}_{g-1,v+1}(\zeta,\xi,n_2,\dots,n_v;b,c) \bigg] \frac{x_1^{-(\zeta+\xi+2)-1} x_2^{-n_2} \cdots x_v^{-n_v}}{n_2 \cdots n_v} \\
& = c^2 x_1^{-1} \sum_{\zeta=1}^{\infty} \sum_{\xi=1}^{\infty} \sum_{n_2=1}^{\infty} \cdots \sum_{n_v=1}^{\infty} \frac{\widetilde{M}_{g-1,v+1}(\zeta,\xi,n_2,\dots,n_v;b,c)}{n_2 \cdots n_v} x_1^{-\zeta-1}x_1^{\xi-1} x_2^{-n_2} \cdots x_v^{-n_v} \\
& = c^2 x_1^{-1} \frac{\partial}{\partial u_1} \frac{\partial}{\partial u_2} F_{g-1,v+1}^{\widetilde{M}(b,c)}(u_1,u_2,x_2,\dots,x_v) \bigg|_{u_1=u_2=x_1}. \\
\end{align*}

For the fourth term, 
\begin{align*}
& \quad \text{IV}_{g,v}(x_1,x_2,\dots,x_v) \\
& = \sum_{n_1=1}^{\infty} \cdots \sum_{n_v=1}^{\infty} \bigg[ c^2 \sum_{\zeta+\xi=n_1-2} \sum_{g_1+g_2=g, \; I \sqcup J = \{2,\dots,v\}} \widetilde{M}_{g_1,|I|+1}(\zeta,n_I;b,c)\widetilde{M}_{g_2,|J|+1}(\xi,n_J;b,c) \bigg] \\
& \quad \cdot \frac{x_1^{-n_1-1} x_2^{-n_2} \cdots x_v^{-n_v}}{n_2 \cdots n_v} \\
& = c^2 \sum_{g_1+g_2=g, \; I \sqcup J = \{2,\dots,v\}} \bigg[ \sum_{\zeta=1}^{\infty} \sum_{n_I=1}^{\infty} \sum_{\xi=1}^{\infty} \sum_{n_J=1}^{\infty} \widetilde{M}_{g_1,|I|+1}(\zeta,n_I;b,c)\widetilde{M}_{g_2,|J|+1}(\xi,n_J;b,c) \bigg] \\
& \quad \cdot \frac{x_1^{-(\zeta+\xi+2)-1} \prod_{i \in I} x_i^{-n_i} \prod_{j\in J} x_j^{-n_j} }{\prod_{i \in I} n_i \prod_{j\in J} n_j} \\
& = c^2 x_1^{-1} \sum_{g_1+g_2=g, \; I \sqcup J = \{2,\dots,v\}} \frac{\partial}{\partial x_1} F_{g_1,|I|+1}^{\widetilde{M}(b,c)}(x_1,x_I) \cdot \frac{\partial}{\partial x_1} F_{g_2,|J|+1}^{\widetilde{M}(b,c)}(x_1,x_J). \\
\end{align*}

Now, when $(g,v) = (0,2)$, we have
\begin{align*}
\frac{\partial}{\partial x_1} F_{0,2}^{\widetilde{M}(b,c)}(x_1,x_2) & = bx_1^{-1} \frac{\partial}{\partial x_1} F_{0,2}^{\widetilde{M}(b,c)}(x_1,x_2) \\
& \quad - c^2 \frac{1}{x_1(x_2-x_1)} \bigg[ -\frac{\partial}{\partial x_1} F_{0,1}^{\widetilde{M}(b,c)}(x_1) + \frac{\partial}{\partial x_2} F_{0,1}^{\widetilde{M}(b,c)}(x_2) \bigg] \\
& \quad - 2 c^2 x_1^{-1} \frac{\partial}{\partial x_1} F_{0,1}^{\widetilde{M}(b,c)}(x_1) \cdot \frac{\partial}{\partial x_1} F_{0,2}^{\widetilde{M}(b,c)}(x_1,x_2). \\
\end{align*}

This implies
\begin{align*}
0 & = \bigg( \frac{x_1-b}{c^2} + 2 \frac{\partial}{\partial x_1} F_{0,1}^{\widetilde{M}(b,c)}(x_1) \bigg) \frac{\partial}{\partial x_1} F_{0,2}^{\widetilde{M}(b,c)}(x_1,x_2) \\
& \quad + \frac{1}{(x_2-x_1)} \bigg( -\frac{\partial}{\partial x_1} F_{0,1}^{\widetilde{M}(b,c)}(x_1) + \frac{\partial}{\partial x_2} F_{0,1}^{\widetilde{M}(b,c)}(x_2) \bigg). \\
\end{align*}

The change of variables formula in equation \eqref{eqn:bcmxtsub} then implies
$$\frac{\partial x_i}{\partial t_i} = c \frac{-8t_i}{(t_i^2-1)^2}.$$

Now, observe that
\begin{align*}
\frac{1}{(x_j-x_1)} & = \frac{1}{\displaystyle \bigg( 2c\frac{t_j^2+1}{t_j^2-1} + b \bigg) - \bigg( 2c\frac{t_1^2+1}{t_1^2-1} + b \bigg)} \\
& = \frac{1}{2c} \frac{(t_1^2-1)(t_j^2-1)}{(t_j^2+1)(t_1^2-1) - (t_1^2+1)(t_j^2-1)} \\
& = \frac{1}{4c} \frac{(t_1^2-1)(t_j^2-1)}{(t_1^2 - t_j^2)}. \\
\end{align*}

And, recall from equation \eqref{eqn:bcmdr01} that, for this choice of $t_1$,
$$\frac{\partial}{\partial t_1} F_{0,1}^{\widetilde{M}(b,c)}(t_1) = \frac{8t_1}{(t_1+1)(t_1-1)^3}.$$

Hence, changing variables from $x_i$ to $t_i$ and plugging in for $F_{0,1}^{\widetilde{M}(b,c)}(t_1)$ gives
\begin{align*}
0 & = \bigg( \frac{2}{c} \frac{t_1^2+1}{t_1^2-1} + 2 \frac{\partial t_1}{\partial x_1} \frac{\partial}{\partial t_1} F_{0,1}^{\widetilde{M}(b,c)}(t_1) \bigg) \frac{\partial t_1}{\partial x_1} \frac{\partial}{\partial t_1} F_{0,2}^{\widetilde{M}(b,c)}(t_1,t_2) \\
& \quad + \frac{1}{4c} \frac{(t_1^2-1)(t_2^2-1)}{(t_1^2 - t_2^2)} \bigg[ - \frac{\partial t_1}{\partial x_1} \frac{\partial}{\partial t_1} F_{0,1}^{\widetilde{M}(b,c)}(t_1) + \frac{\partial t_2}{\partial x_2} \frac{\partial}{\partial t_2} F_{0,1}^{\widetilde{M}(b,c)}(t_2) \bigg] \\
& = \bigg( \frac{2}{c} \frac{t_1^2+1}{t_1^2-1} + 2 \frac{(t_1^2-1)^2}{-8ct_1} \frac{8t_1}{(t_1+1)(t_1-1)^3} \bigg) \frac{(t_1^2-1)^2}{-8ct_1} \frac{\partial}{\partial t_1} F_{0,2}^{\widetilde{M}(b,c)}(t_1,t_2) \\
& \quad + \frac{1}{4c} \frac{(t_1^2-1)(t_2^2-1)}{(t_1^2 - t_2^2)} \bigg[ - \frac{(t_1^2-1)^2}{-8ct_1} \frac{8t_1}{(t_1+1)(t_1-1)^3} + \frac{(t_2^2-1)^2}{-8ct_2} \frac{8t_2}{(t_2+1)(t_2-1)^3} \bigg] \\
& = \frac{(t_1^2-1)^2}{-4c^2t_1} \bigg( \frac{(t_1^2+1)}{(t_1-1)(t_1+1)} - \frac{(t_1+1)}{(t_1-1)} \bigg) \frac{\partial}{\partial t_1} F_{0,2}^{\widetilde{M}(b,c)}(t_1,t_2) \\
& \quad + \frac{1}{4c^2} \frac{(t_1^2-1)(t_2^2-1)}{(t_1^2 - t_2^2)} \bigg[ \frac{(t_1^2-1)}{(t_1-1)^2} - \frac{(t_2^2-1)}{(t_2-1)^2} \bigg] \\
& = \frac{(t_1^2-1)^2}{-4c^2t_1} \bigg( \frac{ -2t_1 }{(t_1-1)(t_1+1)} \bigg) \frac{\partial}{\partial t_1} F_{0,2}^{\widetilde{M}(b,c)}(t_1,t_2) \\
& \quad + \frac{1}{4c^2} \frac{(t_1^2-1)(t_2^2-1)}{(t_1^2 - t_2^2)} \bigg[ \frac{(t_1^2-1)}{(t_1-1)^2} - \frac{(t_2^2-1)}{(t_2-1)^2} \bigg] \\
& = \frac{(t_1^2-1)}{2c^2} \frac{\partial}{\partial t_1} F_{0,2}^{\widetilde{M}(b,c)}(t_1,t_2) + \frac{1}{4c^2} \frac{(t_1^2-1)(t_2^2-1)}{(t_1^2 - t_2^2)} \bigg[ \frac{(t_1^2-1)}{(t_1-1)^2} - \frac{(t_2^2-1)}{(t_2-1)^2} \bigg]. \\
\end{align*}

Therefore,
\begin{align*}
\frac{\partial}{\partial t_1} F_{0,2}^{\widetilde{M}(b,c)}(t_1,t_2) & = \frac{-1}{4c^2} \frac{2c^2}{(t_1^2-1)} \frac{(t_1^2-1)(t_2^2-1)}{(t_1^2 - t_2^2)} \bigg[ \frac{(t_1^2-1)}{(t_1-1)^2} - \frac{(t_2^2-1)}{(t_2-1)^2} \bigg] \\
& = \frac{1}{2} \frac{(t_2^2-1)}{(t_1^2 - t_2^2)} \bigg[ \frac{(t_2+1)}{(t_2-1)} - \frac{(t_1+1)}{(t_1-1)} \bigg] \\
& = \frac{1}{2} \frac{(t_2^2-1)}{(t_1^2 - t_2^2)} \bigg[ \frac{2(t_1-t_2)}{(t_1-1)(t_2-1)} \bigg] \\
& = \frac{(t_2+1)}{(t_1-1)(t_1+t_2)}. \\
\end{align*}

Now, returning to the general case, we want to write the sum in $\text{IV}_{g,v}(x_1,x_2,\dots,x_v)$ in such a way that it does not contain any $(g,v)$ or $(g,v-1)$ terms (so that we can combine these terms with the comparable terms occurring elsewhere in the formula). And, after pulling out the $(g,v)$ and $(g,v-1)$ terms from this sum, the remaining terms are precisely the ``stable'' terms, i.e. when $2g_1 + |I| - 1 > 0$ and $2g_2 + |J| - 1 > 0$. 

Thus, we see that
\begin{align*}
& \quad \text{IV}_{g,v}(x_1,x_2,\dots,x_v) \\
& = c^2 \bigg[ \sum_{g_1+g_2=g, \; I \sqcup J = \{2,\dots,v\}, \; \text{stable}} \frac{\partial}{\partial x_1} F_{g_1,|I|+1}^{\widetilde{M}(b,c)}(x_1,x_I) \cdot \frac{\partial}{\partial x_1} F_{g_2,|J|+1}^{\widetilde{M}(b,c)}(x_1,x_J) \\
& \quad + 2 \frac{\partial}{\partial x_1} F_{0,1}^{\widetilde{M}(b,c)}(x_1) \cdot \frac{\partial}{\partial x_1} F_{g,v}^{\widetilde{M}(b,c)}(x_1,x_2,\dots,x_v) \\
& \quad + 2 \sum_{j=2}^v \frac{\partial}{\partial x_1} F_{0,2}^{\widetilde{M}(b,c)}(x_1,x_j) \cdot \frac{\partial}{\partial x_1} F_{g,v-1}^{\widetilde{M}(b,c)}(x_1,x_2,\dots,\widehat{x_j},\dots,x_v) \bigg]. \\
\end{align*}

Therefore, putting all four terms back into the original equation, we obtain
\begin{align*}
& \quad \frac{\partial}{\partial x_1} F_{g,v}^{\widetilde{M}(b,c)}(x_1,x_2,\dots,x_v) \\
& = - \bigg\{ -bx_1^{-1} \frac{\partial}{\partial x_1} F_{g,v}^{\widetilde{M}(b,c)}(x_1,\dots,x_v) \\
& \quad + c^2 \sum_{j=2}^v \frac{1}{x_1(x_j-x_1)} \bigg[ -\frac{\partial}{\partial x_1} F_{g,v-1}^{\widetilde{M}(b,c)}(x_1,\dots,\widehat{x_j},\dots,x_v) \\
& \quad + \frac{\partial}{\partial x_j} F_{g,v-1}^{\widetilde{M}(b,c)}(x_2,\dots,x_v) \bigg] \\
& \quad + c^2 x_1^{-1} \frac{\partial}{\partial u_1} \frac{\partial}{\partial u_2} F_{g-1,v+1}^{\widetilde{M}(b,c)}(u_1,u_2,x_2,\dots,x_v) \bigg|_{u_1=u_2=x_1} \\
& \quad + c^2 x_1^{-1} \bigg[ \sum_{g_1+g_2=g, \; I \sqcup J = \{2,\dots,v\}, \; \text{stable}} \frac{\partial}{\partial x_1} F_{g_1,|I|+1}^{\widetilde{M}(b,c)}(x_1,x_I) \cdot \frac{\partial}{\partial x_1} F_{g_2,|J|+1}^{\widetilde{M}(b,c)}(x_1,x_J) \\
& \quad + 2 \frac{\partial}{\partial x_1} F_{0,1}^{\widetilde{M}(b,c)}(x_1) \cdot \frac{\partial}{\partial x_1} F_{g,v}^{\widetilde{M}(b,c)}(x_1,x_2,\dots,x_v) \\
& \quad + 2 \sum_{j=2}^v \frac{\partial}{\partial x_1} F_{0,2}^{\widetilde{M}(b,c)}(x_1,x_j) \cdot \frac{\partial}{\partial x_1} F_{g,v-1}^{\widetilde{M}(b,c)}(x_1,x_2,\dots,\widehat{x_j},\dots,x_v) \bigg] \bigg\}. \\
\end{align*}

This implies
\begin{align*}
0 & = \bigg( \frac{x_1-b}{c^2} \bigg) \bigg[ \frac{\partial}{\partial x_1} F_{g,v}^{\widetilde{M}(b,c)}(x_1,x_2,\dots,x_v) \bigg] \\
& \quad + \sum_{j=2}^v \frac{1}{(x_j-x_1)} \bigg[ -\frac{\partial}{\partial x_1} F_{g,v-1}^{\widetilde{M}(b,c)}(x_1,\dots,\widehat{x_j},\dots,x_v) + \frac{\partial}{\partial x_j} F_{g,v-1}^{\widetilde{M}(b,c)}(x_2,\dots,x_v) \bigg] \\
& \quad + \frac{\partial}{\partial u_1} \frac{\partial}{\partial u_2} F_{g-1,v+1}^{\widetilde{M}(b,c)}(u_1,u_2,x_2,\dots,x_v) \bigg|_{u_1=u_2=x_1} \\
& \quad + \sum_{g_1+g_2=g, \; I \sqcup J = \{2,\dots,v\}, \; \text{stable}} \frac{\partial}{\partial x_1} F_{g_1,|I|+1}^{\widetilde{M}(b,c)}(x_1,x_I) \cdot \frac{\partial}{\partial x_1} F_{g_2,|J|+1}^{\widetilde{M}(b,c)}(x_1,x_J) \\
& \quad + 2 \frac{\partial}{\partial x_1} F_{0,1}^{\widetilde{M}(b,c)}(x_1) \cdot \frac{\partial}{\partial x_1} F_{g,v}^{\widetilde{M}(b,c)}(x_1,x_2,\dots,x_v) \\
& \quad + 2 \sum_{j=2}^v \frac{\partial}{\partial x_1} F_{0,2}^{\widetilde{M}(b,c)}(x_1,x_j) \cdot \frac{\partial}{\partial x_1} F_{g,v-1}^{\widetilde{M}(b,c)}(x_1,x_2,\dots,\widehat{x_j},\dots,x_v). \\
\end{align*}

And, hence,
\begin{align*}
0 & = \bigg( \frac{x_1-b}{c^2} + 2 \frac{\partial}{\partial x_1} F_{0,1}^{\widetilde{M}(b,c)}(x_1) \bigg) \bigg[ \frac{\partial}{\partial x_1} F_{g,v}^{\widetilde{M}(b,c)}(x_1,x_2,\dots,x_v) \bigg] \\
& \quad + \sum_{j=2}^v \frac{1}{(x_j-x_1)} \bigg[ \frac{\partial}{\partial x_j} F_{g,v-1}^{\widetilde{M}(b,c)}(x_2,\dots,x_v) \bigg] \\
& \quad + \sum_{j=2}^v \bigg( 2 \frac{\partial}{\partial x_1} F_{0,2}^{\widetilde{M}(b,c)}(x_1,x_j) - \frac{1}{(x_j-x_1)} \bigg) \bigg[ \frac{\partial}{\partial x_1} F_{g,v-1}^{\widetilde{M}(b,c)}(x_1,\dots,\widehat{x_j},\dots,x_v) \bigg] \\
& \quad + \frac{\partial}{\partial u_1} \frac{\partial}{\partial u_2} F_{g-1,v+1}^{\widetilde{M}(b,c)}(u_1,u_2,x_2,\dots,x_v) \bigg|_{u_1=u_2=x_1} \\
& \quad + \sum_{g_1+g_2=g, \; I \sqcup J = \{2,\dots,v\}, \; \text{stable}} \frac{\partial}{\partial x_1} F_{g_1,|I|+1}^{\widetilde{M}(b,c)}(x_1,x_I) \cdot \frac{\partial}{\partial x_1} F_{g_2,|J|+1}^{\widetilde{M}(b,c)}(x_1,x_J) \\
\end{align*}

After applying the substitution for $x_i$ in terms of $t_i$ from equation \eqref{eqn:bcmxtsub}, this becomes
\begin{align*}
0 & = \bigg( \frac{2}{c} \frac{t_1^2+1}{t_1^2-1} + 2 \frac{\partial t_1}{\partial x_1} \frac{\partial}{\partial t_1} F_{0,1}^{\widetilde{M}(b,c)}(t_1) \bigg) \bigg[ \frac{\partial t_1}{\partial x_1} \frac{\partial}{\partial t_1} F_{g,v}^{\widetilde{M}(b,c)}(t_1,t_2,\dots,t_v) \bigg] \\
& \quad + \sum_{j=2}^v \bigg( \frac{1}{4c} \frac{(t_1^2-1)(t_j^2-1)}{(t_1^2 - t_j^2)} \bigg) \bigg[ \frac{\partial t_j}{\partial x_j} \frac{\partial}{\partial t_j} F_{g,v-1}^{\widetilde{M}(b,c)}(t_2,\dots,t_v) \bigg] \\
& \quad + \sum_{j=2}^v \bigg( 2 \frac{\partial t_1}{\partial x_1} \frac{\partial}{\partial t_1} F_{0,2}^{\widetilde{M}(b,c)}(t_1,t_j) - \frac{1}{4c} \frac{(t_1^2-1)(t_j^2-1)}{(t_1^2 - t_j^2)} \bigg) \\
& \quad \cdot \bigg[ \frac{\partial t_1}{\partial x_1} \frac{\partial}{\partial t_1} F_{g,v-1}^{\widetilde{M}(b,c)}(t_1,\dots,\widehat{t_j},\dots,t_v) \bigg] \\
& \quad + \bigg( \frac{\partial t_1}{\partial x_1} \bigg)^2 \frac{\partial}{\partial u_1} \frac{\partial}{\partial u_2} F_{g-1,v+1}^{\widetilde{M}(b,c)}(u_1,u_2,t_2,\dots,t_v) \bigg|_{u_1=u_2=t_1} \\
& \quad + \sum_{g_1+g_2=g, \; I \sqcup J = \{2,\dots,v\}, \; \text{stable}} \frac{\partial t_1}{\partial x_1} \frac{\partial}{\partial t_1} F_{g_1,|I|+1}^{\widetilde{M}(b,c)}(t_1,t_I) \cdot \frac{\partial t_1}{\partial x_1} \frac{\partial}{\partial t_1} F_{g_2,|J|+1}^{\widetilde{M}(b,c)}(t_1,t_J) \\
& = \bigg( \frac{2}{c} \frac{t_1^2+1}{t_1^2-1} + 2 \frac{(t_1^2-1)^2}{-8ct_1} \frac{8t_1}{(t_1+1)(t_1-1)^3} \bigg) \bigg[ \frac{(t_1^2-1)^2}{-8ct_1} \frac{\partial}{\partial t_1} F_{g,v}^{\widetilde{M}(b,c)}(t_1,t_2,\dots,t_v) \bigg] \\
& \quad + \sum_{j=2}^v \bigg( \frac{1}{4c} \frac{(t_1^2-1)(t_j^2-1)}{(t_1^2 - t_j^2)} \bigg) \bigg[ \frac{(t_j^2-1)^2}{-8ct_j} \frac{\partial}{\partial t_j} F_{g,v-1}^{\widetilde{M}(b,c)}(t_2,\dots,t_v) \bigg] \\
& \quad + \sum_{j=2}^v \bigg( 2 \frac{(t_1^2-1)^2}{-8ct_1} \frac{(t_j+1)}{(t_1-1)(t_1+t_j)} - \frac{1}{4c} \frac{(t_1^2-1)(t_j^2-1)}{(t_1^2 - t_j^2)} \bigg) \\
& \quad \cdot \bigg[ \frac{(t_1^2-1)^2}{-8ct_1} \frac{\partial}{\partial t_1} F_{g,v-1}^{\widetilde{M}(b,c)}(t_1,\dots,\widehat{t_j},\dots,t_v) \bigg] \\
& \quad + \bigg( \frac{(t_1^2-1)^2}{-8ct_1} \bigg)^2 \frac{\partial}{\partial u_1} \frac{\partial}{\partial u_2} F_{g-1,v+1}^{\widetilde{M}(b,c)}(u_1,u_2,t_2,\dots,t_v) \bigg|_{u_1=u_2=t_1} \\
& \quad + \sum_{g_1+g_2=g, \; I \sqcup J = \{2,\dots,v\}, \; \text{stable}} \frac{(t_1^2-1)^2}{-8ct_1} \frac{\partial}{\partial t_1} F_{g_1,|I|+1}^{\widetilde{M}(b,c)}(t_1,t_I) \cdot \frac{(t_1^2-1)^2}{-8ct_1} \frac{\partial}{\partial t_1} F_{g_2,|J|+1}^{\widetilde{M}(b,c)}(t_1,t_J) \\
& = \frac{(t_1^2-1)^2}{-4c^2t_1} \bigg( \frac{t_1^2+1}{t_1^2-1} - \frac{(t_1^2-1)}{(t_1-1)^2} \bigg) \bigg[ \frac{\partial}{\partial t_1} F_{g,v}^{\widetilde{M}(b,c)}(t_1,t_2,\dots,t_v) \bigg] \\
& \quad - \sum_{j=2}^v \frac{(t_1^2-1)(t_j^2-1)^3}{32 c^2 t_j(t_1^2 - t_j^2)} \bigg[ \frac{\partial}{\partial t_j} F_{g,v-1}^{\widetilde{M}(b,c)}(t_2,\dots,t_v) \bigg] \\
& \quad + \sum_{j=2}^v \frac{(t_1^2-1)^3}{32 c^2 t_1} \bigg( \frac{(t_1+1)(t_j+1)}{t_1(t_1+t_j)} + \frac{(t_j^2-1)}{(t_1^2 - t_j^2)} \bigg) \bigg[ \frac{\partial}{\partial t_1} F_{g,v-1}^{\widetilde{M}(b,c)}(t_1,\dots,\widehat{t_j},\dots,t_v) \bigg] \\
& \quad + \frac{(t_1^2-1)^4}{64 c^2 t_1^2} \frac{\partial}{\partial u_1} \frac{\partial}{\partial u_2} F_{g-1,v+1}^{\widetilde{M}(b,c)}(u_1,u_2,t_2,\dots,t_v) \bigg|_{u_1=u_2=t_1} \\
& \quad + \frac{(t_1^2-1)^4}{64 c^2 t_1^2} \sum_{g_1+g_2=g, \; I \sqcup J = \{2,\dots,v\}, \; \text{stable}} \frac{\partial}{\partial t_1} F_{g_1,|I|+1}^{\widetilde{M}(b,c)}(t_1,t_I) \cdot \frac{\partial}{\partial t_1} F_{g_2,|J|+1}^{\widetilde{M}(b,c)}(t_1,t_J) \\
& = \frac{(t_1^2-1)^2}{-4c^2t_1} \bigg( \frac{(t_1^2+1) - (t_1+1)^2}{(t_1+1)(t_1-1)} \bigg) \bigg[ \frac{\partial}{\partial t_1} F_{g,v}^{\widetilde{M}(b,c)}(t_1,t_2,\dots,t_v) \bigg] \\
& \quad - \sum_{j=2}^v \frac{(t_1^2-1)(t_j^2-1)^3}{32 c^2 t_j(t_1^2 - t_j^2)} \bigg[ \frac{\partial}{\partial t_j} F_{g,v-1}^{\widetilde{M}(b,c)}(t_2,\dots,t_v) \bigg] \\
& \quad + \sum_{j=2}^v \frac{(t_1^2-1)^3}{32 c^2 t_1} \bigg( \frac{t_1+1}{t_1} + \frac{t_j-1}{t_1-t_j} \bigg) \bigg( \frac{t_j+1}{t_1+t_j} \bigg) \bigg[ \frac{\partial}{\partial t_1} F_{g,v-1}^{\widetilde{M}(b,c)}(t_1,\dots,\widehat{t_j},\dots,t_v) \bigg] \\
& \quad + \frac{(t_1^2-1)^4}{64 c^2 t_1^2} \frac{\partial}{\partial u_1} \frac{\partial}{\partial u_2} F_{g-1,v+1}^{\widetilde{M}(b,c)}(u_1,u_2,t_2,\dots,t_v) \bigg|_{u_1=u_2=t_1} \\
& \quad + \frac{(t_1^2-1)^4}{64 c^2 t_1^2} \sum_{g_1+g_2=g, \; I \sqcup J = \{2,\dots,v\}, \; \text{stable}} \frac{\partial}{\partial t_1} F_{g_1,|I|+1}^{\widetilde{M}(b,c)}(t_1,t_I) \cdot \frac{\partial}{\partial t_1} F_{g_2,|J|+1}^{\widetilde{M}(b,c)}(t_1,t_J) \\
& = \frac{-2t_1(t_1^2-1)}{-4c^2t_1} \bigg[ \frac{\partial}{\partial t_1} F_{g,v}^{\widetilde{M}(b,c)}(t_1,t_2,\dots,t_v) \bigg] \\
& \quad - \sum_{j=2}^v \frac{(t_1^2-1)(t_j^2-1)^3}{32 c^2 t_j(t_1^2 - t_j^2)} \bigg[ \frac{\partial}{\partial t_j} F_{g,v-1}^{\widetilde{M}(b,c)}(t_2,\dots,t_v) \bigg] \\
& \quad + \sum_{j=2}^v \frac{(t_1^2-1)^3}{32 c^2 t_1} \bigg( \frac{t_1^2-t_j}{t_1(t_1-t_j)} \bigg) \bigg( \frac{t_j+1}{t_1+t_j} \bigg) \bigg[ \frac{\partial}{\partial t_1} F_{g,v-1}^{\widetilde{M}(b,c)}(t_1,\dots,\widehat{t_j},\dots,t_v) \bigg] \\
& \quad + \frac{(t_1^2-1)^4}{64 c^2 t_1^2} \frac{\partial}{\partial u_1} \frac{\partial}{\partial u_2} F_{g-1,v+1}^{\widetilde{M}(b,c)}(u_1,u_2,t_2,\dots,t_v) \bigg|_{u_1=u_2=t_1} \\
& \quad + \frac{(t_1^2-1)^4}{64 c^2 t_1^2} \sum_{g_1+g_2=g, \; I \sqcup J = \{2,\dots,v\}, \; \text{stable}} \frac{\partial}{\partial t_1} F_{g_1,|I|+1}^{\widetilde{M}(b,c)}(t_1,t_I) \cdot \frac{\partial}{\partial t_1} F_{g_2,|J|+1}^{\widetilde{M}(b,c)}(t_1,t_J) \\
& = \frac{(t_1^2-1)}{2c^2} \bigg[ \frac{\partial}{\partial t_1} F_{g,v}^{\widetilde{M}(b,c)}(t_1,t_2,\dots,t_v) \bigg] \\
& \quad - \sum_{j=2}^v \frac{(t_1^2-1)(t_j^2-1)^3}{32 c^2 t_j(t_1^2 - t_j^2)} \bigg[ \frac{\partial}{\partial t_j} F_{g,v-1}^{\widetilde{M}(b,c)}(t_2,\dots,t_v) \bigg] \\
& \quad + \sum_{j=2}^v \frac{(t_1^2-1)^3}{32 c^2 t_1^2} \bigg( 1 + \frac{(t_1^2-1)t_j}{t_1^2-t_j^2} \bigg) \bigg[ \frac{\partial}{\partial t_1} F_{g,v-1}^{\widetilde{M}(b,c)}(t_1,\dots,\widehat{t_j},\dots,t_v) \bigg] \\
& \quad + \frac{(t_1^2-1)^4}{64 c^2 t_1^2} \frac{\partial}{\partial u_1} \frac{\partial}{\partial u_2} F_{g-1,v+1}^{\widetilde{M}(b,c)}(u_1,u_2,t_2,\dots,t_v) \bigg|_{u_1=u_2=t_1} \\
& \quad + \frac{(t_1^2-1)^4}{64 c^2 t_1^2} \sum_{g_1+g_2=g, \; I \sqcup J = \{2,\dots,v\}, \; \text{stable}} \frac{\partial}{\partial t_1} F_{g_1,|I|+1}^{\widetilde{M}(b,c)}(t_1,t_I) \cdot \frac{\partial}{\partial t_1} F_{g_2,|J|+1}^{\widetilde{M}(b,c)}(t_1,t_J). \\
\end{align*}

Therefore,
\begin{align*}
& \quad \frac{\partial}{\partial t_1} F_{g,v}^{\widetilde{M}(b,c)}(t_1,t_2,\dots,t_v) \\
& = - \frac{2c^2}{(t_1^2-1)} \bigg\{ - \sum_{j=2}^v \frac{(t_1^2-1)(t_j^2-1)^3}{32 c^2 t_j(t_1^2 - t_j^2)} \bigg[ \frac{\partial}{\partial t_j} F_{g,v-1}^{\widetilde{M}(b,c)}(t_2,\dots,t_v) \bigg] \\
& \quad + \sum_{j=2}^v \frac{(t_1^2-1)^3}{32 c^2 t_1^2} \bigg( 1 + \frac{(t_1^2-1)t_j}{t_1^2-t_j^2} \bigg) \bigg[ \frac{\partial}{\partial t_1} F_{g,v-1}^{\widetilde{M}(b,c)}(t_1,\dots,\widehat{t_j},\dots,t_v) \bigg] \\
& \quad + \frac{(t_1^2-1)^4}{64 c^2 t_1^2} \frac{\partial}{\partial u_1} \frac{\partial}{\partial u_2} F_{g-1,v+1}^{\widetilde{M}(b,c)}(u_1,u_2,t_2,\dots,t_v) \bigg|_{u_1=u_2=t_1} \\
& \quad + \frac{(t_1^2-1)^4}{64 c^2 t_1^2} \sum_{g_1+g_2=g, \; I \sqcup J = \{2,\dots,v\}, \; \text{stable}} \frac{\partial}{\partial t_1} F_{g_1,|I|+1}^{\widetilde{M}(b,c)}(t_1,t_I) \cdot \frac{\partial}{\partial t_1} F_{g_2,|J|+1}^{\widetilde{M}(b,c)}(t_1,t_J) \bigg\} \\
& = \sum_{j=2}^v \frac{(t_j^2-1)^3}{16 t_j(t_1^2 - t_j^2)} \bigg[ \frac{\partial}{\partial t_j} F_{g,v-1}^{\widetilde{M}(b,c)}(t_2,\dots,t_v) \bigg] \\
& \quad - \sum_{j=2}^v \bigg[ \frac{(t_1^2-1)^3 t_j}{16 t_1^2 (t_1^2 - t_j^2) } + \frac{(t_1^2-1)^2}{16 t_1^2} \bigg] \bigg[ \frac{\partial}{\partial t_1} F_{g,v-1}^{\widetilde{M}(b,c)}(t_1,\dots,\widehat{t_j},\dots,t_v) \bigg] \\
& \quad - \frac{(t_1^2-1)^3}{32 t_1^2} \frac{\partial}{\partial u_1} \frac{\partial}{\partial u_2} F_{g-1,v+1}^{\widetilde{M}(b,c)}(u_1,u_2,t_2,\dots,t_v) \bigg|_{u_1=u_2=t_1} \\
& \quad - \frac{(t_1^2-1)^3}{32 t_1^2} \sum_{g_1+g_2=g, \; I \sqcup J = \{2,\dots,v\}, \; \text{stable}} \frac{\partial}{\partial t_1} F_{g_1,|I|+1}^{\widetilde{M}(b,c)}(t_1,t_I) \cdot \frac{\partial}{\partial t_1} F_{g_2,|J|+1}^{\widetilde{M}(b,c)}(t_1,t_J). \\
\end{align*}

This completes the proof of Theorem~\ref{thm:bcmdiffrec}.
\end{proof}


\subsection{Detailed Proof of the $bc$-Motzkin Topological Recursion}
\label{subsect:appendixtoprec}

In this section, we provide a detailed proof of the topological recursion for generalized $bc$-Motzkin numbers, which was given in Theorem~\ref{thm:bcmtoprec}.

\begin{proof}

With notation as in Section~\ref{sect:toprecformula}, we have the following computations.

For the first term,
\begin{align*}
& \quad \text{I}_{g,v}(t_1,t_2,\dots,t_v) \\
& = \frac{\partial}{\partial t_2} \cdots \frac{\partial}{\partial t_v} \bigg[ - \frac{1}{16} \sum_{j=2}^v \bigg[ \frac{t_j}{t_1^2 - t_j^2} \bigg( \frac{(t_1^2-1)^3}{t_1^2} \frac{\partial}{\partial t_1} F_{g,v-1}^{\widetilde{M}(b,c)}(t_1,\dots,\widehat{t_j},\dots,t_v) \\
& \quad - \frac{(t_j^2-1)^3}{t_j^2} \frac{\partial}{\partial t_j} F_{g,v-1}^{\widetilde{M}(b,c)}(t_2,\dots,t_v) \bigg) \bigg] dt_1 dt_2 \cdots dt_v \\
& = - \frac{1}{16} \frac{\partial}{\partial t_j} \sum_{j=2}^v \frac{t_j}{t_1^2 - t_j^2} \bigg( \frac{(t_1^2-1)^3}{t_1^2} \frac{\partial}{\partial t_1} \cdots \widehat{\frac{\partial}{\partial t_j}} \cdots \frac{\partial}{\partial t_v}  F_{g,v-1}^{\widetilde{M}(b,c)}(t_1,\dots,\widehat{t_j},\dots,t_v) \\
& \quad - \frac{(t_j^2-1)^3}{t_j^2} \frac{\partial}{\partial t_2} \cdots \frac{\partial}{\partial t_v}  F_{g,v-1}^{\widetilde{M}(b,c)}(t_2,\dots,t_v) \bigg) dt_1 dt_2 \cdots dt_v \\
& = - \frac{1}{16} \frac{\partial}{\partial t_j} \sum_{j=2}^v \frac{t_j}{t_1^2 - t_j^2} \bigg[ \frac{(t_1^2-1)^3}{t_1^2} W_{g,v-1}^{\widetilde{M}(b,c)}(t_1,\dots,\widehat{t_j},\dots,t_v) \; dt_j \\
& \quad - \frac{(t_j^2-1)^3}{t_j^2} W_{g,v-1}^{\widetilde{M}(b,c)}(t_2,\dots,t_v) \; dt_1 \bigg] \\
& = \frac{1}{16} \sum_{j=2}^v \bigg[ \frac{\partial}{\partial t_j} \bigg( \frac{t_j}{t_1^2 - t_j^2} \frac{(t_j^2-1)^3}{t_j^2} W_{g,v-1}^{\widetilde{M}(b,c)}(t_2,\dots,t_v) \bigg) dt_1 \\
& \quad - \frac{(t_1^2 - t_j^2)(1)-t_j(-2t_j)}{(t_1^2 - t_j^2)^2}\frac{(t_1^2-1)^3}{t_1^2} W_{g,v-1}^{\widetilde{M}(b,c)}(t_1,\dots,\widehat{t_j},\dots,t_v) \; dt_j \bigg] \\
& = \frac{1}{16} \sum_{j=2}^v \bigg[ \frac{\partial}{\partial t_j} \bigg( \frac{t_j}{t_1^2 - t_j^2} \frac{(t_j^2-1)^3}{t_j^2} W_{g,v-1}^{\widetilde{M}(b,c)}(t_2,\dots,t_v) \bigg) dt_1 \\
& \quad - \frac{t_1^2 + t_j^2}{(t_1^2 - t_j^2)^2}\frac{(t_1^2-1)^3}{t_1^2} W_{g,v-1}^{\widetilde{M}(b,c)}(t_1,\dots,\widehat{t_j},\dots,t_v) \; dt_j \bigg]. \\
\end{align*}

And, from the topological recursion formula, we have
\begin{align*}
& \quad - \frac{1}{64} \frac{1}{2\pi i} \int_{\gamma} \bigg( \frac{1}{t+t_1} + \frac{1}{t-t_1} \bigg) \frac{(t^2-1)^3}{t^2} \cdot \frac{1}{dt} \cdot dt_1 \\
& \quad \times \sum_{j=2}^v \bigg( W_{0,2}^{\widetilde{M}(b,c)}(t,t_j) W_{g,v-1}^{\widetilde{M}(b,c)}(-t,t_2,\dots,\widehat{t_j},\dots,t_v) \\
& \quad + W_{0,2}^{\widetilde{M}(b,c)}(-t,t_j) W_{g,v-1}^{\widetilde{M}(b,c)}(t,t_2,\dots,\widehat{t_j},\dots,t_v) \bigg) \\
& = - \frac{1}{64} \frac{1}{2\pi i} \int_{\gamma} \bigg( \frac{2(t^2-1)^3/t}{(t+t_1)(t-t_1)} \bigg) \cdot \frac{1}{dt} \cdot dt_1 \\
& \quad \times \sum_{j=2}^v \bigg( \frac{dt \cdot dt_j}{(t-t_j)^2} w_{g,v-1}^{\widetilde{M}(b,c)}(-t,t_2,\dots,\widehat{t_j},\dots,t_v) \; (-dt) \cdot dt_2 \cdots \widehat {dt_j} \cdots dt_v \\
& \quad + \frac{-dt \cdot dt_j}{(-t-t_j)^2} w_{g,v-1}^{\widetilde{M}(b,c)}(t,t_2,\dots,\widehat{t_j},\dots,t_v) \; dt \cdot dt_2 \cdots \widehat {dt_j} \cdots dt_v \bigg) \\
& = - \frac{1}{64} \frac{1}{2\pi i} \int_{\gamma} \bigg( \frac{2(t^2-1)^3/t}{(t+t_1)(t-t_1)} \bigg) \cdot \frac{1}{dt} \cdot dt_1 \\
& \quad \times \sum_{j=2}^v \bigg( -\frac{1}{(t-t_j)^2} w_{g,v-1}^{\widetilde{M}(b,c)}(t,t_2,\dots,\widehat{t_j},\dots,t_v) \\
& \quad - \frac{1}{(t+t_j)^2} w_{g,v-1}^{\widetilde{M}(b,c)}(t,t_2,\dots,\widehat{t_j},\dots,t_v) \bigg) dt \cdot dt \cdot dt_2 \cdots dt_v \\
& = \frac{1}{64} \sum_{j=2}^v \bigg[ \frac{1}{2\pi i} \int_{\gamma} \bigg( \frac{2(t^2-1)^3/t}{(t+t_1)(t-t_1)} \bigg)\bigg( \frac{1}{(t-t_j)^2} + \frac{1}{(t+t_j)^2}  \bigg) \\
& \quad \cdot w_{g,v-1}^{\widetilde{M}(b,c)}(t,t_2,\dots,\widehat{t_j},\dots,t_v) \; dt \bigg] dt_1 \cdot dt_2 \cdots dt_v \\
& = - \frac{1}{16} \sum_{j=2}^v \bigg[ \bigg( \text{Res}_{t=t_1} + \text{Res}_{t=-t_1} + \text{Res}_{t=t_j} + \text{Res}_{t=-t_j} \bigg) \bigg( \frac{(t^2+t_j^2)(t^2-1)^3/t}{(t+t_1)(t-t_1)(t-t_j)^2(t+t_j)^2} \bigg) \\
& \quad \times w_{g,v-1}^{\widetilde{M}(b,c)}(t,t_2,\dots,\widehat{t_j},\dots,t_v) \bigg] dt_1 \cdot dt_2 \cdots dt_v \\
& = - \frac{1}{16} \sum_{j=2}^v \bigg\{ \bigg( \frac{(t_1^2+t_j^2)(t_1^2-1)^3/t_1}{(t_1+t_1)(t_1-t_j)^2(t_1+t_j)^2} \\
& \quad + \frac{(t_1^2+t_j^2)(t_1^2-1)^3/(-t_1)}{(-t_1-t_1)(-t_1-t_j)^2(-t_1+t_j)^2} \bigg) w_{g,v-1}^{\widetilde{M}(b,c)}(t_1,t_2,\dots,\widehat{t_j},\dots,t_v) \\
& \quad + \bigg[ \frac{\partial}{\partial t} \bigg( \frac{(t^2+t_j^2)(t^2-1)^3/t}{(t+t_1)(t-t_1)(t+t_j)^2} \times w_{g,v-1}^{\widetilde{M}(b,c)}(t,t_2,\dots,\widehat{t_j},\dots,t_v) \bigg)\bigg|_{t=t_j} \\
& \quad + \frac{\partial}{\partial t} \bigg( \frac{(t^2+t_j^2)(t^2-1)^3/t}{(t+t_1)(t-t_1)(t-t_j)^2} \times w_{g,v-1}^{\widetilde{M}(b,c)}(t,t_2,\dots,\widehat{t_j},\dots,t_v) \bigg)\bigg|_{t=-t_j} \bigg] \bigg\} dt_1 \cdot dt_2 \cdots dt_v \\
& = - \frac{1}{16} \sum_{j=2}^v \bigg\{ 2 \bigg( \frac{(t_1^2+t_j^2)(t_1^2-1)^3}{2t_1^2(t_1-t_j)^2(t_1+t_j)^2} \bigg) w_{g,v-1}^{\widetilde{M}(b,c)}(t_1,t_2,\dots,\widehat{t_j},\dots,t_v) \\
& \quad + \frac{\partial}{\partial t_j} \bigg( \frac{(t_j^2-1)^3}{t_j(t_j^2-t_1^2)} w_{g,v-1}^{\widetilde{M}(b,c)}(t_2,\dots,t_v) \bigg) \bigg\} dt_1 \cdot dt_2 \cdots dt_v \\
& = - \frac{1}{16} \sum_{j=2}^v \bigg[ \frac{(t_1^2+t_j^2)(t_1^2-1)^3}{t_1^2(t_1^2-t_j^2)^2} W_{g,v-1}^{\widetilde{M}(b,c)}(t_1,t_2,\dots,\widehat{t_j},\dots,t_v) \; dt_j \\
& \quad + \frac{\partial}{\partial t_j} \bigg( \frac{(t_j^2-1)^3}{t_j(t_j^2-t_1^2)} W_{g,v-1}^{\widetilde{M}(b,c)}(t_2,\dots,t_v) \bigg) \; dt_1 \bigg] \\
& = \frac{1}{16} \sum_{j=2}^v \bigg[ \frac{\partial}{\partial t_j} \bigg( \frac{(t_j^2-1)^3}{t_j(t_1^2-t_j^2)} W_{g,v-1}^{\widetilde{M}(b,c)}(t_2,\dots,t_v) \bigg) \; dt_1 \\
& \quad - \frac{(t_1^2+t_j^2)(t_1^2-1)^3}{t_1^2(t_1^2-t_j^2)^2} W_{g,v-1}^{\widetilde{M}(b,c)}(t_1,t_2,\dots,\widehat{t_j},\dots,t_v) \; dt_j \bigg] \\
& = \text{I}_{g,v}(t_1,t_2,\dots,t_v), \\
\end{align*}

where we used that
\begin{align*}
& \quad \frac{\partial}{\partial t} \bigg( \frac{(t^2+t_j^2)(t^2-1)^3/t}{(t+t_1)(t-t_1)(t+t_j)^2} \times w_{g,v-1}^{\widetilde{M}(b,c)}(t,t_2,\dots,\widehat{t_j},\dots,t_v) \bigg)\bigg|_{t=t_j} \\
& = \bigg[ \frac{(t^2-1)^3}{t(t+t_1)(t-t_1)} w_{g,v-1}^{\widetilde{M}(b,c)}(t,t_2,\dots,\widehat{t_j},\dots,t_v) \bigg( \frac{\partial}{\partial t} \frac{(t^2+t_j^2)}{(t+t_j)^2} \bigg) \\
& \quad + \frac{(t^2+t_j^2)}{(t+t_j)^2} \frac{\partial}{\partial t} \bigg( \frac{(t^2-1)^3}{t(t+t_1)(t-t_1)} w_{g,v-1}^{\widetilde{M}(b,c)}(t,t_2,\dots,\widehat{t_j},\dots,t_v) \bigg)\bigg]_{t=t_j} \\
& = \frac{(t_j^2-1)^3}{t_j(t_j+t_1)(t_j-t_1)} w_{g,v-1}^{\widetilde{M}(b,c)}(t_2,\dots,t_v) \bigg( \frac{(t_j+t_j)^2(2t_j) - (t_j^2+t_j^2)2(t_j+t_j)}{(t_j+t_j)^4} \bigg) \\
& \quad + \frac{(t_j^2+t_j^2)}{(t_j+t_j)^2} \frac{\partial}{\partial t_j} \bigg( \frac{(t_j^2-1)^3}{t_j(t_j+t_1)(t_j-t_1)} w_{g,v-1}^{\widetilde{M}(b,c)}(t_2,\dots,t_v) \bigg) \\
& = \frac{(t_j^2-1)^3}{t_j(t_j+t_1)(t_j-t_1)} w_{g,v-1}^{\widetilde{M}(b,c)}(t_2,\dots,t_v) \bigg( \frac{8t_j^3 - 8t_j^3}{(2t_j)^4} \bigg)  \\
& \quad + \frac{2t_j^2}{(2t_j)^2} \frac{\partial}{\partial t_j} \bigg( \frac{(t_j^2-1)^3}{t_j(t_j+t_1)(t_j-t_1)} w_{g,v-1}^{\widetilde{M}(b,c)}(t_2,\dots,t_v) \bigg) \\
& = \frac{1}{2} \frac{\partial}{\partial t_j} \bigg( \frac{(t_j^2-1)^3}{t_j(t_j^2-t_1^2)} w_{g,v-1}^{\widetilde{M}(b,c)}(t_2,\dots,t_v) \bigg) \\
\end{align*}

and
\begin{align*}
& \quad \frac{\partial}{\partial t} \bigg( \frac{(t^2+t_j^2)(t^2-1)^3/t}{(t+t_1)(t-t_1)(t-t_j)^2} \times w_{g,v-1}^{\widetilde{M}(b,c)}(t,t_2,\dots,\widehat{t_j},\dots,t_v) \bigg)\bigg|_{t=-t_j} \\
& = \bigg[ \frac{(t^2-1)^3}{t(t+t_1)(t-t_1)} w_{g,v-1}^{\widetilde{M}(b,c)}(t,t_2,\dots,\widehat{t_j},\dots,t_v) \bigg( \frac{\partial}{\partial t} \frac{(t^2+t_j^2)}{(t-t_j)^2} \bigg) \\
& \quad + \frac{(t^2+t_j^2)}{(t-t_j)^2} \frac{\partial}{\partial t} \bigg( \frac{(t^2-1)^3}{t(t+t_1)(t-t_1)} w_{g,v-1}^{\widetilde{M}(b,c)}(t,t_2,\dots,\widehat{t_j},\dots,t_v) \bigg)\bigg]_{t=-t_j} \\
& = \frac{(t_j^2-1)^3}{-t_j(-t_j+t_1)(-t_j-t_1)} w_{g,v-1}^{\widetilde{M}(b,c)}(t_2,\dots,t_v) \bigg( \frac{(-t_j-t_j)^2(-2t_j) - (t_j^2+t_j^2)2(-t_j-t_j)}{(-t_j-t_j)^4} \bigg) \\
& \quad + \frac{(t_j^2+t_j^2)}{(-t_j-t_j)^2} \bigg( - \frac{\partial}{\partial t_j} \bigg) \bigg( \frac{(t_j^2-1)^3}{-t_j(-t_j+t_1)(-t_j-t_1)} w_{g,v-1}^{\widetilde{M}(b,c)}(t_2,\dots,t_v) \bigg) \\
& = \frac{(t_j^2-1)^3}{-t_j(-t_j+t_1)(-t_j-t_1)} w_{g,v-1}^{\widetilde{M}(b,c)}(t_2,\dots,t_v) \bigg( \frac{-8t_j^3 + 8t_j^3}{(2t_j)^4} \bigg) \\
& \quad + \frac{2t_j^2}{(2t_j)^2} \bigg( - \frac{\partial}{\partial t_j} \bigg) \bigg( \frac{(t_j^2-1)^3}{t_j(-t_j+t_1)(t_j+t_1)} w_{g,v-1}^{\widetilde{M}(b,c)}(t_2,\dots,t_v) \bigg) \\
& = \frac{1}{2} \frac{\partial}{\partial t_j} \bigg( \frac{(t_j^2-1)^3}{t_j(t_j^2-t_1^2)} w_{g,v-1}^{\widetilde{M}(b,c)}(t_2,\dots,t_v) \bigg). \\
\end{align*}

The second term in the differential recursion becomes zero when we differentiate with respect to $t_j$, so we have
$$\text{II}_{g,v}(t_1,t_2,\dots,t_v) = 0.$$

Now, we may compute
\begin{align*}
& \quad - \frac{1}{64} \frac{1}{2\pi i} \int_{\gamma} \bigg( \frac{1}{t+t_1} + \frac{1}{t-t_1} \bigg) \frac{(t^2-1)^3}{t^2} \; dt \\
& = \frac{1}{64} \bigg[ \text{Res}_{t=t_1} \frac{2(t^2-1)^3/t}{(t+t_1)(t-t_1)} + \text{Res}_{t=-t_1} \frac{2(t^2-1)^3/t}{(t+t_1)(t-t_1)} \bigg] \\
& = \frac{1}{64} \bigg[ \frac{2(t_1^2-1)^3/t_1}{2t_1} + \frac{2(t_1^2-1)^3/(-t_1)}{-2t_1} \bigg] \\
& = \frac{1}{32} \frac{(t_1^2-1)^3}{t_1^2}. \\
\end{align*}

Thus, for the third term in the differential recursion, we have
\begin{align*}
& \quad \text{III}_{g,v}(t_1,t_2,\dots,t_v) \\
& = \frac{\partial}{\partial t_2} \cdots \frac{\partial}{\partial t_v} \bigg[ - \frac{1}{32} \frac{(t_1^2-1)^3}{t_1^2} \frac{\partial}{\partial u_1} \frac{\partial}{\partial u_2} F_{g-1,v+1}^{\widetilde{M}(b,c)}(u_1,u_2,t_2,\dots,t_v) \bigg|_{u_1=u_2=t_1}  \bigg] dt_1 dt_2 \cdots dt_v \\
& = - \frac{1}{32} \frac{(t_1^2-1)^3}{t_1^2} W_{g-1,v+1}^{\widetilde{M}(b,c)}(u_1,u_2,t_2,\dots,t_v) \frac{1}{du_1} \frac{1}{du_2} \bigg|_{u_1=u_2=t_1} dt_1 \\
& = - \frac{1}{32} \frac{(t_1^2-1)^3}{t_1^2} W_{g-1,v+1}^{\widetilde{M}(b,c)}(t_1,t_1,t_2,\dots,t_v) \frac{1}{dt_1}. \\
\end{align*}

And, from the topological recursion formula, we have
\begin{align*}
& \quad - \frac{1}{64} \frac{1}{2\pi i} \int_{\gamma} \bigg( \frac{1}{t+t_1} + \frac{1}{t-t_1} \bigg) \frac{(t^2-1)^3}{t^2} \cdot \frac{1}{dt} \cdot dt_1 \times W_{g-1,v+1}^{\widetilde{M}(b,c)}(t,-t,t_2,\dots,t_v) \\
& = - \frac{1}{64} \frac{1}{2\pi i} \int_{\gamma} \bigg( \frac{1}{t+t_1} + \frac{1}{t-t_1} \bigg) \frac{(t^2-1)^3}{t^2} \cdot \frac{1}{dt} \cdot dt_1 \\ 
& \quad \times w_{g-1,v+1}^{\widetilde{M}(b,c)}(t,-t,t_2,\dots,t_v) \; dt \cdot (-dt) \cdot dt_2 \cdots dt_v \\
& = \frac{1}{64} \bigg[ \frac{1}{2\pi i} \int_{\gamma} \bigg( \frac{1}{t+t_1} + \frac{1}{t-t_1} \bigg) \frac{(t^2-1)^3}{t^2} \times w_{g-1,v+1}^{\widetilde{M}(b,c)}(t,t,t_2,\dots,t_v) \; dt \bigg] \; dt_1 \cdot dt_2 \cdots dt_v \\
& = -\frac{1}{32} \frac{(t_1^2-1)^3}{t_1^2} w_{g-1,v+1}^{\widetilde{M}(b,c)}(t_1,t_1,t_2,\dots,t_v) \; dt_1 \cdot dt_2 \cdots dt_v \\
& = -\frac{1}{32} \frac{(t_1^2-1)^3}{t_1^2} W_{g-1,v+1}^{\widetilde{M}(b,c)}(t_1,t_1,t_2,\dots,t_v) \; \frac{1}{dt_1} \\
& =  \text{III}_{g,v}^{\widetilde{M}(b,c)}(t_1,t_2,\dots,t_v). \\
\end{align*}

Finally, for the fourth term, we have
\begin{align*}
& \quad \text{IV}_{g,v}(t_1,t_2,\dots,t_v) \\
& = \frac{\partial}{\partial t_2} \cdots \frac{\partial}{\partial t_v} \bigg[ - \frac{1}{32} \frac{(t_1^2-1)^3}{t_1^2} \\
& \quad \cdot \sum_{g_1+g_2=g, \; I \sqcup J = \{2,\dots,v\}, \; \text{stable}} \frac{\partial}{\partial t_1} F_{g_1,|I|+1}^{\widetilde{M}(b,c)}(t_1,t_I) \cdot \frac{\partial}{\partial t_1} F_{g_2,|J|+1}^{\widetilde{M}(b,c)}(t_1,t_J) \bigg] dt_1 dt_2 \cdots dt_v \\
& = - \frac{1}{32} \frac{(t_1^2-1)^3}{t_1^2} \sum_{g_1+g_2=g, \; I \sqcup J = \{2,\dots,v\}, \; \text{stable}} \frac{\partial}{\partial t_1} \frac{\partial}{\partial t_I} F_{g_1,|I|+1}^{\widetilde{M}(b,c)}(t_1,t_I) \; dt_1 dt_I \\
& \quad \cdot \frac{\partial}{\partial t_1} \frac{\partial}{\partial t_J} F_{g_2,|J|+1}^{\widetilde{M}(b,c)}(t_1,t_J) dt_1 dt_J \frac{1}{dt_1} \\
& = - \frac{1}{32} \frac{(t_1^2-1)^3}{t_1^2} \sum_{g_1+g_2=g, \; I \sqcup J = \{2,\dots,v\}, \; \text{stable}} W_{g_1,|I|+1}^{\widetilde{M}(b,c)}(t_1,t_I) \cdot W_{g_2,|J|+1}^{\widetilde{M}(b,c)}(t_1,t_J) \cdot \frac{1}{dt_1}. \\
\end{align*}

And, from the topological recursion formula, we have
\begin{align*}
& \quad - \frac{1}{64} \frac{1}{2\pi i} \int_{\gamma} \bigg( \frac{1}{t+t_1} + \frac{1}{t-t_1} \bigg) \frac{(t^2-1)^3}{t^2} \cdot \frac{1}{dt} \cdot dt_1 \\
& \quad \times \bigg[ \sum_{g_1+g_2=g, \; I \sqcup J = \{2,\dots,v\}, \; \text{stable}} \frac{\partial}{\partial t_1} W_{g_1,|I|+1}^{\widetilde{M}(b,c)}(t,t_I) W_{g_2,|J|+1}^{\widetilde{M}(b,c)}(-t,t_J) \bigg] \\
& = - \frac{1}{64} \frac{1}{2\pi i} \int_{\gamma} \bigg( \frac{1}{t+t_1} + \frac{1}{t-t_1} \bigg) \frac{(t^2-1)^3}{t^2} \cdot \frac{1}{dt} \cdot dt_1 \\
& \quad \times\sum_{g_1+g_2=g, \; I \sqcup J = \{2,\dots,v\}, \; \text{stable}}  \bigg[ w_{g_1,|I|+1}^D(t,t_I) \; dt \cdot \prod_{i \in I} dt_i \bigg] \bigg[ w_{g_2,|J|+1}^D(-t,t_J) \; (-dt)  \cdot \prod_{j \in J} dt_j \bigg] \\
& = \frac{1}{64} \frac{1}{2\pi i} \int_{\gamma} \bigg( \frac{1}{t+t_1} + \frac{1}{t-t_1} \bigg) \frac{(t^2-1)^3}{t^2} \\
& \quad \times\sum_{g_1+g_2=g, \; I \sqcup J = \{2,\dots,v\}, \; \text{stable}} w_{g_1,|I|+1}^D(t,t_I) w_{g_2,|J|+1}^D(t,t_J) \; dt \cdot dt_1 \cdot dt_2 \cdots dt_v \\
& = - \frac{1}{32} \frac{(t_1^2-1)^3}{t_1^2} \sum_{g_1+g_2=g, \; I \sqcup J = \{2,\dots,v\}, \; \text{stable}} w_{g_1,|I|+1}^D(t_1,t_I) w_{g_2,|J|+1}^D(t_1,t_J) \; dt_1 \cdot dt_2 \cdots dt_v \\
& = - \frac{1}{32} \frac{(t_1^2-1)^3}{t_1^2} \sum_{g_1+g_2=g, \; I \sqcup J = \{2,\dots,v\}, \; \text{stable}} W_{g_1,|I|+1}^D(t_1,t_I) W_{g_2,|J|+1}^D(t_1,t_J) \; \frac{1}{dt_1} \\
& = \text{IV}_{g,v}(t_1,t_2,\dots,t_v). \\
\end{align*}

This completes the proof of the theorem.
\end{proof}


\begin{ack}

The author would like to thank her thesis advisor, Motohico Mulase, for many helpful conversations about this project. She would also like to thank the UC Davis mathematics department for their financial support during her time studying there as a graduate student, as well as the generous donors of the Schwarze Scholarship in mathematics, whose funding helped to make the completion of this project possible. The author is also very grateful to the many folks who have helped her to improve her problem-solving and critical thinking skills. Without their help and advice, this project would likely have not been completed.

\end{ack}



\providecommand{\bysame}{\leavevmode\hbox to3em{\hrulefill}\thinspace}

\bibliographystyle{amsplain}


\end{document}